\documentclass{amsproc}

\usepackage{amsmath}
\usepackage{xcolor}
\usepackage{amssymb}
\usepackage{url}
\usepackage{supertabular} 
\usepackage{mathabx}
\usepackage{multirow}
\usepackage{graphics}
\usepackage{float}
\usepackage{graphicx}
\usepackage[utf8]{inputenc}
\usepackage{makecell}
\usepackage{bigstrut}
\usepackage{longtable}
\usepackage{afterpage} 
\allowdisplaybreaks

\newcommand{\ud}{\mathrm{d}} 
\newcommand{\uD}{\mathrm{D}} 

\renewcommand{\phi}{\varphi}

\newcommand{\splitatcommas}[1]{%
  \begingroup
  \begingroup\lccode`~=`, \lowercase{\endgroup
    \edef~{\mathchar\the\mathcode`, \penalty0 \noexpand\hspace{0pt plus 1em}}%
  }\mathcode`,="8000 #1%
  \endgroup
}

\newtheorem{theorem}{Theorem}[section]
\newtheorem{lemma}[theorem]{Lemma}

\theoremstyle{definition}

\newtheorem{example}[theorem]{Example}

\theoremstyle{remark}

\theoremstyle{remark}

\theoremstyle{conjecture}
\newtheorem{conjecture}[theorem]{Conjecture}

\numberwithin{equation}{section}

\newmuskip\pFqskip
\mathchardef\pFcomma=\mathcode`, 

\newcommand*\pFq[5]{%
  \begingroup
  \begingroup\lccode`~=`,
    \lowercase{\endgroup\def~}{\pFcomma\mkern\pFqskip}%
  \mathcode`,=\string"8000
  {}_{#1}F_{#2}\biggl(\genfrac..{0pt}{}{#3}{#4};#5\biggr)%
  \endgroup
}

\title
[Quadratic irrational series for $1/\pi$]{Quadratic irrational analogues of Ramanujan's series for $1/\pi$}
\date{\today}
\author{John M. Campbell}
\author{Shaun Cooper}
\author{Dongxi Ye}
\address{
Department of Mathematics and Statistics, Dalhousie University, 
 6299 South St, Halifax, Canada 
 E-mail: jh241966@dal.ca
}
\address{
School of Mathematical and Computational Sciences, Massey University,
Private Bag 102904, North Shore Mail Centre, Auckland, New Zealand
E-mail: s.cooper@massey.ac.nz
}

\address{
School of AI and Liberal Arts, Beijing Normal-Hong Kong Baptist University, Zhuhai 519082, Guangdong,
People's Republic of China
E-mail: dongxiye@bnbu.edu.cn
}



\thanks{Dongxi Ye was supported by the Guangdong Basic and Applied Basic Research Foundation (Grant No. 2024A1515030222).}

\usepackage{hyperref}

\hypersetup{
 colorlinks = true,
 urlcolor = blue,
 linkcolor = black,
 citecolor = blue
}

\begin{document}

\begin{abstract}
About 40 years ago Jonathan and Peter Borwein discovered the series identity  
$$
\sum_{n=0}^\infty \frac{(-1)^n(6n)!}{(3n)!(n!)^3}  \frac{(A+nB)}{C^{n+1/2}} = \frac{1}{12\pi}
$$
where 
\begin{align*}
    A&=1657145277365+212175710912\sqrt{61}, \\
    B&=107578229802750+13773980892672\sqrt{61}, \\
    C&=\left(5280(236674+30303\sqrt{61})\right)^3
\end{align*}
which adds roughly 25 digits of accuracy per term. They noted that if each of the quadratic
irrationals
$A$, $B$ and $C$ is replaced by their conjugates, that is, each number $a+b\sqrt{61}$
is changed to $a-b\sqrt{61}$, then the
resulting series also converges to a rational multiple of $1/\pi$.
They gave several other examples of quadratic irrational series for $1/\pi$, and noted that
the conjugate series converges to another rational multiple of $1/\pi$ or in some
cases the conjugate series diverges.

The purpose of this work
is to provide an explanation and classification of such series.
Our classification includes Ramanujan's 17 original series, as well as series of the Borweins,
Chudnovskys, Sato and others.
We extend the classification to genus-zero subgroups $\Gamma_0(\ell)+$, that is, for each
$\splitatcommas{\ell \in \big\{1,2,3,\ldots,36,38,39,41,42,44,45,46,47,49,50,51,54,55,56,59,60,62,66,69,70,
71,78,87,92,94,95,105,110,119\big\}}$ we calculate the Hauptmoduls, associated
weight two modular forms, and the corresponding rational and real quadratic irrational series for $1/\pi$.

The classification reveals many interrelations among the different series. For example,
we show that the Borweins' series above, and its conjugate, are equivalent
by hypergeometric transformation formulas
to the level~7 rational series
$$
\sum_{n=0}^\infty \left\{\sum_{j=0}^n {n \choose j}^2{2j \choose n} {n+j \choose j}\right\} (11895n+1286) \frac{(-1)^n}{22^{3n+3}}
= \frac{1}{\pi\sqrt{7}}.
$$

\end{abstract}

\maketitle

\section{Introduction}
S.~Ramanujan~\cite{ramanujan_pi} introduced four types of series for $1/\pi$ and gave seventeen
examples. The most famous is
\begin{equation}
\label{E:ramanujan-gosper}
\sum_{n=0}^\infty {4n \choose 2n}{2n \choose n}^2\frac{(1103+26390n)}{99^{4n+2}}
=\frac{1}{2\sqrt{2}}\times \frac{1}{\pi}
\end{equation}
which adds approximately 8 decimal places of accuracy per term. The series converges
sufficiently fast that it was used in a world record setting calculation of $\pi$
by R.~W.~Gosper
in~1985, e.g., 
see~\cite[p. 341]{borwein-book}, \cite{ramanujan_and_pi}, \cite{bbe3}.
Since all the terms in the series
are rational numbers, we say~\eqref{E:ramanujan-gosper} is a rational series 
for~$1/\pi$. Technically, it is a ``rational series for $\frac{1}{2\sqrt{2}}\times \frac{1}{\pi}$'', but for
naming purposes the term ``rational'' will refer to the terms in the series.
As we will see in the classification, most ``rational series for $1/\pi$'' converge to a
quadratic irrational times $1/\pi$. Moreover, the classification shows they
are actually series for $\frac{1}{2\pi}$, e.g., see~\eqref{E:CCL}, but we do not propose renaming the series to reflect this.

Of Ramanujan's 17 examples, 16 are rational series. The other series~\cite[(30)]{ramanujan_pi} involves
quadratic irrationals and can be written as
\begin{equation}
\label{E:ramirrational}
\sum_{n=0}^\infty {2n \choose n}^3 \left(n+\frac{31}{270+48\sqrt{5}}\right) 
\frac{1}{2^{4n}(1+\sqrt{5})^{8n}}
= \frac{16}{15+21\sqrt{5}} \times  \frac{1}{\pi}.
\end{equation}
Further examples of quadratic irrational series were given by the
Borweins~\cite{borweinrr}. They also computed series involving cubic irrational numbers corresponding
to class number three fields in~\cite{borweinclass3}.

The Borweins' example in the abstract may be rewritten
\begin{equation}
\label{E:deg427}
\sum_{n=0}^\infty {6n \choose 3n}{3n \choose n}{2n \choose n}
\left(n+\lambda\right)X^n 
= \frac{1}{\pi}\times \frac{1}{\sqrt{427}} \times \frac{1}{\sqrt{1-1728X}}
\end{equation}
where
$$
X=\frac{-1}{2^{15}\cdot3^3\cdot5^3\cdot11^3\cdot(236674 + 30303\sqrt{61})^3}
$$
and
$$
\lambda=\frac{291728347791 - 28276402880\sqrt{61}}{2\cdot3^2\cdot7\cdot19\cdot23\cdot47\cdot61\cdot103\cdot283}.
$$
The advantages of this form over the formula in the abstract are
as follows. First, only
two parameters $X$ and $\lambda$ need to be specified, whereas the Borweins' 
formula requires values for three parameters $A$, $B$ and $C$.
Second, the conjugate series obtained by replacing $X$ and $\lambda$
by their conjugates
$$
\overline{X}=\frac{-1}{2^{15}\cdot3^3\cdot5^3\cdot11^3\cdot(236674 - 30303\sqrt{61})^3}
$$
and
$$
\overline{\lambda}=\frac{291728347791 + 28276402880\sqrt{61}}{2\cdot3^2\cdot7\cdot19\cdot23\cdot47\cdot61\cdot103\cdot283}
$$
is given by
\begin{equation}
\label{E:deg427c}
\sum_{n=0}^\infty {6n \choose 3n}{3n \choose n}{2n \choose n}
\left(n+\overline{\lambda}\right)\overline{X}^n 
= \frac{1}{\pi}\times \frac{7}{\sqrt{427}} \times \frac{1}{\sqrt{1-1728\overline{X}}}.
\end{equation}
The conjugate series given by the Borweins has $A$, $B$ and $C$ replaced with
$-\overline{A}$, $-\overline{B}$ and $\overline{C}$, respectively, and the negative
signs that appear with conjugate values of $A$ and $B$ suggest the notation
is not optimal.\footnote{The conjugate series~\cite[p. 360]{borweinrr} given by the Borweins 
contains a typographical error that is repeated in~\cite[p. 282]{borweinclass3}:
the factor $7\times 12$ that multiplies the conjugate series should be $12/7$.}
Third, the form of the series~\eqref{E:deg427} is an instance of a theorem 
of H. H. Chan, S. H. Chan and Z.-G. Liu~\cite{domb} that encapsulates all of Ramanujan's examples and
allows classification of further series that were initiated by T.~Sato\cite{sato}.

The goal of this work is to provide a complete classification that includes series such
as~\eqref{E:ramanujan-gosper}--\eqref{E:deg427c}. Specifically, for each $\Gamma_0(\ell)+$
for which the modular curve has genus 0 we exhibit a Hauptmodul~$X$ and the corresponding
weight two modular form~$Z$ and determine the rational and quadratic series for $1/\pi$.

The work is organized as follows.
The underlying theory, which consists of four theorems, is outlined in Section~\ref{S:Theory}.
The first three theorems (Theorems~\ref{T:1}–\ref{T:4}) involve a Hauptmodul $X$, a
corresponding weight two modular form $Z$, and two polynomials $G$ and $H$.
The special cases of Theorems~\ref{T:1}–\ref{T:4} for levels $1$ and $30$ are discussed in Section~\ref{S:Hauptmoduls}, and the data for each genus-zero group $\Gamma_0(\ell)+$ are summarized in Tables~\ref{T:00}, \ref{T:Z}, and \ref{T:GH}.
In Section~\ref{S:pi}, several well-known series for $1/\pi$ are given to illustrate the
fourth theorem (Theorem~\ref{T:3}), and the data in Tables~\ref{T:1.1}–\ref{T:105.1} are discussed.
A brief historical summary is provided in Section~\ref{S:historical}.
Convergence is analyzed in Section~\ref{S:Con}.
Some general theorems that relate a quadratic irrational series to its conjugate are
given in Sections~\ref{S:level1} and \ref{S:level2}. This includes precise details of how the
Borweins' conjugate series are equivalent to the rational level~7 series in the abstract.
Further general results are given in Sections~\ref{S:general} and \ref{S:remarks}.
Section~\ref{S:tables} contains tables:
\begin{itemize}
\item
Table~\ref{T:00} lists the Hauptmoduls $X$;
\item
Table~\ref{T:Z} lists the corresponding weight two modular forms $Z$;
\item
Table~\ref{T:GH} lists the associated polynomials $G$ and $H$;
\item
Tables~\ref{T:1.1}–\ref{T:105.1} contain data for series for $1/\pi$;
\item
Table~\ref{T:00X} lists cases in which the coefficients $T(n)$ have known formulas
in terms of binomial coefficients.
\end{itemize}

We end this section with some notation that will be used throughout.
Let $\tau$ be a complex variable with $\text{Im}(\tau)>0$. Let $q=\exp(2\pi i\tau)$ so
that $|q|<1$. Dedekind's eta function is
$$
\eta(\tau) = q^{1/24}\prod_{j=1}^\infty (1-q^j),
$$
and Ramanujan's Eisenstein series
$P$, $Q$ and $R$ are defined by
$$
P(q) = 1-24\sum_{j=1}^\infty \frac{jq^j}{1-q^j}, \quad
Q(q) = 1+240\sum_{j=1}^\infty \frac{j^3q^j}{1-q^j}, \quad
R(q) = 1-504\sum_{j=1}^\infty \frac{j^5q^j}{1-q^j}.
$$
They satisfy the properties, e.g., see~\cite[p. 92]{spirit}, 
\cite{chanrde},
\cite[pp. 29--31]{cooperbook},
\cite{ramanujan_arithmetical}:
\begin{equation}
\label{E:rde1}
q\frac{\ud}{\ud q} \log \eta^{24}(q) = P(q), \qquad Q^3(q)-R^2(q) = 1728 \eta^{24}(\tau)
\end{equation}
and
\begin{equation}
\label{E:rde2}
q\frac{\ud P}{\ud q}=\frac{P^2-Q}{12}, \quad
q\frac{\ud Q}{\ud q}=\frac{PQ-R}{3}, \quad
q\frac{\ud R}{\ud q}=\frac{PR-Q^2}{2}.
\end{equation}
For any positive integer $m$, let $\eta_m$ and $P_m$ be defined by
$$
\eta_m = \eta(m\tau) = q^{m/24} \prod_{j=1}^\infty (1-q^{mj})
$$
and
$$
P_m = P(q^m) = 1-24\sum_{j=1}^\infty \frac{jq^{mj}}{1-q^{mj}}.
$$

\section{Theory}
\label{S:Theory}
In this section, we outline the basic theory for
Ramanujan-type series for~$1/\pi$ such as~\eqref{E:ramanujan-gosper}--\eqref{E:deg427c}.
It will always be assumed that $X=X(q)=q+O(q^{2})$ is a Hauptmodul of a Fuchsian group~$\Gamma$ that is a normal extension of the Fricke extension of some $\Gamma_{0}(\ell)$, i.e., $\Gamma \trianglerighteq \langle \Gamma_{0}(\ell),\begin{pmatrix}0&-1\\\ell&0\end{pmatrix}\rangle$, so that it satisfies
\begin{equation}
\label{E:X}
X(e^{-2\pi\sqrt{t/\ell}})=X(e^{-2\pi/\sqrt{t\ell}}).
\end{equation}
Examples of Hauptmoduls are given in Table~\ref{T:00}.

The first result shows that a Hauptmodul $X$ is naturally associated with
a weight two modular form~$Z$.
\begin{theorem}
\label{T:1}
Suppose $X=X(q)$ is as above.
There exist polynomials 
\begin{equation}
\label{E:GHdefn}
G(x) = 1+\sum_{j=1}^{k_1} g_jx^j \quad\text{and}\quad H(x) = \sum_{j=1}^{k_2} h_j x^j
\end{equation}
such that the weight two modular form $Z$ defined by
\begin{equation}
\label{E:Zdefn}
Z=\frac{1}{\sqrt{G(X)}} \, q\frac{\ud}{\ud q} \log X
\end{equation}
satisfies
\begin{equation}
\label{E:Zde}
\frac{1}{Z_q}\left(\frac{Z_{q}^2}{Z^3}\right)_{\!q} = H(X),
\end{equation}
where $F_q:=q\frac{\ud F}{\ud q}$.
That is, $\frac{1}{Z_q}\left(\frac{Z_{q}^2}{Z^3}\right)_{\!q}$ is a polynomial in~$X$, and the constant
term of the polynomial is zero. 
\end{theorem}

\begin{proof}
By the quotient rule from calculus we have
$$
\frac{1}{Z_q}\left(\frac{Z_{q}^2}{Z^3}\right)_{\!q}
= \frac{2ZZ_{qq}-3(Z_q)^2}{Z^4}
$$
and so \eqref{E:Zde} can be written as
\begin{equation}
    \label{E:Zrewrite}
    2ZZ_{qq}-3(Z_q)^2 = Z^4H.
\end{equation}
Therefore, the theorem is just a restatement of
the result of T.~Huber, D.~Schultz and D.~Ye~\cite[Theorem~2.2]{huberActa}. 
\end{proof}

The next result shows that when the weight two modular form~$Z$ is expanded
in powers of the Hauptmodul~$X$, the coefficients satisfy a linear recurrence
relation which can be written explicitly in terms of the coefficients
of the polynomials $G$ and $H$.
\begin{theorem}
\label{T:2}
Suppose $X$, $Z$, $G$ and $H$ are as in Theorem~\ref{T:1}, and let $k=\text{\rm{max}}\left\{k_1,k_2\right\}$ be the maximum of the degrees of $G(x)$ and $H(x)$.
The coefficients $T(n)$ in the expansion
\begin{equation}
\label{E:Zseries}
Z=\sum_{n=0}^\infty T(n)X^n
\end{equation}
satisfy the $(k+1)$-term linear recurrence relation
\begin{equation}
\label{E:rec}
\sum_{j=0}^{k_1} (n+1)(n+1-j)(n+1-\tfrac{j}{2})g_j\,T(n+1-j) 
= \sum_{j=1}^{k_2} (n+1-\tfrac{j}{2})h_j\,T(n+1-j)
\end{equation}
where $g_j$ and $h_j$ are as in~\eqref{E:GHdefn} and $g_0$ is defined to be~$1$.
Equivalently,
\begin{equation}
\label{E:rec1}
(n+1)^3T(n+1)=\sum_{j=1}^k (n+1-\tfrac{j}{2})
\left\{h_j-(n+1)(n+1-j)g_j\right\}T(n+1-j)
\end{equation}
where $g_j$ (or $h_j)$ is defined to be zero if $j>k_1$ (resp., if $j>k_2$). 
\end{theorem}

Before starting the proof, it is worth pointing out that the series~\eqref{E:Zseries}
converges at least for $|X|<|r|$, where $r$ is a root of $G(x)$ of minimum
modulus. This follows either from the Frobenius theory of power series
solutions of linear differential equations e.g.,~\cite{frobenius},~\cite[p. 67]{poole}
or from the asymptotic formula for $T(n)$
in~\cite[Th. 10.1]{cooperAperyLike}.
A more precise analysis of convergence will be provided in Section~\ref{S:Con}.

\begin{proof}
By the definition of~$Z$ in~\eqref{E:Zdefn} and the chain rule, we have
$$
q\frac{\ud}{\ud q} = q\frac{\ud X}{\ud q} \frac{\ud}{\ud X}
= ZX\sqrt{G(X)} \frac{\ud}{\ud X} = Z\uD
$$
where $\uD$ is the differential operator defined by
$$
\uD = X\sqrt{G(X)} \frac{\ud}{\ud X}.
$$
With this change of variables, the differential equation~\eqref{E:Zrewrite} becomes
$$
2Z^2D(ZD(Z)) - 3Z^2(DZ)^2 = Z^4H.
$$
On dividing by $Z^2$ and applying the product rule to the first term, this simplifies to
\begin{equation}
\label{E:nonlinear}
2Z(D^2Z) - (DZ)^2 = Z^2H
\end{equation}
and the recurrence relation for the coefficients follows 
from\footnote{The differential equation~\eqref{E:nonlinear} was studied in~\cite[Theorem~3.1]{cooperAperyLike} in the form
\begin{equation}
\label{E:nonlinear1}
2Z(D^2Z) - (DZ)^2 = 2Z^2H,
\end{equation}
that is, with $2H$ in place of $H$. See also~\cite[p. 292 and Tables 14.1 and 14.2]{cooperbook}.
We have decided to use~\eqref{E:nonlinear}
instead of~\eqref{E:nonlinear1}
because then the recurrence relation~\eqref{E:rec} and the check~\eqref{E:check1}
become simpler with this change of notation.
The polynomials $G$ and $H$ in Theorem~\ref{T:1} correspond to the polynomials
$w$ and $R$, respectively, in~\cite{huberActa}.}~\cite[Theorem~3.1]{cooperAperyLike}.
\end{proof}

The next result is a partial converse to Theorem~\ref{T:1}.
It shows that the Hauptmodul~$X$ is determined by the
polynomials~$G(x)$ and~$H(x)$.
\begin{theorem}
\label{T:4}
Let $X$, $Z$, $G$ and $H$ be as for Theorem~\ref{T:1}.
The coefficients in the $q$-expansions
$$
X = q+a_2q^2+a_3q^3+\cdots\quad\text{and}\quad Z=1+b_1q+b_2q^2+\cdots
$$
can be calculated from the polynomials $G(x)$ and $H(x)$.
\end{theorem}
\begin{proof}
We use induction on $n$. Suppose that the coefficients $a_2,\,a_3,\ldots,a_n$ are
known and we wish to compute $a_{n+1}$ for some $n\geq 1$. Since
\begin{align*}
\frac{1}{X}\,q\frac{\ud X}{\ud q}
&= \frac{q+2a_2q^2+3a_3q^3+\cdots+(n+1)a_{n+1}q^{n+1}+\cdots}
{q+a_2q^2+a_3q^3+\cdots+a_{n+1}q^{n+1}+\cdots} \\
&= \frac{1+2a_2q+3a_3q^2+\cdots+(n+1)a_{n+1}q^{n}+\cdots}
{1+a_2q+a_3q^2+\cdots+a_{n+1}q^{n}+\cdots} \\
\end{align*}
the coefficient of $q^n$ is given by
$$
[q^n] \left(\frac{1}{X}\,q\frac{\ud X}{\ud q}\right)
= na_{n+1} + (\text{some function of $a_2,\,a_3,\ldots,a_n$}).
$$
Hence by~\eqref{E:Zdefn} it follows that
$$
[q^n] Z =na_{n+1} + (\text{some function of $a_2,\,a_3,\ldots,a_n$}).
$$
From this, we deduce further that
$$
[q^n] \left(ZZ_{qq}\right) = n^3a_{n+1} + (\text{some function of $a_2,\,a_3,\ldots,a_n$})
$$
while the coefficients of $q^n$ in $Z_{q}^2$ and $Z^4H(X)$
are functions of $a_2,\,a_3,\ldots,a_n$ only, and do not involve $a_{n+1}$.
It follows that equating coefficients of $q^n$ in~\eqref{E:Zrewrite} gives an
expression for $a_{n+1}$ in terms of $a_2,\,a_3,\ldots,a_n$. In this way, the $q$-expansion
of $X$ is deduced from the polynomials $G(x)$ and $H(x)$. The $q$-expansion of $Z$
is then obtained from the $q$-expansion of $X$ by using~\eqref{E:Zdefn}.
\end{proof}

A computer implementation of Theorem~\ref{T:4} is given in the next section --- see
Fig.~\ref{fig:1}.
If we compute the $q$-expansions of~\eqref{E:Zdefn} and~\eqref{E:Zseries}
and compare coefficients of~$q$ then we deduce
\begin{equation}
    \label{E:check2}
    b_1 = \frac{h_1}{2}=T(1),
\end{equation}
while comparing coefficients of~$q$ in~\eqref{E:Zrewrite} gives
\begin{equation}
    \label{E:check1}
    a_2=\frac12(g_1+h_1).
\end{equation}
They may be used as checks when working with these functions.

The final result of this section, due to
Chan et al.~\cite{domb}, provides a classification of series for $1/\pi$. In view of Theorems~\ref{T:1} and~\ref{T:2}, the hypotheses reduce to
the single assumption of the existence of a Hauptmodul that satisfies~\eqref{E:X}.
\begin{theorem}
\label{T:3}
\begin{enumerate}
\item Suppose $X=X(q)=q+O(q^2)$ is a Hauptmodul of a Fuchsian group $\Gamma$ that is a normal
extension of some $\Gamma_0(\ell)$.
\item
Let $Z(q)$, $G(x)$, $H(x)$ and $T(n)$ be as for Theorems~\ref{T:1} and~\ref{T:2}.
\item 
Let $N$ be a positive rational number and suppose that
$X=X(q)$ and \mbox{$Y=Y(q)=X(q^N)$} are related by the modular equation $f(X,Y)=0.$
\item
\label{E:lambda}
Let
$$
\lambda = \lambda(q) = \frac{X}{2} \, \frac{\ud}{\ud X}
\left(\frac{Y}{X} \left(\frac{G(Y)}{G(Y)}\right)^{1/2} \div \frac{\ud Y}{\ud X}\right)
$$
where the derivatives are calculated from the modular equation by implicit differentiation.
\item (Positive series)
Let $X_{\ell,N} = X(e^{-2\pi\sqrt{N/\ell}})$ and 
$\lambda_{\ell,N} = \lambda(e^{-2\pi/\sqrt{N\ell}}).$
If $|X_{\ell,N}|<|r|$, where $r$ is a root of $G(x)$ of minimum modulus, then
\begin{equation}
    \label{E:CCL}
    \sum_{n=0}^\infty T(n)(n+\lambda_{\ell,N})\left(X_{\ell,N}\right)^n
    = \frac{1}{2\pi} \times \frac{1}{\sqrt{G(X_{\ell,N})}} \times \sqrt{\frac{\ell}{N}}.
\end{equation}
\item (Negative series)
Let $X_{L,N} = X(-e^{-2\pi\sqrt{N/L}})$ and 
$\lambda_{L,N} = \lambda(-e^{-2\pi/\sqrt{NL}})$ where 
$L=\frac{4\ell}{\text{\rm{gcd}}(4,\ell)}$
and {\rm{gcd}} is the greatest common divisor.
If $|X_{L,N}|<|r|$, where $r$ is a root of $G(x)$ of minimum modulus, then
\begin{equation}
    \label{E:CCLminus}
    \sum_{n=0}^\infty T(n)(n+\lambda_{L,N})\left(X_{L,N}\right)^n
    = \frac{1}{2\pi} \times \frac{1}{\sqrt{G(X_{L,N})}} \times \sqrt{\frac{L}{N}}.
\end{equation}
\end{enumerate}
\end{theorem}
\begin{proof}
This is Theorem~2.1 in the work by H. H. Chan, S. H. Chan and Z.-G. Liu~\cite{domb}.
See also~\cite[Th. 14.14]{cooperbook}.
\end{proof}

The series~\eqref{E:ramanujan-gosper},
\eqref{E:ramirrational}, \eqref{E:deg427} and~\eqref{E:deg427c} are instances of
Theorem~\ref{T:3} corresponding to $(\ell,N) = (2,29)$, $(4,15)$, $(1,427)$ and $(1,61/7)$,
respectively. Examples corresponding to $(\ell,N) = (4,3),\;(5,2)$ and $(3,2)$
are worked out
in detail in~\cite[Sec. 3]{domb}, \cite{chancooper}
and~\cite[pp. 632--635]{cooperbook}, respectively. Further examples 
are discussed in Section~\ref{S:pi}.

\section{Examples of Hauptmoduls and associated weight 2 modular forms}
\label{S:Hauptmoduls}
Hauptmoduls $X=X(q)$ for those $\Gamma_0(\ell)+$ for
which the modular curve has genus 0 are listed in Table~\ref{T:00}.
See also~\cite[Tables 3 and 4a]{conwaynorton}, \cite[Table 3]{huberActa}
and~\cite[Table 3]{jorgenson}.
The corresponding weight two modular forms \mbox{$Z=Z(q)$} that occur
in Theorem~\ref{T:1} are listed in
Table~\ref{T:Z} and the associated polynomials $G(x)$ and $H(x)$ 
are listed in Table~\ref{T:GH}.
In general, the Hauptmodul~$X$ is chosen to be such that it has a zero at the cusp $[i\infty]$, and it has a pole at either a cusp apart from $[i\infty]$ or an elliptic point of period $h>2$, or an elliptic point of period~$2$ if there is no cusp other than~$[i\infty]$ or elliptic point of period~$h>2$. 

We illustrate by considering
the case $\ell=1$ in detail. From Tables~\ref{T:00} and~\ref{T:GH} we have
$$
X = \frac{\eta_1^{24}}{Q^3(q)},\quad G(x)=1-1728x\quad\text{and}\quad H(x)=240x.
$$
By \eqref{E:Zdefn} the associated weight 2 modular form is given by
$$
Z = \frac{1}{\sqrt{G(X)}} q \frac{\ud}{\ud q} \log X
= \frac{1}{\sqrt{1-1728\frac{\eta^{24}(\tau)}{Q^3(q)}}}
\;q\frac{\ud}{\ud q} \left(\log \eta^{24}(\tau)-\log Q^3(q)\right).
$$
By~\eqref{E:rde1} and~\eqref{E:rde2} this simplifies to
$$
Z = \frac{Q^{3/2}(q)}{R(q)} \left(P(q) - \frac{P(q)Q(q)-R(q)}{Q(q)}\right)
$$
and therefore
$$
Z=Q^{1/2}=Q^{1/2}(q).
$$
Now we will show that $Z$ satisfies~\eqref{E:Zde}.
The derivatives of $Z$ may be calculated using~\eqref{E:rde2} and we find that
$$
Z_q = \frac{1}{6Q^{1/2}}\left(PQ-R\right)
$$
and
$$
Z_{qq} = \frac{1}{72Q^{3/2}}\left(3P^2Q^2+5Q^3-6PQR-2R^2\right).
$$
Hence,
\begin{align*}
2ZZ_{qq}-3(Z_q)^2
&= \frac{1}{36Q}\left(3P^2Q^2+5Q^3-6PQR-2R^2\right) - \frac{1}{12Q}(PQ-R)^2 \\
&= \frac{5}{36Q}\left(Q^3 - R^2\right).
\intertext{On applying~\eqref{E:rde1} we deduce that}
2ZZ_{qq}-3(Z_q)^2
&=  \frac{240 \eta_1^{24}}{Q} = 240 \left(\frac{\eta_1^{24}}{Q^3}\right) \times Q^2 = 240XZ^4
\end{align*}
and so $Z$ satisfies~\eqref{E:Zrewrite}, and hence also~\eqref{E:Zde}, with $H(x) = 240x$.

Next, by~\eqref{E:rec1} the coefficients $T(n)$ in the expansion 
$\displaystyle{Z=\sum_{n=0}^\infty T(n)X^n}$
satisfy the two-term recurrence relation
$$
(n+1)^3T(n+1) = (n+1-\tfrac12) \left\{ 240 - (-1728)(n+1)n\right\}T(n)
$$
and this simplifies to
$$
T(n+1) = 24(2n+1)(6n+5)(6n+1)T(n).
$$
Since $T(0)=1$ it follows that
\begin{equation}
\label{E:Tn}
T(n) = {6n \choose 3n}{3n \choose n}{2n \choose n}.
\end{equation}
Let us apply the checks~\eqref{E:check2} and~\eqref{E:check1} with 
$G(x)=1-1728x$ and $H(x)=240x$.
By~\eqref{E:check1} we should have
$$
2a_2 = g_1+h_1 = -1728 + 240 = -1488 \quad\text{so} \quad a_2 = -744
$$
and this is in agreement with the coefficient of $q^2$ in the expansion
$$
X = \frac{\eta^{24}}{Q^3(q)} = q - 744q^2 + 356652q^3+O(q^4).
$$
For the other check~\eqref{E:check2} we should have
$$
2T(1) = h_1 = 240 \quad \text{and so} \quad T(1) = 120,
$$
and this agrees with the formula for $T(1)$ given by~\eqref{E:Tn}, since
$$
{6 \choose 3} \times {3 \choose 1} \times {2 \choose 1} = 20 \times 3 \times 2 = 120.
$$
The Maple code in Fig.~\ref{fig:1} takes $G(x)=1-1728x$
and $H(x)=240x$ as input, and
uses the method described in Theorem~\ref{T:4} to
compute the expansions of the Hauptmodul $X$ and the associated weight 2 modular
form $Z$ as far as $q^m$, where $m$ is set to~20. In preparing the tables for
this manuscript, the polynomials $G(x)$ and $H(x)$ for each group $\Gamma_0(\ell)+$
were entered and used to compute
the $q$-expansions of $X$ and $Z$ as far as $q^{140}$. The $q$-expansions
were then checked against the formulas in Tables~\ref{T:00} and~\ref{T:Z}.
The data in Table~\ref{T:GH} agrees with~\cite[Table 5]{huberActa}, except
for $\ell=29, 41, 87, 94, 105$ where we have made some corrections and
for $\ell=28, 70, 78$ which were omitted.

\begin{figure}[htp]
    \centering
    \includegraphics[width=12cm]{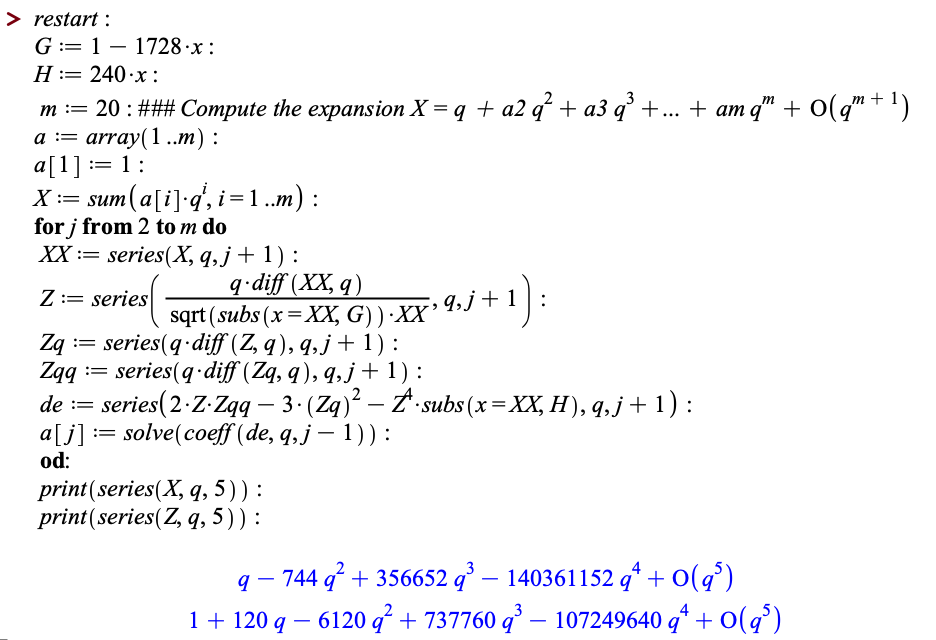}
    \caption{Maple code to compute the expansions of $X$ and $Z$, given $G$ and $H$.}
    \label{fig:1}
\end{figure}

The choice of Hauptmodul is not necessarily unique as the following
example from $\Gamma_0(30)+$ illustrates.
From~\cite[Table 3]{huberActa} or~\cite[Table 3]{jorgenson} we have
$$
\frac{1}{X} = \frac{\eta_1\eta_2\eta_{15}\eta_{30}}{\eta_3\eta_5\eta_6\eta_{10}}
+ \frac{\eta_3\eta_5\eta_6\eta_{10}}{\eta_1\eta_2\eta_{15}\eta_{30}} +c
= \frac{1}{q}+(c+1)+4q+2q^2+6q^3+\cdots
$$
for an arbitrary constant $c$ and so
$$
X = q-(c+1)q^2+(c^2+2c-3)q^3-(c^3+3c^2-5c-5)q^4+\cdots.
$$
The polynomials $G(x)$ and $H(x)$ in Theorem~\ref{T:1} are given by
\begin{align}
G(x) &= (1-(c+6)x)(1-(c+2)x)(1-(c+1)x)(1-(c-2)x)(1-(c-3)x) \label{E:G30}
\intertext{and}
H(x) &= (3c+2)x-(\tfrac{57}{4}c^2+22c-29)x^2+(\tfrac{99}{4}c^3+59c^2-143c-92)x^3
\label{E:H30} \\
&\quad -(\tfrac{75}{4}c^4+60c^3-\tfrac{855}{4}c^2-255c+225)x^4 \nonumber \\
& \qquad +\tfrac{21}{4}(c+6)(c+2)(c+1)(c-2)(c-3)x^5 \nonumber
\end{align}
while the associated weight 2 modular form is given by
$$
Z=\left(\frac{\eta_1\eta_2\eta_3\eta_5\eta_6\eta_{10}\eta_{15}\eta_{30}}{X^3}\right)^{1/2} = 1+\tfrac12(3c+2)q+\tfrac38(c^2+12)q^2+\cdots.
$$
The degrees of $G(x)$ and $H(x)$ both reduce by 1 if $c\in\left\{-6,-2,-1,2,3\right\}$.
We choose the value $c=-2$ to coincide with the choice made
in~\cite[Table 5]{huberActa}, but
any of the other four values could equally well have been used.

Let $X_\ell$ be a Hauptmodul in Table~\ref{T:00} and let $Z_\ell$ be the corresponding
weight two modular form in Table~\ref{T:Z}.
Let $T_\ell(n)$ be the coefficients in the expansion
$$
Z_\ell = \sum_{n=0}^\infty T_\ell(n)X_\ell^n.
$$
A recurrence relation for the coefficients $T_{\ell}(n)$ is given in Theorem~\ref{T:2}.
Let $F_\ell(x)$ be the function defined by 
\begin{equation}
    \label{E:Fdef}
F_\ell(x) = \sum_{n=0}^\infty T_\ell(n)x^n.    
\end{equation}
For example, by~\eqref{E:Tn} we have
$$
    F_1(x) = \sum_{n=0}^\infty {6n \choose 3n}{3n \choose n}{2n \choose n}x^n
    = \pFq{3}{2}{\frac16,\frac12,\frac56}{1,1}{1728x}
$$
while similar analysis for levels 2, 3 and 4 gives
$$
F_2(x) = \sum_{n=0}^\infty {4n \choose 2n}{2n \choose n}^2x^n
    = \pFq{3}{2}{\frac14,\frac12,\frac34}{1,1}{256x},
$$
$$
F_3(x) = \sum_{n=0}^\infty {3n \choose n}{2n \choose n}^2x^n
    = \pFq{3}{2}{\frac13,\frac12,\frac23}{1,1}{108x},
$$
and
$$
F_4(x) = \sum_{n=0}^\infty {2n \choose n}^3x^n
    = \pFq{3}{2}{\frac12,\frac12,\frac12}{1,1}{64x}.
$$
Formulas for $T_\ell(n)$ in terms of binomial coefficients, or as sums of binomial coefficients,
are well-known for $\ell\in\left\{1,2,3,4,5,6,7,8,9,10,12\right\}$; see Table~\ref{T:00X}.

\section{Examples of series for $1/\pi$}
\label{S:pi}
In this section, we discuss some well-known examples, and indicate their locations in the
tables.
\subsection{Bauer's series (1859)}
The series
\begin{equation}
\label{E:Bauer}
\frac14-\frac54\left(\frac12\right)^3+\frac94\left(\frac{1\cdot3}{2\cdot4}\right)^3
-\frac{13}{4}\left(\frac{1\cdot3\cdot5}{2\cdot4\cdot6}\right)^3+\cdots= \frac{1}{2\pi}
\end{equation}
was published by Bauer, \cite{bauer}. It corresponds to the data $X_4<0$, $N=2$ in Table~\ref{T:4.1}.
The $n$th term in the series is
$$
{2n \choose n}^3 \left(n+\frac14\right) \left(\frac{-1}{64}\right)^n
\sim \frac{(-1)^n}{(\pi n)^{3/2}}\left(n+\frac14\right)
= O\left(\frac{1}{\sqrt{n}}\right) \quad \text{as} \quad n\rightarrow\infty,
$$
therefore Bauer's series is conditionally convergent and not practical for computing~$\pi$.
Nevertheless, it is the first known example of a Ramanujan-type series for~$1/\pi$.

\subsection{Ramanujan's series (1914)}
\label{S:ramanujan}
Ramanujan \cite{ramanujan_pi} gave 17 series for $1/\pi$, which appear in our tables as:

\medskip
\noindent
\begin{tabular}{lll}
Table~\ref{T:1.1}, & $X_1>0$, & $N=3,\,7$ \\
Table~\ref{T:2.1}, & $X_2>0$, & $N=3,\,5,\,9,\,11,\,29$ \\
Table~\ref{T:2.2}, & $X_2<0$, & $N=5,\,9,\,13,\,25,\,37$ \\
Table~\ref{T:3.1}, & $X_3>0$, & $N=4,\,5$ \\
Table~\ref{T:4.1}, & $X_4>0$, & $N=3,\,7,\,15$. \\
\end{tabular}

\medskip
The parameters $X_\ell$ and $\lambda$ in Ramanujan's series are all rational
numbers except for the series~\eqref{E:ramirrational} which corresponds to the data
$X_4>0$, $N=15$ in Table~\ref{T:4.1} and involves quadratic irrational numbers. 

The series~\eqref{E:ramanujan-gosper} used by Gosper to calculate $\pi$ corresponds to the data
$X_2>0$, $N=29$ in Table~\ref{T:2.1}.

The series corresponding to $X_4>0$, $N=3$ and $N=7$ in Table~\ref{T:4.1} were
used as a plot device in the Walt Disney movie {\em High School Musical};
see~\cite{bbc} for further details.

\subsection{The Chudnovskys' series (1987)}
The series corresponding to $X_1<0$, $N=163$ in Table~\ref{T:1.2}
was discovered by David and Gregory Chudnovsky~\cite{chud}. The series is
\begin{equation}
\label{E:chudnovsky}
12\sum_{n=0}^\infty {6n \choose 3n}{3n \choose n}{2n \choose n}
(545140134n+13591409) \frac{(-1)^n}{640320^{3n+3/2}} = \frac{1}{\pi}
\end{equation}
and each term adds approximately 14 decimal places of precision.
The Chudnovskys used their series to calculate $\pi$ to more than two billion decimal
places, e.g., see \cite{preston}, which at the time was a world record. 
Because the Chudnovskys' series is the fastest converging series with rational
coefficients, it is often used in world record calculations of~$\pi$.

\subsection{The 36 rational series}
\label{S:36rational}
Sixteen of Ramanujan's series involve rational values of $X$ and $\lambda$. 
Other rational series include Bauer's series~\eqref{E:Bauer} and 
the Chudnovskys' series~\eqref{E:chudnovsky}. Further series discovered in
\cite{BerndtChanLiaw2001,borwein-book,borweinrr,clt,chud}
bring the number of known rational series for levels $1\leq \ell \leq 4$ to 36.
For the record, the 36 rational series are:

\medskip
\noindent
\begin{tabular}{lll}
Table~\ref{T:1.1}, & $X_1>0$, & $N=2,\,3,\,4,\,7$ \\
Table~\ref{T:1.2}, & $X_1<0$, & $N=7,\,11,\,19,\,27,\,43,\,67,\,163$ \\
Table~\ref{T:2.1}, & $X_2>0$, & $N=2,\,3,\,5,\,9,\,11,\,29$ \\
Table~\ref{T:2.2}, & $X_2<0$, & $N=5,\,7,\,9,\,13,\,25,\,37$ \\
Table~\ref{T:3.1}, & $X_3>0$, & $N=2,\,4,\,5$ \\
Table~\ref{T:3.2}, & $X_3<0$, & $N=9,\,17,\,25,\,41,\,49,\,89$ \\
Table~\ref{T:4.1}, & $X_4>0$, & $N=3,\,7$ \\
Table~\ref{T:4.1}, & $X_4<0$, & $N=2,\,4.$ \\
\end{tabular}

\medskip

If divergent series are included, then there are two additional rational series,
corresponding to $X_2<0$, $N=3$, and $X_3<0$, $N=5$;
see Tables~\ref{T:2.2} and~\ref{T:3.2}, respectively. The interpretation of
divergent series is discussed in Sec.~\ref{S:other}.

The question of whether the list of 36 rational series is complete
has been partially investigated by I.~Chen, G.~Glebov and R.~Goenka~\cite[Sec. 12]{chen2}.

\subsection{The Borweins' series (1987)}
The Borweins' series~\cite{borweinrr} in the abstract
correspond to the data $X_1<0$, $N=427$ and $N=61/7$ 
in Table~\ref{T:1.2}. By~\cite[Th.~5.1]{coopergeye} (see also Theorem~\ref{Th:23571123}) we have
\begin{align}
    \sum_{n=0}^\infty T(n) x^n
    &= \frac{7}{(25-80x+24\sqrt{G})^{1/2}} 
    \sum_{n=0}^\infty {6n \choose 3n}{3n \choose n}{2n \choose n} (y^+)^n \label{E:B7} \\
    &= \frac{1}{(25-80x-24\sqrt{G})^{1/2}} 
    \sum_{n=0}^\infty {6n \choose 3n}{3n \choose n}{2n \choose n} (y^-)^n \nonumber
\end{align}
where
$$
y^\pm = \frac{2x^7}{c\pm d\sqrt{G}}
$$
and 
\begin{align*}
c&=(1-21x + 8x^2 )(1-42x + 454x^2 - 1008x^3 - 1280x^4) \\
d&=(1 - 2x)(1 - 8x)(1 - 24x)(1 - 16x - 8x^2 ), \\
G&= (1+x)(1-27x)
\end{align*}
and $T(n)$ are the coefficients in the expansion $Z_7 = \sum_{n=0}^\infty T(n)X_7^n$,
and where $X_7$ and $Z_7$ are as in Tables~\ref{T:00} and~\ref{T:Z}, respectively.
By applying the operator $x \frac{\ud}{\ud x} + \lambda$ 
to the identity~\eqref{E:B7} it follows that
the Borweins' series are equivalent to the rational level~7 series in Table~\ref{T:7.1}
corresponding to $X_7<0$, $N=61$.

\subsection{Sato's series (2002)}
\label{S:Sato}
In an abstract of a talk at the Annual Meeting of the Mathematical 
Society of Japan \cite{sato}, Takeshi Sato announced the series
\begin{equation}
\label{E:Sato}
\frac{846}{\sqrt{5}} \sum_{n=0}^\infty
t(n) \left(n+\frac12-\frac{25\sqrt5}{141}\right)(-1)^k
\left\{\frac{1}{5\sqrt5}\left(\frac{\sqrt5-1}{2}\right)^{15}\right\}^{n+1/2} = \frac{1}{\pi},
\end{equation}
where the coefficients $t(n)$ are defined by the three-term recurrence relation
\begin{equation}
\label{E:5-rec}
(n+1)^3t(n+1) = -(2n+1)(11n^2+11n+5)t(n)-125n^3t(n-1)
\end{equation}
and initial conditions $t(0)=1$, $t(1)=-5$.
Sato's series is based on modular forms of level 5, whereas for Ramanujan's series
and all series found by other authors before Sato,
the corresponding modular forms are level 1, 2, 3 or 4.
The underlying modular forms and polynomials for Theorem~\ref{T:1} that
give rise to Sato's series are
$$
\tilde{X}_5 = \frac{\eta_5^6}{\eta_1^6}, \quad 
\tilde{Z}_5=\frac{\eta_1^5}{\eta_5}, \quad
\tilde{G}_5 = 1+22x+125x^2,\quad\text{and}\quad
\tilde{H}_5 = -10x-125x^2
$$
and it is straightforward to check that the substitution of these data into Theorem~\ref{T:2}
gives the recurrence relation~\eqref{E:5-rec}.
Sato's series is related to the data in our Tables~\ref{T:00}, \ref{T:Z} and~\ref{T:5.2}
through the companion identity given by~\cite[Th. 2.1]{chancooper}.
For, by~\cite[Cor. 3.6]{chancooper} or~\cite[(6.73)]{cooperbook} we have
\begin{equation}
\label{E:sato_1}
\sum_{n=0}^\infty {2n \choose n}s(n) x^{n+1/2}
= \sum_{n=0}^\infty t(n)y^{n+1/2}
\end{equation}
where
\begin{equation}
\label{E:G5}
y=\frac{2x}{1-22x+\sqrt{G}} \quad \text{and} \quad G=1-44x-16x^2
\end{equation}
and $s(n)$ is defined by the three-term recurrence relation
$$
(n+1)^2s(n+1) = (11n^2+11n+3)s(n)+n^2 s(n-1)
$$
and initial conditions $s(0)=1$, $s(1)=3$. Sato's formula is obtained from
the series corresponding to $X_5<0$, $N=47$ 
in Table~\ref{T:5.2} by applying the operator
$x \frac{\ud}{\ud x} + \lambda-\frac12$ to~\eqref{E:sato_1} and then setting
$x=-1/15228$, $\lambda=71/682$ and noting that
$$
\left.\frac{2x}{1-22x+\sqrt{1-44x-16x^2}}\right|_{x=-1/15228} 
= \frac{-1}{5\sqrt5}\left(\frac{\sqrt5-1}{2}\right)^{15}.
$$
It can also be shown (see Theorem~\ref{Th:23571123}) that
\begin{align}
    \sum_{n=0}^\infty {2n \choose n}s(n) x^n
    &= \frac{5}{(13-36x+12\sqrt{G})^{1/2}} 
    \sum_{n=0}^\infty {6n \choose 3n}{3n \choose n}{2n \choose n} (y^+)^n \label{E:B5} \\
    &= \frac{1}{(13-36x-12\sqrt{G})^{1/2}} 
    \sum_{n=0}^\infty {6n \choose 3n}{3n \choose n}{2n \choose n} (y^-)^n \nonumber
\end{align}
where
$$
y^\pm = \frac{2x^5}{c\pm d\sqrt{G}}
$$
and $G$ is as in~\eqref{E:G5}, and
\begin{align*}
c&=1-80x+1890x^2-12600x^3+7776x^4+3456x^5 \\
d&=(1-4x)(1-18x)(1-36x).
\end{align*}
Since
$$
\left.\frac{2x^5}{c\pm d\sqrt{G}}\right|_{x=-1/15228}
= \frac{-1}{2^{12}\cdot3^3\cdot11^3\cdot(8875\pm3969\sqrt5)^3},
$$
it follows that Sato's series is also equivalent to the level~1 series
corresponding to $X_1<0$ and $N=235$ and $N=47/5$ in Table~\ref{T:1.2}. 

Sato also gave 7 series based on the level 6 data
$$
\tilde{X}_6 = \frac{\eta_1^{12}\eta_6^{12}}{\eta_2^{12}\eta_3^{12}}, \quad 
\tilde{Z}_6=\frac{\eta_2^7\eta_3^7}{\eta_1^5\eta_6^5}, \quad
\tilde{G}_6 = 1-34x+x^2,\quad\text{and}\quad
\tilde{H}_6 = 10x-x^2.
$$
The coefficients in the expansion $\tilde{Z}_6=\sum_{n=0}^\infty A(n) \tilde{X}_6^n$
are the Ap{\'e}ry numbers, and by Theorem~\ref{T:2} they satisfy the recurrence relation
$$
(n+1)^3A(n+1) = (2n+1)(17n^2+17n+5)A(n)-n^3A(n-1).
$$
Sato's series are obtained by applying the operator
$x \frac{\ud}{\ud x} + \lambda-\frac12$ to the identity~\cite[Ch. 6, Ex. 12(a), the
second expression for $G_a(x)$]{cooperbook}
$$
\sum_{n=0}^\infty T(n) x^{n+1/2}
= \sum_{n=0}^\infty A(n)\left(\frac{1+2x-\sqrt{1+4x}}{2x}\right)^{n+1/2} 
$$
and using the data in Table~\ref{T:6.1} corresponding to $N=5$, 7, 13, 17, 35, 55 and 77.

\subsection{The other series in Tables~\ref{T:1.1}--\ref{T:105.1}}
\label{S:other}
We briefly discuss the remaining series in Tables~\ref{T:1.1}--\ref{T:105.1}
using Table~\ref{T:5.2} as an example.
Since the level is~5, Tables~\ref{T:00} and~\ref{T:Z}
imply the modular parameterization is
$$
Z=\sum_{n=0}^\infty T(n)X^n
$$
where
$$
X = \frac{\eta_5^6/\eta_1^6}{1+22\eta_5^6/\eta_1^6+125\eta_5^{12}/\eta_1^{12}}
\quad\text{and}\quad
Z=\frac{5P_5-P_1}{4}.
$$
From Table~\ref{T:GH} the associated polynomials are
$$
G(x) = 1-44x-16x^2 \quad\text{and}\quad H(x) = 12x(1+x)
$$
and Theorem~\ref{T:2} implies the coefficients $T(n)$ satisfy the three-term recurrence
relation
$$
(n+1)^3T(n+1)=2(2n+1)(11n^2+11n+3)T(n)+4n(4n^2-1)T(n-1).
$$
By Table~\ref{T:00X} an alternative formula for the coefficients is the binomial sum
$$
T(n) = {2n \choose n} \sum_{k=0}^n {n \choose k}^2{n+k \choose k}.
$$
By~\cite[Th. 10.1]{cooperAperyLike},
$$
T(n) = O\left(\frac{R^n}{n^{3/2}} \right)
$$
where $R=r_0^{-1}$ and $r_0 = 8/(1+\sqrt{5})^5 \approx 0.023$ is the root of $G(x)$ of
smallest modulus. It follows that for any constant $\lambda$, the series 
$$
\sum_{n=0}^\infty T(n)(n+\lambda)X^n
$$
converges for $-r_0 \leq X < r_0$ and diverges otherwise.

The title of Table~\ref{T:5.2} indicates that the power series variable~$X$ is negative.
Since $L=4\ell/\text{gcd}(4,\ell) = 20/\text{gcd}(4,5) = 20$ the
series~\eqref{E:CCLminus} takes the form
\begin{equation}
    \label{E:level5series}
    \sum_{n=0}^\infty T(n)(n+\lambda)X^n
    = \frac{1}{2\pi} \times \frac{1}{\sqrt{1-44X-16X^2}} \times \sqrt{\frac{20}{N}}
\end{equation}
where $X=X_5(-e^{-2\pi\sqrt{N/L}})=X_5(-e^{-\pi\sqrt{N/5}})$.

The series corresponding to $N=1$, $2$, $3$, $5$ and $7$ in Table~\ref{T:5.2}
all diverge because $X$ is outside the 
interval of convergence and this is indicated by placing the
symbol~``d'' next to the value of~$N$ in the table. Divergent series can be 
interpreted by means of the formula
\begin{equation}
    \label{E:level5series1}
    X\left.\left(q\frac{\ud Z}{\ud q}\right)\right/\left(q\frac{\ud X}{\ud q}\right) +\lambda Z
    = \frac{1}{2\pi} \times \frac{1}{\sqrt{1-44X-16X^2}} \times \sqrt{\frac{20}{N}}
\end{equation}
which is obtained from~\eqref{E:level5series} by the chain rule, and
where both sides are evaluated at $q=-e^{-\pi\sqrt{N/5}}$.
The $q$-expansions of $X$ and $Z$ can be used in~\eqref{E:level5series1}
to numerically compute the values of $\lambda$
corresponding to the divergent series in Table~\ref{T:5.2}.
The value of $\lambda$ corresponding to $N=1$ in Table~\ref{T:5.2} is
undefined because $X_5(-e^{-\pi\sqrt{1/5}}) = -8/(\sqrt{5}-1)^5$ is a root of $G(x)$
and therefore the right hand side of~\eqref{E:level5series} is undefined.

The series corresponding to $N=47$ in Table~\ref{T:5.2} is (see~\cite[(4)]{chancooper})
$$
\sum_{n=0}^\infty T(n)\left(n+\frac{71}{682}\right)\left(\frac{-1}{15228}\right)^n
    = \frac{1}{2\pi} \times \frac{3807\sqrt{5}}{8525} \times \sqrt{\frac{20}{47}}
    = \frac{81\sqrt{47}}{1705}
$$
and this is equivalent to Sato's series~\eqref{E:Sato} as discussed in Section~\ref{S:Sato}.

The conjugate series of a quadratic irrational series is the series obtained
by replacing each number $a+b\sqrt{r}$ with $a-b\sqrt{r}$ for both $X$ and $\lambda$.
Three possible things can occur with the conjugate of a quadratic irrational value of~$X$:

\begin{enumerate}
\item 
The value of $X$ changes sign, as for $N=1$, $2$, $5$, $9$ and $17$. The
conjugate series appear in the previous Table~\ref{T:5.1} as $N=1$, $8$, $5$, $9$ and $17$,
respectively. The text ``Tbl.~\ref{T:5.1}'' is inserted in the column
headed~$\overline{N}$ to indicate that the conjugate series appears in the other table.
\item 
The conjugate series occurs as an instance of~\eqref{E:level5series}
or~\eqref{E:level5series1} for another value of~$N$. This happens
for $N=15$, $39$, $63$, $87$, $111$, $119$, $143$, $159$ and $287$ where the
respective conjugate series correspond to the instances
$N=5/3$, $13/3$, $9/7$, $29/3$, $37/3$, $17/9$, $13/11$, $53/3$ and $41/7$,
as recorded in the column headed~$\overline{N}$.
Some of the conjugate series diverge, indicated by the additional letter ``d'' in which
case the data are valid for the formula~\eqref{E:level5series1}.
\item 
For $-1<q<1$ we
have\footnote{Obviously, $X(0)=0$. For positive $q$ it is easy to show
using~\eqref{E:Zdefn} that
$X(q)$ has a single maxima at \mbox{$q=\exp(-2\pi/\sqrt{5})$}, and the modular
transformation for the Dedekind eta function, e.g., ~\cite[Th. 2.8]{cooperbook}
can be used to show
$\lim_{q\rightarrow 1^-} = 0$. For negative $q$,
$X(q)$ has a single minima at \mbox{$q=-\exp(-\pi/\sqrt{5})$}, and
$\lim_{q\rightarrow -1^+} = 0$.} 
$$
\frac{-8}{(\sqrt{5}-1)^5} = X(-e^{-\pi/\sqrt{5}})
\leq X(q) \leq X(e^{-2\pi/\sqrt{5}}) = \frac{8}{(1+\sqrt{5})^5}$$
and for each value 
$N=11$, $31$, $71$, $191$ and $311$ the conjugate value of $X$ falls outside of this
range. In these cases we simply write the single
letter ``d'' (with no accompanying numerical value)
in the column headed~$\overline{N}$.
\end{enumerate}

This concludes our detailed discussion of Table~\ref{T:5.2}. The data in the
other Tables~\ref{T:1.1}--\ref{T:105.1} is interpreted in a similar way.

For each series in the tables, reference to the first discovery and the first proof is given, where known.
Occasionally, additional references are given.
We have not attempted to refer in the tables to
every author who has studied a particular series. Instead, refer
to the historical survey in the next section.

\section{Historical survey}
\label{S:historical}
At this point, it is appropriate to give a brief historical account.
Readers seeking an introduction to the topic are encouraged to begin with the survey
articles~\cite{bbc,ramanujan_and_pi,bbe3,gs} 
and the books~\cite{sourcebook,borwein-book,chan-book}.

The subject began with Ramanujan's original 1914
paper~\cite{ramanujan_pi}, which introduced
several families of rapidly converging series for~$1/\pi$
derived from Jacobi's theory of elliptic functions and what he termed
``corresponding theories". Ramanujan's 17 series for 
$1/\pi$ are described above in Sec.~\ref{S:ramanujan} and appear
as equations (28)--(44) in~\cite{ramanujan_pi}.
All 17 series also occur in Ramanujan's notebooks; see~\cite[pp. 352--354]{Part4}
for details.
Three of the seventeen series---those of level~4 with $N=3$,~7 and 15---are
recorded in Ramanujan's lost notebook~\cite[p. 370]{lost};
see~\cite[pp. 375--384]{lost2} and~\cite[pp. 382--383]{lost5} for
analysis and commentary by G. E. Andrews and B. C. Berndt.

Bauer's series~\eqref{E:Bauer} (1859) is a special case of Ramanujan's theory.
It does not appear among Ramanujan's series
in~\cite{ramanujan_pi}, but Ramanujan did record it in his
second notebook~\cite[Ch. 10, pp. 23--24]{Part2},
\cite[p.~118, Eq. 14]{notebooks}
and it is in his first letter
to Hardy in 1913, e.g., see~\cite[p.~25,~(3)]{letters}.
More information about Bauer's series is given
by J. M. Campbell and P.~Levrie~\cite{campbell-levrie}.

Ramanujan’s series for $1/\pi$ in~\cite{ramanujan_pi} were first studied by
S.~Chowla in 1928~\cite{chowla1, chowla2, chowla3}, who gave a general result
for level~4 and proved the series corresponding to $X_4>0$, $N=3$ in Table~\ref{T:4.1}.

Ramanujan's paper rose to prominence in the 1980s
with the world record setting calculation of R. W. Gosper~\cite[p. 341]{borwein-book}
using~\eqref{E:ramanujan-gosper},
the book by Jonathan and Peter Borwein~\cite{borwein-book}, and the 
discovery of~\eqref{E:chudnovsky} by David and Gregory Chudnovsky~\cite{chud}.
The topic became more accessible following the detailed study of the ``corresponding theories''
by B. C. Berndt, S.~Bhargava and F.~G.~Garvan in 1995~\cite{bhargava}.

Additional series were obtained by the Borweins
in~\cite{borwein-indian,borweinrr,borweinclass3}. An overview of
the Borweins' work on $\pi$ has been given by R.~P.~Brent~\cite{brent}.

The theory behind the Chudnovskys' formula~\eqref{E:chudnovsky} is outlined in their paper~\cite{chud}. An interesting popular account of the Chudnovskys' work
has been given by R.~Preston~\cite{preston}.

In 2000, H.~H. Chan and W.-C. Liaw analysed the level~3 theory in~\cite{CL2000}
and discovered the series corresponding to $N=7$, 10, 11, 14, 19, 26, 31, 34 and
59 in Table~\ref{T:3.1}. They also exhibited series corresponding to
$N=35$, 55, 70, 91, 110, 115, 119, 154 and 455 that involve algebraic numbers of
degree $>2$ that can be expressed as sums of real quadratic numbers.
In the following year, Chan, Liaw and V.~Tan further investigated
the level~3 theory~\cite{clt},
and discovered the series corresponding to $N=9$, 17, 25, 41, 49 and 89
in Table~\ref{T:3.2}.

B.~C.~Berndt, H.~H.~Chan, and W.-C.~Liaw~\cite{BerndtChanLiaw2001} analysed the level 2 theory in~2001, and discovered and proved three new series: 
$N=2$ and $N=7$ in Table~\ref{T:2.1}, and $N=7$ in Table~\ref{T:2.2}.
Berndt and Chan~\cite{bc2001} obtained the level 1 series in Table~\ref{T:1.2}
corresponding to $N=35$ and 51. They also determined the series corresponding
to $N=3315$, involving quadratic irrationals of degree 8,
which converges at the extremely fast rate of more than 73 decimal places per term.

Sato's famous abstract~\cite{sato}, discussed in Sec.~\ref{S:Sato}, was published in 2002.

Sato's abstract motivated H.~H.~Chan, S.~H.~Chan and Z.-G.~Liu~\cite{domb} (2004)
to discover and prove their fundamental result (Theorem~\ref{T:3}).
They also discovered and proved the series corresponding
to $\ell=12$, $N=5$ in Table~\ref{T:12.1}, where the underlying coefficients 
are the Domb numbers.

In work completed in 2004 and published in 2008, J.~Guillera~\cite{g04} used the WZ-method to prove 
continous one-parameter extensions of seven series for $1/\pi$.
Specifically, setting $a=0$ in~\cite[Identities 1--7]{g04} gives the series in our tables corresponding to
$(X_4>0,N=3)$, $(X_4>0,N=7)$, $(X_4<0,N=4)$, $(X_2<0,N=5)$, $(X_1<0,N=11)$, $(X_2>0,N=3)$
and $(X_2<0,N=9)$, respectively. Guillera also gave three similar extensions for series
involving $1/\pi^2$.

In 2005, W. Zudilin~\cite{z2005} used the Chudnovskys’ series~\eqref{E:chudnovsky} to obtain an improved estimate for the irrationality exponent $\mu(\pi\sqrt{10005})$; see also~\cite{sequel}.
In 2008, he surveyed proof methods
for Ramanujan-type series, generalizations, and open problems~\cite{secondwind}.

In 2008, N.~D.~Baruah and B.~C.~Berndt~\cite{BB2008} proved
the series corresponding to level~2 and $X_2>0$ for $N=2$, 3, 4, 5, 7, 9, 11
and the series for level~3 and $X_3>0$ for $N=2$, 3, 5.
In 2010~\cite{BB2010}, they proved 49 series for $1/\pi$ corresponding
to levels 1, 2 and 4, including several series that are new; 42 of these series are
referenced in Tables~\ref{T:1.1}, \ref{T:1.2}, \ref{T:2.1}, \ref{T:2.2} and~\ref{T:4.1}
below, while the other 7 series involve algebraic numbers of degree $>2$ which are outside
the scope of this work.
Different proofs for series corresponding to
levels 1, 2 and 4 and $N\in\left\{3,5,7,9,25\right\}$
have been given by Baruah and Nayak~\cite{nayak} in 2010.

In 2009, M.~Rogers~\cite{rogers} discovered two further series involving Domb numbers 
corresponding to the cases $X_{12}>0$, $N=2$ and $X_{12}<0$, $N=4$ in Table~\ref{T:12.1}.
He also gave a level 6 series where the coefficients
are ${2n \choose n}\sum_{k=0}^n{2k \choose k}{n \choose k}^2$.

Also in 2009, H.~H. Chan and H.~Verrill~\cite{verrill} discovered several additional series
involving the Domb numbers, along with corresponding series for
the Ap{\'e}ry and Almkvist--Zudilin numbers.

In 2011, H.~H.~Chan, Y.~Tanigawa, Y.~Yang and W.~Zudilin~\cite{ctyz} discovered
the level 6 series in Table~\ref{T:6.1} corresponding to $N=2$, 3, 5, 7, 13 and 17,
along with two other families of identities associated
with subgroups of $\Gamma_0(6)$
in $\operatorname{SL}_2(\mathbb{R})$. They also gave the level 10
series corresponding to $N=3$ in Table~\ref{T:10.1}.

An extensive set of conjectures was published in 47 versions by Z.-W. Sun~\cite{zwsun} 
during 2011--2014. Further conjectures have been given by Sun in~\cite{sun14,sun23,Sun}. Solutions
to some of the
conjectures include the works by H. H. Chan, J. Wan and W. Zudilin~\cite{ChanWanZudilin2013};
M.~Rogers and A.~Straub~\cite{RogersStraub};
S.~Cooper, J.~Wan and W.~Zudilin~\cite{alchemy}; and L.~Wang and Y.~Yang~\cite{wangyang}.

Series for levels 7, 10 and 18 were developed in~\cite{cooperlevel7} and
level 10 series were studied further in~\cite{cooper10a}.

Chan and Cooper~\cite{chancooper} gave a classification of series for levels 
$1\leq \ell \leq 9$ in~2012, exhibiting 93 series, 40 of which were believed to be new.
Each series was also shown to have a companion identity, some of which appear
in~\cite{domb} and~\cite{verrill} and involve Ap{\'e}ry, Domb, and Almkvist--Zudilin numbers.

Quadratic irrational series for levels $1\leq \ell \leq 4$ were studied in the 2012 thesis
of~A.~M.~Aldawoud~\cite{aldawoud}, leading to the discovery of several new series.

A series with complex coefficients was exhibited in 2011 by
J. Guillera and W. Zudilin~\cite[(48)]{gz}.
This led T. Piezas \cite{piezas} in 2012 to discover the series
$$
\sum_{n=0}^\infty {3n \choose n}{2n \choose n}^2
\frac{(17+i)n+3}{(1-i)^{3n}(1+2i)^{6n}} = \frac{-3(1+2i)^3}{4\pi} \quad \text{where}\quad i^2=-1
$$
which involves only Gaussian rational numbers.
It is a consequence of the fact that
$$
X_3(e^{2\pi i/3}\cdot e^{-4\pi\sqrt{2}/3}) 
= \left.\frac{\eta_1^{12}\eta_3^{12}}{(\eta_1^{12}+27\eta_3^{12})^2}
\right|_{q=e^{2\pi i/3}\cdot e^{-4\pi\sqrt{2}/3}}
= \frac{1}{(1-i)^{3}(1+2i)^{6}}.
$$
Further examples of complex series have been given by
H. H. Chan, J. Wan and W. Zudilin~\cite{ChanWanZudilin2012} in 2012 and
G.~Almkvist and J.~Guillera~\cite{almkvist} in 2013.

A different method, which produces families of series for $1/\pi$ that involve a continuous parameter,
was given by J. Wan~\cite{wan} in 2014.

Quadratic irrational series corresponding to $N=5$, 8 and 18 in the level 1 theory
with $X_1>0$ were found by
N.~D.~Bagis and M.~L.~Glasser~\cite{bagis2} in 2012. They also determined the series corresponding
to $N=27$, 432 and 1728 which involve algebraic numbers of degree $>2$.
In the following year~\cite{bagis} they found further irrational series corresponding to levels~1 and~4.

Level 13 series were investigated in the undergraduate
thesis of D.~Ye in 2013 \cite{YeThesis}, with
results published in~\cite{cooperye131, cooperye13} joint with S. Cooper.

Series for levels 11 and 23 were given by S. Cooper, J. Ge and D. Ye~\cite{coopergeye} in~2015.
The corresponding results for levels 14 and 15 were
developed by Cooper and Ye~\cite{cooperye14} in 2016.
The level 16 theory was developed by Ye~\cite{ye16} in 2016.
Series for levels $1\leq \ell \leq 12$ were classified in~\cite[Ch. 14]{cooperbook}.

In 2018, A. Berkovich, H. H. Chan and M. J. Schlosser~\cite{bcs} developed a technique
and used it to prove several series for levels 2, 3 and 4, including the
Ramanujan--Gosper series~\eqref{E:ramanujan-gosper}.
They also discovered new level 4 series for $X_4>0$ corresponding to $N=11$ and $23$,
and for $X_4<0$ for $N=5$, 14 and 46; all of these new series involve irrational
numbers of degree $>2$.

T. Huber, D. Schultz, and D. Ye obtained series for level 20 in 2018~\cite{huber20} and for level 17 in 2020~\cite{huber17}.
They developed a general theory of Ramanujan--Sato series, including the fundamental result
discussed above in Theorem~\ref{T:1}, in~\cite{huberActa} (2023).

Series for levels 21, 22, 33 and 35 were found by
T. Anusha, E. N. Bhuvan, S. Cooper and K. R. Vasuki in~\cite{anusha} (2019).

Detailed proofs of several of Ramanujan's series were given in a series
of articles by J.~Guillera \cite{g16,g18,g20,g21,g21a,g25} during 2016--2025.
In particular, proofs of the Chudnovskys' series~\eqref{E:chudnovsky}
and the Ramanujan--Gosper series~\eqref{E:ramanujan-gosper} are in~\cite{g21a} and~\cite{g25}, respectively.
Interesting extensions have been given in~\cite{g15,g17,g19,g25a} and in work
with M. Rogers in~\cite{g-rogers}. In 2021, H. Cohen and J. Guillera~\cite{cohen}
gave a unified proof of
the 36 rational series in Sec.~\ref{S:36rational}, along with a
detailed survey of several generalizations.

The Chudnovskys' method was analysed by I. Chen and G. Glebov in 2018~\cite{chen1} and 
used to prove the rational level~1 series in Tables~\ref{T:1.1} and~\ref{T:1.2}.
Proofs of all~36 rational series in Sec.~\ref{S:36rational} were given
by I. Chen, G. Glebov and R. Goenka in~2022~\cite{chen2}.

Y.~Zhao proved the Ramanujan--Gosper series~\eqref{E:ramanujan-gosper} and the level 2 series with $N=37$, $X_2<0$ in 2020~\cite{zhao}, and proved the Chudnovskys’ series~\eqref{E:chudnovsky} in~\cite{zhao163}.
Another proof of the Chudnovskys' series was given in 2022 by L.~Milla~\cite{milla}.

In 2024, D. Ye~\cite{ye24} showed how Ramanujan-type series arise from genus zero subgroups of $\text{SL}_2(\mathbb{R})$
that are commensurable with $\text{SL}_2(\mathbb{Z})$ and illustrated the theory with both known and new
examples.

In 2024, A. Babei, L. Beneish, M. Roy, H. Swisher, B. Tobin and F.-T. Tu \cite{babei}
used the Chan--Chan--Liu Theorem (Theorem~\ref{T:3}) to derive several series for $1/\pi$.

In 2005, J. M. Campbell~\cite{CampbellLevel3} proved the level 3 series
in Table~\ref{T:3.2} for $N=19, 31,$ and~$33$. He proved two series involving Domb numbers~\cite{CampbellDomb}.

Connections of Ramanujan's series with conformal field theory have
been discussed by F. Bhat and A. Sinha \cite{bhat} in 2025.

Due to the size of the subject, this historical survey has been deliberately restricted
to Ramanujan-type series for $1/\pi$. We have not included the related topics of congruences, $q$-analogues
or iterative formulas.
There is also the topic of series that converge to $1/\pi^m$ for $m>1$ and to other values; here
we just quote the single example
$$
\sum_{n=0}^\infty {4n \choose 2n}{2n \choose n}^8
\frac{(43680n^4 +20632n^3 +4340n^2 +466n+21)}{2^{32n+11}} = \frac{1}{\pi^4}
$$
discovered by Jim Cullen in 2010 (see \cite[(46)]{g19a}) that has been proved recently by
K.~C.~Au~\cite{au} and J.~Guillera~\cite{g26}.
Each of these topics is a large subject in its own right.

\section{Convergence}
\label{S:Con}
In this section we analyse convergence of the series
\begin{equation}
\label{E:type}
\sum_{n=0}^\infty T(n) (n+\lambda) X^n
\end{equation}
where $T(n)$ is as for Theorem~\ref{T:2}.
By direct calculation, each polynomial $G_\ell(x)$ in Table~\ref{T:GH} has a unique root
of minimum modulus which we call $r_\ell$.
Since $r_\ell$ is real and non-zero we may define sets $S^+$ and $S^-$
by
$$
S^+ = \left\{ \ell: r_\ell>0 \right\} \quad\text{and}\quad 
S^- = \left\{ \ell: r_\ell<0 \right\}.
$$
For example, since $G_4(x)= 1-64x$, the minimal root is $r_4=\frac1{64}>0$ and so $4\in S^+$.
As another example, $G_{30}(x) = (1+5x)(1+4x)(1+x)(1-4x)$
has minimal root $r_{30}=-\tfrac{1}{5}<0$ and so $30 \in S^-$.
We find that
\begin{align}
S^+ &= \left\{ 1, 2, 3, 4, 5, 6, 7, 8, 9, 10, 11, 12, 13, 14, 15, 16, 17, 18, 19, 20, 21, 22, 23, 24, 25,  \right. \nonumber \\
& \qquad 26, 27, 28, 29, 31, 32, 33, 34, 35, 36, 38, 39, 41, 44, 45, 46, 47, 49, 51, 54, 55, 56, \nonumber \\
& \qquad \left.  59, 60, 62, 66, 69, 70, 71, 78, 87, 92, 94, 119\right\} \nonumber
\intertext{and}
S^-&= \left\{30,42,50,95,105,110\right\} \label{E:Sminus}
\end{align}
so that $S = S^+ \cup S^-$ is the set of $\ell$ for which the modular
curve associated with $\Gamma_0(\ell)+$ has genus zero.

The property of whether $\ell \in S^+$ or $\ell \in S^-$
depends on the choice of Hauptmodul. For example, from~\eqref{E:G30} we have that
$$
G_{30}(x) = \begin{cases}
(1+9x)(1+8x)(1+5x)(1+4x) & \text{if $c=-6$} \\
(1+5x)(1+4x)(1+x)(1-4x) & \text{if $c=-2$} \\
(1+4x)(1+3x)(1-x)(1-5x) & \text{if $c=-1$} \\
(1+x)(1-3x)(1-4x)(1-8x) & \text{if $c=2$} \\
(1-x)(1-4x)(1-5x)(1-9x) & \text{if $c=3$}
\end{cases}
$$
and therefore $r_{30}=-\tfrac19, -\tfrac15, \tfrac15,\tfrac18$ or $\tfrac19$
according to whether $c=-6,-2,-1,2$ or $3$, respectively. Hence, $30 \in S^-$ if
$c=-6$ or $c=-2$ and $30 \in S^+$ in the other cases. The polynomial
$G_{30}(x)$ in Table~\ref{T:GH} uses the value $c=-2$, and this accounts for 
$30 \in S^-$ in~\eqref{E:Sminus}.

We are now ready for a precise analysis
of convergence of series of the type~\eqref{E:type}
that occur in Theorem~\ref{T:3}.
\begin{theorem}
The following statements hold.
\begin{enumerate}
\item
For each $\ell\in S^+$ we have $X(e^{-2\pi/\sqrt{\ell}})=r_\ell>0.$
\item
For each $\ell\in S^-$ we have $X_\ell(-e^{-2\pi/\sqrt{L}})=r_\ell<0$, where
$$
L = \frac{4\ell}{(4,\ell)} = \begin{cases} 
\ell & \text{if} \quad\ell \equiv 0 \pmod{4}, \\
2\ell & \text{if} \quad\ell \equiv 2 \pmod{4}, \\
4\ell & \text{if} \quad\ell \equiv 1 \pmod{2}.
\end{cases}
$$
\item
For any constant $\lambda$, the series~\eqref{E:type}
converges absolutely for $|X| < |r_\ell|$.
\item
For $\ell \in S^+$, the series~\eqref{E:type} converges
conditionally for $X=-|r_\ell|$ and diverges for $X=|r_\ell|$.
\item
For $\ell \in S^-$, the series~\eqref{E:type}
converges conditionally for $X=|r_\ell|$ and
diverges for $X=-|r_\ell|$.
\end{enumerate}
\end{theorem}
\begin{proof}
Statements (1) and (2) can be verified on a case-by-case basis. 
The other statements follow from the asymptotic formula for~$T(n)$ in~\cite[Th. 10.1]{cooperAperyLike}.
\end{proof}
From the asymptotic formula for~$T(n)$,
for $\ell \in S^+$ the coefficients
$T_\ell(n)$ are either all positive or are positive for $n$ sufficiently large
(it appears that $n\geq 18$ is large enough). 
And for $\ell \in S^-$ the coefficients
$T_\ell(n)$ either all alternate in sign, or they alternate in sign for
$n$ sufficiently large
(it appears that $n\geq 39$ is large enough).

\begin{example}
If $\ell=4\in S^+$, the associated polynomial is $G_4(x)= 1-64x$ so the
minimal root is $r_4=\frac1{64}>0$. Furthermore,
$$
X_4(e^{-2\pi/\sqrt{4}}) = \left.q\prod_{j=1}^\infty
\frac{(1-q^j)^{24}(1-q^{4j})^{24}}{(1-q^{2j})^{48}}\right|_{q=e^{-\pi}} = \frac{1}{64}
= r_4.
$$
And, the series
$$
\sum_{n=0}^\infty {2n \choose n}^3 (n+\lambda)X^n
$$
converges absolutely for $-\frac1{64}<X<\frac1{64}$, converges conditionally
for $X=-\frac1{64}$, and diverges for $X=\frac1{64}$.

In the conditionally convergent case we also have
$$
X(-e^{-\pi\sqrt{2}})=-\tfrac1{64}=-r_4
$$
and application of Theorem~\ref{T:3} implies $\lambda=\tfrac14$. The resulting series
in this case is Bauer's series~\eqref{E:Bauer}.
It converges extremely slowly: the terms go to zero like $1/\sqrt{n}$, 
which is even slower than the Leibniz formula
$$
\frac{\pi}{4} = \arctan 1 = 1-\frac13+\frac15-\frac17+\cdots.
$$
\end{example}

We shall call a conditionally convergent series of the form~\eqref{E:CCL}
or~\eqref{E:CCLminus} a Bauer series for~$1/\pi$.
Bauer series can be identified for $\ell = 4, 16, 24, 28, 30, 36, 60, 105$,
with Bauer's original series corresponding to $\ell=4$.
For other values of $\ell$ we have been unable to identify the value of~$q$
for which $X_\ell(q) = -r_\ell$.

\begin{example}
A Bauer series for $\Gamma_0(30)+$ 
is
\begin{equation}
    \label{E:Bauer30}
\sum_{n=0}^\infty T(n) (n+\tfrac23) \left(\tfrac15\right)^n=\frac{25}{12\pi}
\end{equation}
where the coefficients $T(n)$ are defined by the recurrence relation~\eqref{E:rec1}
using the polynomials~\eqref{E:G30} and \eqref{E:H30} with $c=-2$, i.e.,
$$
G_{30}(x) = (1+5x)(1+4x)(1+x)(1-4x)
\quad\text{and}\quad
H_{30}(x) =-4x(1-4x-58x^2-75x^3).
$$
A proof of~\eqref{E:Bauer30} using Theorem~\ref{T:3} is as follows.
The modular equation of degree~2 is
\begin{align*}
& f(x,y) := \\
 &   \frac{1}{x^4}+\frac{1}{y^4}
    -\left(\frac{1}{x^3y^2}+\frac{1}{x^2y^3}\right)
    -\left(\frac{1}{x^3y}+\frac{1}{xy^3}\right)
    +12\left(\frac{1}{x^3}+\frac{1}{y^3}\right)
    \\
 & \quad   -\frac{6}{x^2y^2}
    +5\left(\frac{1}{x^2y}+\frac{1}{xy^2}\right)
    +56\left(\frac{1}{x^2}+\frac{1}{y^2}\right)
    +\frac{66}{xy}
    +120\left(\frac1x+\frac1y\right)
    +100=0
\end{align*}
where $x=X_{30}(q)$ and $y=X_{30}(q^2).$
By~\eqref{E:X} we have
$$
X_{30}(e^{-2\pi\sqrt{2/30}})=X_{30}(e^{-2\pi/\sqrt{2\times 30}})
$$
and so $X_{30}(q^2)=X_{30}(q)$ when $q=e^{-2\pi/\sqrt{2\times 30}}$.
Since the only positive root of
$$
f(x,x) = \frac{-2(1-5x)(1+x)^2(1+6x+10x^2)}{x^5}
$$
is $\tfrac15$, it follows that $X_{30}(e^{-2\pi\sqrt{2/30}})=\tfrac15$, and this accounts
for the power series variable in~\eqref{E:Bauer30}. The value of~$\lambda$
can be calculated using Part~\eqref{E:lambda} of Theorem~\ref{T:3}
and following the procedure
detailed in~\cite[p. 635]{cooperbook}, and thus we obtain~$\lambda=\tfrac23.$
The value on the right hand side can be computed by~\eqref{E:CCL} and is given by
$$
\frac{1}{2\pi} \times \frac{1}{\sqrt{G_{30}(\tfrac15)}} \times \sqrt{\frac{30}{2}}
= \frac{1}{2\pi} \times \sqrt{\frac{125}{108}} \times \sqrt{\frac{30}{2}} =
\frac{25}{12\pi}.
$$
A numerical check of~\eqref{E:Bauer30} by directly summing the series is not possible
because the convergence is too slow: the terms $T(n)(n+\tfrac23)(\tfrac15)^n$ 
oscillate in sign and decay in magnitude like $1/\sqrt{n}$.
Instead, the $q$-expansions can be used for the check. For we have
$$
\sum_{n=0}^\infty T(n)(n+\lambda)X^n
= X\frac{\ud Z}{\ud X} + \lambda Z = X\frac{\ud Z/\ud q}{\ud X/\ud q} + \lambda X,
$$
where the $q$-expansions for $X$ and $Z$ can be calculated by the method described in
Theorem~\ref{T:4}, i.e., use the Maple code in Fig.~\ref{fig:1} and just
change the polynomials $G(x)$ and $H(x)$ to $G_{30}(x)$ and $H_{30}(x)$, respectively.
By computing each series up to $q^{50}$
and substituting the values $q=e^{-2\pi\sqrt{2/30}}$ and $\lambda=\tfrac23$,
we obtain numerical confirmation of~\eqref{E:Bauer30} to 21 decimal places.
\end{example}

In~\cite[Th. 9.1]{hme}, hypergeometric
transformation formulas were used to show that Bauer's series is equivalent to
the level 1 series corresponding to $N=2$ in Table~\ref{T:1.1} and also to the
level 2 series corresponding to $N=9$ in Table~\ref{T:2.1}. 
In a similar way the Bauer series for level~30
in~\eqref{E:Bauer30} is equivalent to the level 15 series corresponding to $N=4$ in Table~\ref{T:15.1}. This can be realized via the following result.

\begin{theorem}
    Consider $X_{15}$ and $X_{30}$ as modular functions for $\Gamma_{0}(15)+\cap \Gamma_{0}(30)+$. Then one has
    $$
   \frac{1}{2\sqrt{2}} X_{15}\sqrt{1-\sqrt{1-16X_{30}^{2}}}Z_{15}(X_{15})=X_{30}^{2}Z_{30}(X_{30}),
    $$
    and
    $$
    5X_{15}^2 X_{30}^3+11X_{15}^2 X_{30}^2-2 X_{15}X_{30}^3+7 X_{15}^2 X_{30}-3 X_{15} X_{30}^2+X_{30}^3+X_{15}^2-X_{15} X_{30}=0.
$$
In particular, when $X_{30}=\frac{1}{5}$, one of the roots of the modular equation above is $\frac{1}{30}$, exactly the value of $X_{15}$ for $N=4$ given in Table~\ref{T:15.1}.
\end{theorem}

\begin{proof}
    The verifications of the modular identities are routine, so we omit the details. We briefly explain how $X_{30}(i\sqrt{2/30})=\frac{1}{5}$ corresponding to $N=2$ in Table~\ref{T:30.1} is related to $X_{15}(i\sqrt{4/15})=\frac{1}{30}$ corresponding to $N=4$ in Table~\ref{T:15.1} via the identity above. First of all, note by the Galois correspondence of ramified coverings, there is an index-2 subgroup $\Gamma$ of $\Gamma_{0}(30)+$ such that $\mathbb{C}(X(\Gamma))=\mathbb{C}(X_{15},X_{30})$ and $\Gamma\backslash\Gamma_{0}(30)+=\left\{I, \xi=\begin{pmatrix}
        0&-1\\30&0
    \end{pmatrix}\right\}$. View $X_{30}(\tau)$ as a fixed transcendental element of $\mathbb{C}(X(\Gamma))$ that satisfies $X_{30}(\tau)=q+O(q^{2})$ near the cusp $[i\infty]$. So, the Galois conjugates of 
    $$
    5y^2 X_{30}^3+11y^2 X_{30}^2-2 yX_{30}^3+7 y^2 X_{30}-3 y X_{30}^2+X_{30}^3+y^2-y X_{30}=0
    $$
    over $\mathbb{C}(X_{30})$ are exactly $X_{15}(\tau)=q+O(q^{2})$ and $X_{15}|\xi$, and when $X_{30}=\frac{1}{5}$, these are either $\frac{1}{12}$ or $\frac{1}{30}$. In the cover corresponding to~$I$, since $X_{15}(\tau)=q+O(q^{2})$, one must have locally $y=X_{30}+O(X_{30}^{2})$ that follows that $X_{15}(i\sqrt{2/30})=\frac{1}{12}$ given that $X_{30}(i\sqrt{2/30})=\frac{1}{5}$. So, as a result, one has $X_{15}(\xi\cdot(i\sqrt{2/30}))=\frac{1}{30}$. Notice that $\xi\cdot(i\sqrt{2/30})=i/2\sqrt{15}=\sigma\cdot(i\sqrt{4/15})$, where $\sigma=\begin{pmatrix}
        0&-1\\15&0
    \end{pmatrix}$, and $X_{15}|\sigma=X_{15}$. Therefore, one has $X_{15}(i\sqrt{4/15})=~\frac{1}{30}$.
\end{proof}

\section{Series equivalent to the level 1 series}
\label{S:level1}
Identities~\eqref{E:B7} and~\eqref{E:B5} involving the Borweins' and Sato's series
are special cases of the instances $\ell=7$ and $\ell=5$, respectively, of
the following more general result.
\begin{theorem}
\label{Th:23571123}
For $\ell \in \left\{2,3,5,7,11,13,23\right\}$, let $a$, $b$, $c$, $d$, $g$ and $K$ be the
corresponding polynomials in Table~\ref{T:23571123}, and note that $g$ in Table~\ref{T:23571123} is the same
as~$G$ in Table~\ref{T:GH} except for $\ell=13$ where $g=1-10x-27x^2$ and $G=(1+x)(1-10x-27x^2)$.
Let
$$
y^+=\frac{2x^\ell}{c+d\sqrt{g}} \quad\text{and}\quad
y^-=\frac{2x^\ell}{c-d\sqrt{g}},
$$
and let $F_\ell(x)$ be as in~\eqref{E:Fdef}.
Then in a neighborhood of $x=0$ we have
\begin{equation}
\label{E:n0}
y^+ =x^\ell+O(x^{\ell+1})\quad\text{and}\quad y^-=x+O(x^{2}),
\end{equation}
and
$$
F_\ell(x) 
= \frac{\ell}{(a+b\sqrt{g})^{1/2}}F_1(y^+) 
=\frac{1}{(a-b\sqrt{g})^{1/2}}F_1(y^-).
$$
Moreover, the following identity holds:
$$
\frac{c^2-d^2g}{4x^{\ell-1}} = \left(\frac{a^2-b^2g}{\ell^2}\right)^3 = K^3.
$$
\end{theorem}
\begin{proof}
The results for $\ell=3$, $7$ and $11$ were proved in~\cite[Th. 5.1]{coopergeye} and the
identities for $\ell=23$ can be obtained from~\cite[Th. 5.1]{coopergeye} by change
of variable. 
The result for $\ell=2$ is equivalent to the
identities~(3.1), (3.2) and (3.9) of~\cite{hme} (see also \cite[(13.21), (13.22), (13.29)]{cooperbook})
by taking
$$
x= \frac{p(1-p)^3(1-4p)^6(1-2p)(1+2p)^3}{(1+20p-48p^2+32p^3-32p^4)^4}
$$
and noting that
$$
\sqrt{g} = \sqrt{1-256x} = \frac{(1+8p^2-32p^3+32p^4)(1-88p+168p^2+32p^3-32p^4)}
{(1+20p-48p^2+32p^3-32p^4)^2}.
$$
In a similar way 
the result for $\ell=5$ is equivalent
to~\cite[(5.1), (5.3), (5.9)]{hme} or \cite[(13.3), (13.5), (13.11)]{cooperbook}.
The result for $\ell=13$ follows by change of variable from the formulas in~\cite{cooperye13}.
\end{proof}
\begin{table}
\caption{Data for Theorem~\ref{Th:23571123}}
\label{T:23571123}
{\renewcommand{\arraystretch}{1.05}
{\begin{tabular}{|c|l|}
\hline
$\ell=2$ & $\begin{array}{l}
a=5/2 \\
b=3/2 \\
c=1-207x+3456x^2 \\
d=1-81x \\
g=1-256x \\ 
K=1+144x 
\end{array}$\\
\hline
$\ell=3$ & $\begin{array}{l}
a=5 \\
b=4 \\
c=1-126x+2944x^2 \\
d=(1-8x)(1-64x) \\
g=1-108x \\ 
K=1+192x\end{array}$  \\
\hline
$\ell=5$ & $\begin{array}{l}
a=13-36x \\
b=12 \\
c=1-80x+1890x^2-12600x^3+7776x^4+3456x^5 \\
d=(1-4x)(1-18x)(1-36x) \\
g=1-44x-16x^2  \\ 
K=1+216x+144x^2\end{array}$  \\
\hline
$\ell=7$ & $\begin{array}{l}
a=25-80x \\
b=24 \\
c=(1-21x+8x^2)(1-42x+454x^2-1008x^3-1280x^4) \\
d=(1-2x)(1-8x)(1-24x)(1-16x-8x^2) \\
g=(1+x)(1-27x)  \\ 
K=1+224x+448x^2\end{array}$  \\
\hline
$\ell=11$ & $\begin{array}{l}
a=61-368x + 352x^2 \\
b=60 \\
c=1-66x + 1793x^2 - 26048x^3 + 221056x^4 - 1132670x^5 + 3535840x^6 \\
\qquad - 6683072x^7+ 7418752x^8 - 4460544x^9 + 1183744x^{10} - 65536x^{11}\\
d=(1-x)(1-2x)(1-4x)(1-7x)(1-16x)(1-12x+16x^2)(1-14x+4x^2) \\
g=1-20x+56x^2-44x^3  \\ 
K=1+224x-192x^2-832x^3+1024x^4\end{array}$  \\
\hline
$\ell=13$ & $\begin{array}{l}
a=85-135x-280x^2 \\
b=84+48x \\
c=1 - 39x + 546x^2 - 2938x^3 + 117x^4+ 42861x^5- 36998x^6- 255294x^7 \\
\qquad + 15366x^8+ 519610x^9 + 342576x^{10} - 14976x^{11}-35840x^{12}\\
d=(1+2x)(1+x)(1-x)(1-3x)(1-8x)(1-11x-8x^2)(1-4x-8x^2)(1-10x-2x^2) \\
g=1-10x-27x^2  \\ 
K=1+234x+1417x^2+1872x^3+832x^4\end{array}$  \\
\hline
$\ell=23$ & $\begin{array}{l}
a= 265 - 940x + 310x^2 - 60x^3 - 135x^4\\
b= 264 - 120x \\
c= 1 - 46x + 920x^2 - 10465x^3 + 74221x^4 - 336927x^5 + 953856x^6 -1470068x^7 \\
\qquad + \, 336674x^8 + 3209696x^9 - 6447728x^{10} +5423124x^{11} + 302266x^{12} \\
\qquad - \,   6612454x^{13}
+8362616x^{14}-4877702x^{15} - 116403x^{16}
+2732814x^{17} \\
\qquad - \, 2492280x^{18} + 1115109x^{19}- 170775x^{20} - 83835x^{21} + 48600x^{22}
- 6750x^{23} \\
d= (1 + x)(1 - x)(1 - 2x)(1 - 3x)(1 - 5x)(1 - 8x + 3x^2 )(1 - 6x - 9x^2 ) \\
\qquad \times (1 - 7x + 3x^2 - 5x^3 ) (1 - 7x + 7x^2 - 3x^3 )(1 - 4x - x^4 ) \\
g = (1 - x^2 + x^3)(1 - 8x + 3x^2 - 7x^3)  \\ 
K=1 + 232x +732x^2 - 968x^3 + 1894x^4- 2280x^5 + 2268x^6 - 1080x^7 + 225x^8
\end{array}$  \\
\hline
\end{tabular}}}
\end{table}

For $\ell=7$,~\eqref{E:n0} gives
$$
y^+ = x^7 + O(x^8)\quad\text{and}\quad y^- = x + O(x^2)\quad \text{as} \quad x\rightarrow 0.
$$
Thus, for the value $x=-1/22^3$ we have
$$
|1728 \, y^+| \approx \frac{1728}{22^{21}} \approx 1.1 \times 10^{-25} \quad \text{and}\quad |1728\,y^-| \approx \frac{1728}{22^{3}} \approx 0.16 
$$
This accounts for an observation of the Borweins~\cite[p. 360]{borweinrr} that the series for $N=427$ in Table~\ref{T:1.2} adds roughly 25 digits per term, while 
convergence for the conjugate series is ``much slower --- less than one digit per term''. 

The identities in Theorem~\ref{Th:23571123} can be used to show that all but three of the
quadratic irrational level 1 series in Tables~\ref{T:1.1} and~\ref{T:1.2} are equivalent to
rational series from other levels. The precise details are as follows.

\begin{theorem}
\label{Th:Level1}
All of the positive level~1 quadratic irrational
series in Table~\ref{T:1.1} except for $N=8$ and $N=16$, and all of 
the negative level~1 quadratic irrational series in Table~\ref{T:1.2} except for $N=4$, are
equivalent to
rational series from levels 2, 3, 5, 7, 11 or 13 by the transformation formulas in
Theorem~\ref{Th:23571123}.
\end{theorem}
\begin{proof}
We first analyse Table~\ref{T:1.1}.
The quadratic irrational series in Table~\ref{T:1.1} corresponding to $N=6$, 10, 18, 22 and 58
and their conjugates are obtained using~$\ell=2$ in Theorem~\ref{Th:23571123}
and the positive level~2 series in Table~\ref{T:1.1} corresponding to $N=3$, 5, 9, 11 and 29, respectively.

The series corresponding to $N=5$, 9, 13, 25 and 37 in Table~\ref{T:1.1} 
are obtained using~$\ell=2$ in Theorem~\ref{Th:23571123} and the
negative level~2 series in Table~\ref{T:1.2} corresponding to the same values of~$N$. 

The series corresponding to $N=12$ and 15 and the conjugate series in Table~\ref{T:1.1} 
are obtained using~$\ell=3$ in Theorem~\ref{Th:23571123} and the
positive level~3 series in Table~\ref{T:3.1} corresponding to $N=4$ and 5, respectively. 

Finally, the series corresponding
to $N=28$ and its conjugate are obtained using~$\ell=7$ in Theorem~\ref{Th:23571123} and the
positive level~7 series in Table~\ref{T:7.1} corresponding to $N=4$.

\medskip

Our analysis for Table~\ref{T:1.2} will be more brief, and
on each line we just state the values of~$N$ in Table~\ref{T:1.2} followed by the associated level
and respective values of $N$.

$N=5$, 9, 13, 25: $\ell=2$, Table~\ref{T:2.2}, use the same respective value of $N$;

$N=15$, 51, 75, 123, 147, 267: $\ell=3$, Table~\ref{T:3.2}, use $N/3$;

$N=35$, 115, 235: $\ell=5$, Table~\ref{T:5.2}, use $N/5$;

$N=91$, 427: $\ell=7$, Table~\ref{T:7.1}, use $N/7$;

$N=99$, 187: $\ell=11$, Table~\ref{T:11.1}, use $N/11$;

$N=403$: $\ell=13$, Table~\ref{T:13.1}, use $N/13$.

\medskip

Sometimes, it is possible to use more than one transformation formula.
For example, the series corresponding to $N=115$ in Table~\ref{T:1.2} can be obtained
either by using $\ell=5$ and $N=23$ in Table~\ref{T:7.1}, or by using $\ell=23$ and $N=5$
in Table~\ref{T:23.1}.
\end{proof}

\section{Other levels}
\label{S:level2}
The ideas in the previous section can be used to prove other results, such as the following theorem
that connects quadratic irrational level~3 series with level~6 series.
\begin{theorem}
\label{T:levels3and6}
Let
$$
y^{+} = \frac{2x^2}{c + d\sqrt{g}}\quad\text{and}\quad
y^{-} = \frac{2x^2}{c - d\sqrt{g}}
$$
where
$$
c=1-19x+16x^2,\qquad d=1-5x\quad\text{and}\quad g=1-32x,
$$
and let
$$
a=\frac52-8x \quad \text{and}\quad b=\frac32.
$$
Let $F_\ell(x)$ be as in~\eqref{E:Fdef}.
Then
$$
\left(\frac{c^2-d^2g}{4x}\right)^{1/3} = \left(\frac{a^2-b^2g}{4}\right)^{1/2} = 1+4x,
$$
$$
y^+ = x^2+O(x^3) \quad\text{and}\quad y^- = x+O(x^2) \quad \text{as}\quad x\rightarrow 0,
$$
and in a neighbourhood\ of $x=0$ we have
$$
F_6(x) = \frac{2}{(a+b\sqrt{g})^{1/2}}F_3(y^+) = \frac{1}{(a-b\sqrt{g})^{1/2}}F_3(y^-).
$$
\end{theorem}
\begin{proof}
This is immediate from~\cite[(3.15), (3.16) and (3.31)]{hme}.
\end{proof}

A consequence of Theorem~\ref{T:levels3and6} is that the positive rational level 6 series
for $N=3$, 5, 7, 13 and 17 in Table~\ref{T:6.1} are equivalent to the quadratic irrational level~3
series corresponding to $N=6$, 10, 14, 26 and 34 and their respective conjugate series in Table~\ref{T:3.1}.
Moreover, when $x=1/50$ we have $y^+=1/1458$ and $y^-=1/108$, so the level~6 series
for $N=2$ in Table~\ref{T:6.1} is equivalent to the series for $N=4$ (convergent) and $N=1$ (undefined)
in Table~\ref{T:3.1}.

Theorem~\ref{T:levels3and6} can also be used to show that the negative rational level 6 series
for $N=3$, 5, 7, 11, 19, 31 and 59 in Table~\ref{T:6.2} are equivalent to the
level~3 series in Table~\ref{T:3.1} and their conjugates in Table~\ref{T:3.2}
for the respective values of~$N$. For example, 
if we let $x=-1/(2^4\cdot5^2\cdot53^2)$ corresponding to $N=59$ in Table~\ref{T:6.2},
then we get $y^\pm = 1/(2^3\cdot3^3\cdot(892\pm525\sqrt{3})^3)$ and these correspond to the values
of~$X_3$ for~$N=59$
in Tables~\ref{T:3.1} and~\ref{T:3.2}.


Transformation formulas for other levels are often best expressed in terms of two variables.
As an example, consider levels~3 and~5.
Let $x=\eta_{3}^2\eta_{15}^2/\eta_{1}^2\eta_{5}^2$ and $y=\eta_{1}^3\eta_{15}^3/\eta_{3}^3\eta_{5}^3$.
Then
$$
9x^2y + xy^2 + 5xy - x + y=0,
$$
$$
w_{3}=\frac{\eta_{3}^{12}}{\eta_{1}^{12}}=\frac{x(1+5x+x y+10x^2)}{1-5y-9xy},
$$
and
$$
w_{5}=\frac{\eta_{5}^6}{\eta_{1}^6}=\frac{27x-2y+9xy+81x^2}{25(1-9xy)}.
$$
Moreover,
$$
\frac{Z_{3}}{Z_{5}}=\frac{1+10y+36xy+5y^2}{1+4y+18xy-y^2}.
$$


\section{Generalizations of hypergeometric transformation formulas}
\label{S:general}
The identities
\begin{align}
    \label{E:3F2transformations}
    \pFq{3}{2}{\frac12,\frac12,\frac12}{1,1}{x}
    &=\frac{2}{\sqrt{4-x}} \pFq{3}{2}{\frac16,\frac12,\frac56}{1,1}{\frac{27x^2}{(4-x)^3}}\\
    &=\frac{1}{\sqrt{1-4x}} \pFq{3}{2}{\frac16,\frac12,\frac56}{1,1}{\frac{-27x}{(1-4x)^3}}
    \nonumber
\end{align}
are part of a family of identities considered by the 
Borweins~\cite[pp. 180--181]{borwein-book} 
and by N. D. Baruah and B. C. Berndt~\cite{BB2010} and that were
subsequently generalized in a different way in~\cite{hme} (see
also~\cite[pp. 600--605]{cooperbook}). Another generalization of these identities is as follows.
 
\begin{theorem}[Levels $\ell$ and $4\ell$]
\label{T:levels1and4}
The pairs of identities
$$
X_{\ell}(q^2) = \Phi(X_{4\ell}(q)),
\qquad
\Phi(x) F_{\ell}(\Phi(x))^a = x^2F_{4\ell}(x)^a$$
and
$$
X_{\ell}(-q) = \Psi(X_{4\ell}(q)),
\qquad
\Psi(x) F_{\ell}(\Psi(x))^a = -xF_{4\ell}(x)^a$$
hold, for the following values of $\ell$, $a$, $\Phi$ and $\Psi$:

\begin{table}[H]
{\renewcommand{\arraystretch}{1.8}
{\begin{tabular}{|c|c|c|c|}
\hline
$\ell$ & $a$ & $\Phi(x)$ & $\Psi(x)$ \\
\hline
$1$ & $6$ & $\frac{x^2}{(1-16x)^3}$ & $\frac{-x}{(1-256x)^3}$\\
\hline
$2$ & $4$ & $\frac{x^2}{(1-4x)^4}$ & \\
\hline
$3$ & $3$ & $\frac{x^2}{(1-4x)^3}$ & $\frac{-x}{(1-16x)^3}$\\
\hline
$4$ & $2$ & $\frac{x^2}{(1-2x)^4}$ & \\
\hline
$5$ & $2$ & $\frac{x^2}{(1-4x)(1-2x)^2}$ & $\frac{-x}{(1-4x)(1-8x)^2}$\\
\hline
$6$ & $2$ & $\frac{x^2}{(1+4x^2)(1-2x)^2}$ & \\
\hline
$7$ & $3/2$ & $\frac{x^2}{(1-2x)^3}$ & $\frac{-x}{(1-4x)^3}$\\
\hline
$8$ & $1$ & $\frac{x^2}{(1-2x+2x^2)^2}$ & \\
\hline
$11$ & $1$ & $\frac{x^2}{1-4x+8x^2-4x^3}$ & $\frac{-x}{1-8x+16x^2-16x^3}$\\
\hline
$14$ & $1$ & $\frac{x^2}{(1-x+2x^2)^2}$ & \\
\hline
$15$ & $1$ & $\frac{x^2}{(1-x)(1-x+4x^2)}$ & $\frac{-x}{(1-4x)(1-x+4x^2)}$\\
\hline
$23$ & $1/2$ & $\frac{x^2}{(1-x)(1-x+2x^2)}$ & $\frac{-x}{(1-2x)(1-x+2x^2)}$\\
\hline
\end{tabular}}}
\end{table}

In addition, for $\ell=9$ we have the pairs of formulas
$$
X_9(q^2) = X_{36}(q)^2\left(\frac{1+X_{36}(q)}{1-3X_{36}(q)}\right), \quad
\frac{1}{1-3x} F_9\left(x^2\left(\frac{1+x}{1-3x}\right)\right) = F_{36}(x)
$$
and
$$
X_{9}(-q) = -X_{36}(q)\left(\frac{1+X_{36}(q)}{1-3X_{36}(q)}\right)^2, \quad
\frac{1}{(1-3x)^2} F_9\left(-x\left(\frac{1+x}{1-3x}\right)^2\right) = F_{36}(x).
$$
\end{theorem}

\begin{proof}
The identities in Theorem~\ref{T:levels1and4} that involve $F_\ell$ can be
proved by showing that each side
satisfies the same third order linear differential equation 
by starting with~\eqref{E:nonlinear} and applying a change of variables,
and checking that both sides satisfy the same initial conditions. This is a routine procedure
that can be done with a computer, e.g., see~\cite[p. 350]{hme}. Similarly, the theory of
modular forms can be used to verify the identities that involve $X_\ell$. We omit the details.
\end{proof}

When $\ell=1$, the identities in Theorem~\ref{T:levels1and4} that involve $F_\ell$ are
\begin{equation}
\label{E:1t}
\frac{1}{(1-16x)^{1/2}}F_1\left(\frac{x^2}{(1-16x)^3}\right)
= F_4(x)
= \frac{1}{(1-256x)^{1/2}}F_1\left(\frac{-x}{(1-256x)^3}\right).
\end{equation}
If we replace $x$ with $x/64$, then these become the identities in~\eqref{E:3F2transformations}.
The corresponding identities involving $X_\ell$ are
$$
X_1(q^2) = \frac{X_4(q)^2}{(1-16X_4(q))^3}
\quad\text{and}\quad
X_1(-q) = \frac{-X_4(q)}{(1-256X_4(q))^3}
$$
and these can be verified by the theory of theta functions.

The next result is a cubic analogue of Theorem~\ref{T:levels1and4}.

\begin{theorem}[Levels $\ell$ and $9\ell$]
\label{T:levels1and9}
The identities
$$
X_{\ell}(q^3) = \Omega(X_{9\ell}(q))
\quad\text{and}\quad
\omega(x) F_{\ell}(\Omega(x)) = F_{9\ell}(x)$$
hold for the following values of $\ell$, $\Omega$ and $\omega$:

\begin{table}[H]
{\renewcommand{\arraystretch}{1.8}
{\begin{tabular}{|c|c|c|}
\hline
$\ell$ & $\Omega(x)$ & $\omega(x)$ \\
\hline
$1$ & $\frac{x^3}{(1-9x)^3(1+3x)^3}$ & $\frac{1}{(1-9x)^{1/2}(1+3x)^{1/2}}$\\ \hline
$2$ & $\frac{x^3(1+x)^3}{(1+4x)^2(1-2x)^4}$ & $\frac{1}{(1-2x)(1+4x)^{1/2}}$\\ \hline
$3$ & $\frac{x^3(1-3x+3x^2)^3}{(1-3x)^6}$ & $\frac{1}{(1-3x)^2}$\\ \hline
$4$ & $\frac{x^3}{(1-3x)^3(1+x)^3}$ & $\frac{1}{(1-3x)^{3/2}(1+x)^{3/2}}$\\ \hline
$5$ & $\frac{x^3}{(1-3x^2)^2(1-3x+3x^2)}$ & $\frac{1}{(1-3x^2)(1-3x+3x^2)^{1/2}}$ \\
\hline
\end{tabular}}}
\end{table}
\end{theorem}

As a consequence of Theorem~\ref{T:levels1and4}, for $\ell\in\left\{1,2,3,4,5,6,7,8,9,11,14,15,23\right\}$,
all of the rational and quadratic irrational Ramanujan type series for level~$4\ell$ can be deduced from the
corresponding series for level~$\ell$. For example, all level~4 series in Table~\ref{T:4.1}
can be deduced from the corresponding level~$1$ series in Table~\ref{T:1.1} by
the formulas in~\eqref{E:1t}.

Similarly, for $\ell\in \left\{1,2,3,4,5\right\}$ all rational and quadratic irrational
Ramanujan type series for level~$9\ell$, with three exceptions, can
be deduced from the corresponding series for level~$\ell$ by the identities in
Theorem~\ref{T:levels1and9};
the exceptions are the negative level 9 series corresponding to $N=1$, $2$ and $3$ in Table~\ref{T:9.1}.

Let $w=\eta_1^2\eta_2^2\eta_9^2\eta_{18}^2/\eta_3^4\eta_6^4$. The level 18 modular forms in
Tables~\ref{T:00} and~\ref{T:Z} are
$$
X_{18} = \frac{w}{1-w} \quad\text{and}\quad
Z_{18}=\frac{18P_{18}-9P_9-12P_6+6P_3+2P_2-P_1}{4(1+4X_{18})^{3/2}}
$$
and the associated polynomials in Table~\ref{T:GH} are
$$
G = (1+x)^2(1+4x)(1-8x)\quad\text{and}\quad H=24x^2(1+x)(2+5x).
$$
On the other hand, the level 18 modular functions in~\cite{cooperlevel7} are
\begin{equation}
\label{E:18ZX}
\tilde{X} = \frac{w}{(1+3w)^2}=\frac{X_{18}(1+X_{18})}{(1+4X_{18})^2}\quad\text{and}\quad
\tilde{Z} = (1+4X_{18})^{3/2}Z_{18}
\end{equation}
with associated polynomials
$$
\tilde{G} = (1-12x)(1-16x)\quad\text{and}\quad \tilde{H}=12x(1-15x).
$$
By the change of variables in~\eqref{E:18ZX} the data in Table~\ref{T:18.1} for
$N=2$, 5, 11, 29, $(X_{18}>0)$ and $N=13$, 25, 37, $(X_{18}<0)$ is equivalent to the data
in~\cite[Table 2]{cooperlevel7}.
Furthermore, by Theorem~\ref{T:levels1and9} in the case $\ell=2$ we have
$$
Z_{18} = \frac{1}{(1-2X_{18})(1+4X_{18})^{1/2}}
F_2\left(\frac{X_{18}^3(1+X_{18})^3}{(1+4X_{18})^2(1-2X_{18})^4}\right)
$$
and combining this with~\eqref{E:18ZX} we deduce
$$
\sum_{n=0}^\infty \tilde{T}(n) \frac{x^n(1+x)^n}{(1+4x)^{2n}}
= \frac{1+4x}{1-2x} \pFq{3}{2}{\frac14,\frac12,\frac34}{1,1}{\frac{256x^3(1+x)^3}{(1+4x)^2(1-2x)^4}}
$$
where $\tilde{T}(n)$ are the coefficients in the expansion $\tilde{Z}=\sum_{n=0}^\infty \tilde{T}(n)\tilde{X}^n$, so that by Theorem~\ref{T:2} applied to $\tilde{G}$ and $\tilde{H}$ we
have (see~\cite[(21)]{cooperlevel7} or \cite[p. 422]{cooperbook}),
$$
(n+1)^3\tilde{T}(n+1)=2(2n+1)(7n^2+7n+3)\tilde{T}(n)-12n(16n^2-1)\tilde{T}(n-1).
$$

The level 13, 16, 17 and 23 series in~\cite{cooperye13}, \cite{ye16}, \cite{huber17}
and~\cite{coopergeye}, respectively, can be
related to the data in Tables~\ref{T:13.1}, \ref{T:16.1}, \ref{T:17.1} and~\ref{T:23.1} by 
similar techniques.

\section{Some remarks}
\label{S:remarks}

\subsection{The value $X_{\ell}(e^{-2\pi \sqrt{N/\ell}})$}
For a given $\ell$, it may be useful to note how to identify all values $X_{\ell}(e^{-2\pi \sqrt{N/\ell}})$ that are at most quadratic irrationals. We briefly illustrate this in the case where $\ell$ is odd and $N\ell$ is
square-free.
Write $\tau_{0}=\frac{\sqrt{-4N\ell}}{2\ell}$. So, the $\mathbb{Z}$-module $[\ell,\ell\tau_{0}]$ is a proper ideal of the order $\mathcal{O}=\mathbb{Z}[\sqrt{-4N\ell}]$. By assumption, $\mathcal{O}$ is either a maximal order or an index-2 order. Either way, $\ell\mathcal{O}$ is ramified. So, suppose that $\ell\mathcal{O}=\mathfrak{l}^{2}$ with $\mathfrak{l}=[\ell,\ell\tau_{0}]$.

Recall that the moduli interpretation of $\Gamma_{0}(\ell)\backslash\mathbb{H}$ is given by
$$
[\tau]\to [\mathbb{C}/[1,\ell\tau], \mathbb{C}/[1,\tau]].
$$
Then the CM point $[\tau_{0}]$ corresponds to the isomorphism class $[\mathbb{C}/\mathcal{O},\mathbb{C}/\mathfrak{l}^{-1}]$. The $\mathbb{Q}[\sqrt{-N\ell}]$-Galois orbit of $[\mathbb{C}/\mathcal{O},\mathbb{C}/\mathfrak{l}^{-1}]$ is exactly $\{[\mathbb{C}/\mathfrak{a},\mathbb{C}/\mathfrak{l}^{-1}\mathfrak{a}]\}$, where $[\mathfrak{a}]$ runs over the ideal class group ${\rm Cl}(\mathcal{O})$ of $\mathcal{O}$, and thus, the $\mathbb{Q}[\sqrt{-N\ell}]$-Galois orbit of $[\tau_{0}]$ is exactly $\{[\tau_{\mathfrak{l}\mathfrak{a}}]\}$, where $\tau_{\mathfrak{a}}$ denotes the CM point induced by an ideal $\mathfrak{a}$.
Moreover,  since $\mathfrak{l}=\overline{\mathfrak{l}}$, this is actually a $\mathbb{Q}$-Galois orbit in $\Gamma_{0}(\ell)\backslash\mathbb{H}$. Under the  rational morphism $X(\Gamma_{0}(\ell))\to X(\Gamma_{0}(\ell)+)$, this descends to be a $\mathbb{Q}$-Galois orbit in $\Gamma_{0}(\ell)+\backslash\mathbb{H}$.
In addition, since $\begin{pmatrix}
    0&-1\\\ell&0
\end{pmatrix}\in\Gamma_{0}(\ell)+$, then $[-\frac{1}{\ell\tau_{0}}]$ is equivalent to $[\tau_{0}]$. Notice that $-\frac{1}{\ell\tau_{0}}=\tau_{\mathfrak{n}\mathfrak{l}}$, where $\mathfrak{n}$ is the ideal above $N$ that must be different from $\mathfrak{l}$, since $N\ell$ is square free. So, the cardinality of the $\mathbb{Q}$-Galois orbit of $[\tau_{0}]$ in $\Gamma_{0}(\ell)+\backslash\mathbb{H}$ is at most $|{\rm Cl}(\mathcal{O})|/2$. Therefore, $X_{\ell}(e^{-2\pi \sqrt{N/\ell}})$ must be an algebraic number of degree at most~$|{\rm Cl}(\mathcal{O})|/2$. So, to identify all the $N$'s such that the value $X_{\ell}(e^{-2\pi \sqrt{N/\ell}})$ is at most quadratic irrational, one only needs to identify all the $N$'s such that $|{\rm Cl}(\mathcal{O})|/2\leq 2$, i.e., $|{\rm Cl}(\mathcal{O})|\leq 4$, where $\mathcal{O}=\mathbb{Z}[\sqrt{-4N\ell}]$. For example, when $\ell=7$, the largest order discriminant such that $|{\rm Cl}(\mathcal{O})|\leq 4$ is $-532$, which corresponds to $N=19$ in Table~\ref{T:7.1}.


\subsection{$X_{1}$ as a cube}
A close look at the values of $X_{1}$ in Tables~\ref{T:1.1} and~\ref{T:1.2} shows that some of them
can be written as cubes of numbers that also lie in $\mathbb{Q}[X_{1}]$. 
This follows from a classical result due to Schertz stating that for an imaginary quadratic point~$\tau_{0}$, when $[1,\tau_{0}]$ is a fractional ideal of some quadratic order with discriminant \emph{not} divisible by~$3$,   $X_{1}(\tau_{0})^{\frac{1}{3}}$ generates the same number field over~$\mathbb{Q}$ as $X_{1}(\tau_{0})$. More generally, Schertz \cite{S76} proved that:
\begin{theorem}[Schertz]\label{scherthm}
    Let 
    $$
    \tau_{0}=\begin{cases}
        \frac{\sqrt{-d}}{2}&\mbox{if $-d\equiv0\pmod{4}$,}\\
        \frac{3+\sqrt{-d}}{2}&\mbox{if $-d\equiv1\pmod{4}$}
    \end{cases}
    $$
    be such that $[1,\tau_{0}]$ is a quadratic order of discriminant~$-d$. Then
    \begin{enumerate}
        \item if $3\nmid d$, one has $\mathbb{Q}[X_{1}(\tau_{0})^{\frac{1}{3}}]=\mathbb{Q}[X_{1}(\tau_{0})]$,

        \item if $3\mid d$, one has $[\mathbb{Q}[X_{1}(\tau_{0})^{\frac{1}{3}}]:\mathbb{Q}[X_{1}(\tau_{0})]]=3$. 
    \end{enumerate}
\end{theorem}

To see how this plays a role, for example, note that in Table~\ref{T:1.1}, for a given positive integer~$N$, $\tau_{0}=\sqrt{-N}$, and thus, $[1,\tau_{0}]$ is the quadratic order of discriminant~$-4N$, and thus, by Theorem~\ref{scherthm}(1), when $3\nmid N$, $X_{1}(\tau_{0})^{\frac{1}{3}}\in \mathbb{Q}[X_{1}(\tau_{0})]$. On the other hand, for the $\tau_{0}$'s considered in Table~\ref{T:1.1} such that the discriminant $-4N$ of $[1,\tau_{0}]$ is divisible by~3, by Theorem~\ref{scherthm}(2), $\mathbb{Q}[X_{1}(\tau_{0})^{\frac{1}{3}}]$ must be a cubic extension of $\mathbb{Q}[X_{1}(\tau_{0})]$. This explains why these $X_{1}(\tau_{0})$ are not written as a cube over $\mathbb{Q}[X_{1}(\tau_{0})]$ in the table.

\vfill
\pagebreak[4]

\section{Tables}
\label{S:tables}
\setlength{\tabcolsep}{10pt} 
\renewcommand{\arraystretch}{2.65} 

\begin{longtable}{ | c | l l |}
   \caption{Hauptmoduln $X_\ell$ for $\Gamma_0(\ell)+$, for Theorem~\ref{T:1}\\}
   \label{T:00}\\
      \hline
      $\ell$ & $X=X_\ell$ & \\
      \hline
      $1$
      & $\displaystyle{X_{1} = \frac{\eta_1^{24}}{Q^3(q)}}$,
      & $\displaystyle{Q(q)=1+240\sum_{j=1}^\infty \frac{j^3q^j}{1-q^j}}$
      \\
      \hline
      $2$
      & $\displaystyle{X_{2} = \frac{w}{(1+64w)^2}}$,
      & $\displaystyle{w=\frac{\eta_2^{24}}{\eta_1^{24}}}$
      \\
      \hline
      $3$
      & $\displaystyle{X_{3} = \frac{w}{(1+27w)^2}
      =\left(\frac{\eta_1\eta_3}{\sigma_3}\right)^{6}}$,
      & $\displaystyle{w=\frac{\eta_3^{12}}{\eta_1^{12}},\qquad
      \sigma_3=\sum_{j,k} q^{j^2+jk+k^2}}$
      \\
      \hline
      $4$
      & $\displaystyle{X_{4} = \frac{w}{(1+16w)^2} 
      = \frac{\eta_1^{24}\eta_4^{24}}{\eta_2^{48}}
      =\frac{\eta_2^{12}}{\phi^{12}(q)}}$,
      & $\displaystyle{w=\frac{\eta_4^8}{\eta_1^8},\qquad\phi(q)=\sum_j q^{j^2}}$
      \\
      \hline
      $5$
      & $\displaystyle{X_{5} = \frac{w}{1+22w+125w^2}}$,
      & $\displaystyle{w=\frac{\eta_5^6}{\eta_1^6}}$
      \\
      \hline
      $6$
      & $\displaystyle{X_{6} = \frac{w}{(1-w)^2}}$,
      & $\displaystyle{w=\frac{\eta_{1}^{12}\eta_6^{12}}{\eta_2^{12}\eta_3^{12}}}$
      \\
      \hline
      $7$
      & $\displaystyle{X_{7} = \frac{w}{1+13w+49w^2}} 
      = \left(\frac{\eta_1\eta_7}{\sigma_7}\right)^3$,
      & $\displaystyle{w=\frac{\eta_7^4}{\eta_1^{4}},\qquad
      \sigma_{7}=\sum_{j,k}q^{j^2+jk+2k^2}}$
      \\
      \hline
      $8$
      & $\displaystyle{X_{8} = \frac{\eta_{1}^8\eta_8^8}{\eta_2^{8}\eta_4^8}}$ &
      \\
      \hline
      $9$
      & $\displaystyle{X_{9} = \frac{\eta_{1}^6\eta_9^6}{\eta_3^{12}}}$ &
      \\
      \hline
      $10$
      & $\displaystyle{X_{10} = \frac{w}{(1+4w)^2}}$,
      & $\displaystyle{w=\frac{\eta_2^{4}\eta_{10}^4}{\eta_1^{4}\eta_5^4}}$
       \\
      \hline
      $11$ &
      $\displaystyle{X_{11}=\left(\frac{\eta_1\eta_{11}}{\sigma_{11}}\right)^2},$
      & $\displaystyle{\sigma_{11}=\sum_{j,k}q^{j^2+jk+3k^2}}$ 
       \\\hline
      $12$
      & $\displaystyle{X_{12} = \frac{w}{(1+w)^2}}$,
      & $\displaystyle{w=\frac{\eta_1^{4}\eta_{12}^4}{\eta_3^{4}\eta_4^4}}$
       \\
      \hline
      $13$
      & $\displaystyle{X_{13} = \frac{w}{1+5w+13w^2}}$,
      & $\displaystyle{w=\frac{\eta_{13}^2}{\eta_1^{2}}}$
      \\
      \hline
      $14$
      & $\displaystyle{X_{14} = \frac{w}{(1+w)(1+8w)}}$,
      & $\displaystyle{w=\frac{\eta_2^{3}\eta_{14}^3}{\eta_1^{3}\eta_7^3}}$
       \\
      \hline
      $15$
      & $\displaystyle{X_{15} = \frac{w}{(1+3w)^2}}$,
      & $\displaystyle{w=\frac{\eta_3^{2}\eta_{15}^2}{\eta_1^{2}\eta_5^2}}$
       \\
      \hline
      $16$ &
      $\displaystyle{
      X_{16} = \frac{\eta_1^4\eta_4^4\eta_{16}^4}{\eta_2^6\eta_8^6}}$
      &  \\
      \hline
      $17$ &
      $\displaystyle{X_{17} = \left(\frac{2\eta_1\eta_{17}}{\sigma_{17}}\right)^2}$,
      &$\displaystyle{\sigma_{17}=
      \sum_{j\,\text{odd}}\sum_k \left\{q^{\frac14(j^2+17k^2)}-q^{\frac14(17j^2+k^2)}\right\}}$ \\
      \hline
      $18$ &
      $\displaystyle{X_{18} = \frac{w}{1-w}}$,
      & $\displaystyle{w=\frac{\eta_1^2\eta_{2}^2\eta_9^2\eta_{18}^2}{\eta_3^{4}\eta_6^4}}$ 
       \\
      \hline
      $19$ &
      $\displaystyle{X_{19}=
      \left(\frac{2\eta_1\eta_{19}}{\sigma_{19A}-\sigma_{19B}}\right)^3}$,
      & $\displaystyle{\sigma_{19A}=\sum_{j,k}q^{\frac12(j^2+2jk+20k^2)},\;
      \sigma_{19B}=\sum_{j,k}q^{\frac12(4j^2+2jk+5k^2)}}$  
      \\
      \hline
      $20$ &
      $\displaystyle{X_{20} = \frac{w}{(1+w)^2}}$,
      &$\displaystyle{w=\frac{\eta_1^2\eta_{20}^2}{\eta_4^{2}\eta_5^2}}$  
       \\
      \hline
      $21$ &
      $\displaystyle{X_{21} = \frac{w}{(1-w)^2}}$,
      &$\displaystyle{w=\frac{\eta_1^2\eta_{21}^2}{\eta_3^{2}\eta_7^2}}$
     \\
      \hline
      $22$ & 
      $\displaystyle{X_{22} = \frac{w}{(1+2w)^2}}$,
      &$\displaystyle{w=\frac{\eta_2^2\eta_{22}^2}{\eta_1^{2}\eta_{11}^2}}$
      \\
      \hline
      $23$ &
      $\displaystyle{X_{23}=\frac{2\eta_1\eta_{23}}{\sigma_{23A}+\sigma_{23B}}},$
      & $\displaystyle{\sigma_{23A}=\sum_{j,k}q^{j^2+jk+6k^2},\;
      \sigma_{23B}=\sum_{j,k}q^{2j^2+jk+3k^2}}$ 
       \\\hline
      $24$ &
      $\displaystyle{X_{24}=\frac{\eta_1^2\eta_3^2\eta_8^2\eta_{24}^2}
      {\eta_2^2\eta_4^2\eta_6^2\eta_{12}^2}}$ &
       \\\hline
      $25$ &
      $\displaystyle{X_{25}=\frac{w}{1+2w+5w^2}},$
      &$\displaystyle{w= \frac{\eta_{25}}{\eta_1}}$
       \\\hline
      $26$ &
      $\displaystyle{X_{26}=\frac{w}{(1-w)^2}},$
      &$\displaystyle{w= \frac{\eta_1^2\eta_{26}^2}
      {\eta_2^2\eta_{13}^2}}$
       \\\hline
      $27$ &
      $\displaystyle{\frac{1}{X_{27}^2}-\frac{3}{X_{27}}=\frac{1}{w}-3,}$
      &$\displaystyle{w= \frac{\eta_{1}^3\eta_{27}^3}{\eta_3^3\eta_{9}^3}
      }$ 
       \\\hline
      $28$ &
      $\displaystyle{X_{28}=\frac{\eta_1^3\eta_4^3\eta_7^3\eta_{28}^3}
      {\eta_2^6\eta_{14}^6}}$
      & 
       \\\hline
      $29$ &
      $\displaystyle{\frac{1}{X_{29}}=\frac{\sigma_{29}}{2\eta_1\eta_{29}}+1},$
      &$\displaystyle{\sigma_{29}=
      \sum_{j\,\text{odd}}\sum_k\left\{q^{\frac14(j^2+29k^2)}-q^{\frac14(29j^2+k^2)}\right\}}$
       \\\hline  
      $30$ &
      $\displaystyle{X_{30}=\frac{w}{(1-w)^2}},$
      &$\displaystyle{w= \frac{\eta_1\eta_2\eta_{15}\eta_{30}}
      {\eta_3\eta_{5}\eta_6\eta_{10}}}$
       \\\hline
      $31$ &
      $\displaystyle{X_{31}=\left(\frac{2\eta_1\eta_{31}}{\sigma_{31A}-\sigma_{31B}}\right)^3},$
      & $\displaystyle{\sigma_{31A}=\sum_{j,k}q^{j^2+jk+8k^2},\;
      \sigma_{31B}=\sum_{j,k}q^{2j^2+jk+4k^2}}$ 
       \\\hline
      $32$ &
      $\displaystyle{X_{32}=\frac{\eta_1^2\eta_4\eta_8\eta_{32}^2}
      {\eta_2^3\eta_{16}^3}}$
      & 
       \\\hline
      $33$ &
      $\displaystyle{X_{33}=\frac{w}{1+w+3w^2}},$
      &$\displaystyle{w= \frac{\eta_3\eta_{33}}{\eta_1\eta_{11}}}$
       \\\hline
      $34$ &
      $\displaystyle{X_{34}=\left(\frac{2\eta_1\eta_{17}}{\sigma_{34}}\right)^2},$
      &$\displaystyle{\sigma_{34}=
      \sum_{j\,\text{odd}}\sum_k \left\{q^{\frac14(j^2+2j k+18k^2)}-q^{\frac14(9j^2+2j k+2k^2)}\right\}}$
       \\\hline  
      $35$ &
      $\displaystyle{X_{35}=\frac{w}{1+w-w^2}},$
      &$\displaystyle{w= \frac{\eta_1\eta_{35}}{\eta_5\eta_7}}$
       \\\hline
      $36$ &
      $\displaystyle{X_{36}=\frac{w}{1+3w}},$
      &$\displaystyle{w= \frac{\eta_2^4\eta_3^6\eta_{12}^6\eta_{18}^4}
      {\eta_1\eta_{4}\eta_6^{16}\eta_9\eta_{36}}}$
       \\\hline
      $38$ &
      $\displaystyle{\frac{1}{X_{38}^2}+\frac{1}{X_{38}}=w,}$
      &$\displaystyle{w= \frac{1}{X_{19}(q^2)}+\frac{1}{X_{19}(q)} -2
      }$ 
       \\\hline
      $39$ &
      $\displaystyle{X_{39}=\frac{w}{(1+w)^2}},$
      &$\displaystyle{w= \frac{\eta_1\eta_{39}}
      {\eta_3\eta_{13}}}$
       \\\hline
      $41$ &
      $\displaystyle{\frac{1}{X_{41}}=\frac{\sigma_{41}}{2\eta_1\eta_{41}}-1},$
      &$\displaystyle{\sigma_{41}=
      \sum_{j+k\,\text{odd}}q^{\frac14(3j^2+4jk+15k^2)}(-1)^k}$
       \\\hline  
      $42$ &
      $\displaystyle{X_{42}=\frac{w}{(1-w)^2}},$
      &$\displaystyle{w= \frac{\eta_1\eta_3\eta_{14}\eta_{42}}
      {\eta_2\eta_{6}\eta_7\eta_{21}}}$
       \\\hline
      $44$ &
      $\displaystyle{X_{44}=\frac{\eta_1^2\eta_4^2\eta_{11}^2\eta_{44}^2}
      {\eta_2^4\eta_{22}^4}}$
      & 
       \\\hline
      $45$ &
      $\displaystyle{X_{45}=\frac{\eta_1\eta_5\eta_{9}\eta_{45}}{\eta_3^2\eta_{15}^2}}$
      & 
       \\\hline 
      $46$ &
      $\displaystyle{X_{46}=\frac{w}{1+w+2w^2},}$
      &$\displaystyle{w= \frac{\eta_2\eta_{46}}{\eta_1\eta_{23}}
      }$ 
       \\\hline
      $47$ &
      $\displaystyle{X_{47}=\frac{2\eta_1\eta_{47}}{\sigma_{47A}-\sigma_{47B}}},$
      & $\displaystyle{\sigma_{47A}=\sum_{j,k}q^{j^2+jk+12k^2},\;
      \sigma_{47B}=\sum_{j,k}q^{2j^2+jk+6k^2}}$
       \\\hline
      $49$ &
      $\displaystyle{\frac{1}{X_{49}^3}-\frac{4}{X_{49}^2}+\frac{3}{X_{49}}=\frac{1}{w}-1,}$
      &$\displaystyle{w= \frac{\eta_{1}^2\eta_{49}^2}{\eta_7^4} 
      }$ 
       \\\hline
      $50$ &
      $\displaystyle{\frac{1}{X_{50}^2}+\frac{3}{X_{50}}=w,}$
      &$\displaystyle{w= \frac{1}{X_{25}(q^2)}+\frac{1}{X_{25}(q)} 
      }$ 
       \\\hline 
      $51$ &
      $\displaystyle{\frac{1}{X_{51}^3}-\frac{6}{X_{51}^2}+\frac{10}{X_{51}}=w,}$
      &$\displaystyle{w= \frac{1}{X_{17}(q^3)}+\frac{1}{X_{17}(q)}+6 
      }$ 
       \\\hline 
      $54$ &
      $\displaystyle{\frac{1}{X_{54}^2}-\frac{1}{X_{54}}=\frac{1}{w}-1,}$
      &$\displaystyle{w= \frac{\eta_1\eta_{2}\eta_{27}\eta_{54}}{\eta_3\eta_6\eta_{9}\eta_{18}}
      }$ 
       \\\hline
      $55$ &
      $\displaystyle{\frac{1}{X_{55}^5}-\frac{5}{X_{55}^4}+\frac{15}{X_{55}^2}
      +\frac{1}{X_{55}}=w,}$
      &$\displaystyle{w= \frac{1}{X_{11}(q^5)}+\frac{1}{X_{11}(q)} 
      }$ 
       \\\hline
      $56$ &
      $\displaystyle{X_{56}=\frac{\eta_1\eta_7\eta_{8}\eta_{56}}{\eta_2\eta_{4}\eta_{14}\eta_{28}}}$
      & 
       \\\hline
      $59$ &
      $\displaystyle{X_{59}=\frac{\sigma_{59A}-\sigma_{59B}}{2\sigma_{59B}}},$
      & $\displaystyle{\sigma_{59A}=\sum_{j,k}q^{j^2+jk+15k^2},\;
      \sigma_{59B}=\sum_{j,k}q^{3j^2+jk+5k^2}}$ 
       \\\hline
      $60$ &
      $\displaystyle{X_{60}=\frac{\eta_1\eta_3\eta_4\eta_5\eta_{12}\eta_{15}\eta_{20}\eta_{60}}
      {\eta_2^2\eta_{6}^2\eta_{10}^2\eta_{30}^2}}$
      & 
       \\\hline
      $62$ &
      $\displaystyle{\frac{1}{X_{62}^2}+\frac{1}{X_{62}}=w,}$
      &$\displaystyle{w= \frac{1}{X_{31}(q^2)}+\frac{1}{X_{31}(q)} +2
      }$ 
       \\\hline
      $66$ &
      $\displaystyle{\frac{1}{X_{66}^2}+\frac{1}{X_{66}}=w,}$
      &$\displaystyle{w= \frac{1}{X_{33}(q^2)}+\frac{1}{X_{33}(q)} +4
      }$ 
       \\\hline
      $69$ &
      $\displaystyle{\frac{1}{X_{69}^3}-\frac{2}{X_{69}}=w,}$
      &$\displaystyle{w= \frac{1}{X_{23}(q^3)}+\frac{1}{X_{23}(q)} -1
      }$ 
       \\\hline
      $70$ &
      $\displaystyle{X_{70}=\frac{w}{(1+w)^2}},$
      & $\displaystyle{w=\frac{\eta_2\eta_5\eta_{7}\eta_{70}}{\eta_1\eta_{10}\eta_{14}\eta_{35}}}$
       \\\hline
      $71$ &
      $\displaystyle{X_{71}=\frac{2\eta_1\eta_{71}}{\sigma_{71A}-\sigma_{71B}}},$
      & $\displaystyle{\sigma_{71A}=\sum_{j,k}q^{2j^2+jk+9k^2},\;
      \sigma_{71B}=\sum_{j,k}q^{3j^2+jk+6k^2}}$
       \\\hline
      $78$ &
      $\displaystyle{X_{78}=\frac{w}{(1+w)^2}},$
      & $\displaystyle{w=\frac{\eta_2\eta_3\eta_{13}\eta_{78}}{\eta_1\eta_6\eta_{26}\eta_{39}}}$
       \\\hline
      $87$ &
      $\displaystyle{\frac{1}{X_{87}^3}+\frac{1}{X_{87}}=w,}$
      &$\displaystyle{w= \frac{1}{X_{29}(q^3)}+\frac{1}{X_{29}(q)} -1
      }$ 
       \\\hline
      $92$ &
      $\displaystyle{X_{92}=\frac{\eta_1\eta_4\eta_{23}\eta_{92}}{\eta_2^2\eta_{46}^2}}$
      & 
       \\\hline
      $94$ &
      $\displaystyle{\frac{1}{X_{94}^2}+\frac{1}{X_{94}}=w,}$
      &$\displaystyle{w= \frac{1}{X_{47}(q^2)}+\frac{1}{X_{47}(q)} +2
      }$ 
       \\\hline
      $95$ &
      $\displaystyle{\frac{1}{X_{95}^5}+\frac{5}{X_{95}^4}+\frac{5}{X_{95}^3}
      -\frac{5}{X_{95}^2}-\frac{9}{X_{95}}=w,}$
      &$\displaystyle{w= \frac{1}{X_{19}(q^5)}+\frac{1}{X_{19}(q)} +2
      }$ 
       \\\hline
      $105$ &
      $\displaystyle{\frac{1}{X_{105}^3}+\frac{3}{X_{105}^2}+\frac{1}{X_{105}}=w,}$
      &$\displaystyle{w= \frac{1}{X_{35}(q^3)}+\frac{1}{X_{35}(q)}
      }$ 
       \\
      \hline 
      $110$ &
      $\displaystyle{\frac{1}{X_{110}^2}+\frac{3}{X_{110}}=w,}$
      &$\displaystyle{w= \frac{1}{X_{55}(q^2)}+\frac{1}{X_{55}(q)} -4
      }$ 
       \\\hline
      $119$ &
      $\displaystyle{\frac{1}{X_{119}^7}-\frac{7}{X_{119}^3}-\frac{7}{X_{119}^2}-\frac{6}{X_{119}}=w,}$
      &$\displaystyle{w= \frac{1}{X_{17}(q^7)}+\frac{1}{X_{17}(q)} +3
      }$ 
       \\
      \hline
  \end{longtable}

  \setlength{\tabcolsep}{9pt} 
\renewcommand{\arraystretch}{2.55} 

\begin{longtable}{ | c | l |}
  \caption{The associated weight 2 modular forms~$Z$ for Theorem~\ref{T:1}\\}
   \label{T:Z}\\
       \hline
      $\ell$ & $Z=Z_\ell$ \\
      \hline
      $1$ &
      $\displaystyle{Z_1 
      = \frac{\eta_1^4}{X_1^{1/6}}
      = Q^{1/2}(q) 
      }$ \\
      \hline
      $2$ &
      $\displaystyle{Z_2 
      = \frac{\eta_1^2\eta_2^2}{X_2^{1/4}}
      = 2P_2-P_1 
      }$  \\
      \hline
      $3$ &
      $\displaystyle{Z_3 
      = \frac{\eta_1^2\eta_3^2}{X_3^{1/3}}
      = \frac{3P_3-P_1}{2}
      }$ \\
      \hline
      $4$ &
      $\displaystyle{Z_4 
      = \frac{\eta_1^2\eta_4^2}{X_4^{5/12}}
      = \frac{4P_4-P_1}{3} 
      }$  \\
      \hline
      $5$ &
      $\displaystyle{Z_5 
      = \frac{\eta_1^2\eta_5^2}{X_5^{1/2}}
      = \frac{5P_5-P_1}{4} 
      }$ \\
      \hline
      $6$ &
      $\displaystyle{Z_6 
      = \frac{\eta_1\eta_2\eta_3\eta_6}{X_6^{1/2}}
      = \frac{6P_6+3P_3-2P_2-P_1}{6} 
      }$  \\
      \hline
      $7$ &
      $\displaystyle{Z_7 
      = \frac{\eta_1^2\eta_7^2}{X_7^{2/3}}
      = \frac{7P_7-P_1}{6} 
      }$ \\
      \hline
      $8$ &
      $\displaystyle{Z_8 
      = \frac{\eta_2^2\eta_4^2}{X_8^{1/2}}
      = \frac{8P_8-2P_4+P_2-P_1}{6} 
      }$  \\
      \hline
      $9$ &
      $\displaystyle{Z_9 
      = \frac{\eta_3^4}{X_9^{1/2}}
      = \frac{9P_9-P_1}{8} 
      }$  \\
      \hline
      $10$ &
      $\displaystyle{Z_{10} 
      = \frac{\eta_1\eta_2\eta_5\eta_{10}}{X_{10}^{3/4}}
      = \frac{10P_{10}+5P_5-2P_2-P_1}{12} 
      }$  \\
      \hline
      $11$ &
      $\displaystyle{Z_{11} 
      = \frac{\eta_1^2\eta_{11}^2}{X_{11}}
      = \frac{11P_{11}-P_1}{2(5-8X_{11})} 
      }$ \\
      \hline
      $12$ &
      $\displaystyle{Z_{12} 
      = \frac{\eta_1\eta_3\eta_4\eta_{12}}{X_{12}^{5/6}}
      = \frac{12P_{12}-6P_6-4P_4+3P_3+2P_2-P_1}{6} 
      }$ \\
      \hline
      $13$ &
      $\displaystyle{Z_{13} 
      = \frac{\eta_1^2\eta_{13}^2}{X_{13}^{7/6}}
      = \frac{13P_{13}-P_1}{12(1+X_{13})^{1/2}}}$ \\
      \hline
      $14$ &
      $\displaystyle{Z_{14} 
      = \frac{\eta_1\eta_2\eta_7\eta_{14}}{X_{14}}
      =\frac{14P_{14}+7P_7-2P_2-P_1}{6(3-11X_{14})}}$ \\
      \hline
      $15$ &
      $\displaystyle{Z_{15} 
      = \frac{\eta_1\eta_3\eta_5\eta_{15}}{X_{15}} 
      =\frac{15P_{15}+5P_5-3P_3-P_1}{8(2-3X_{15})}}$ \\
      \hline
      $16$ &
      $\displaystyle{Z_{16} 
      = \frac{\eta_2^{10}\eta_8^{10}}{\eta_1^4\eta_{4}^8\eta_{16}^4}
      = \frac{4P_8-P_2}{3(1-2X_{16})^2}}$ 
      \\
      \hline
      $17$ &
      $\displaystyle{Z_{17} 
      = \frac{\eta_1^2\eta_{17}^2}{X_{17}^{3/2}}
      = \frac{17P_{17}-P_1}{8(2+X_{17})}
      }$ \\
      \hline
      $18$ &
      $\displaystyle{Z_{18} 
      = \left(\frac{\eta_3^8\eta_{6}^8}{X_{18}^{3}(1+X_{18})^{3}}\right)^{1/4}
      = \frac{18P_{18}-9P_9-12P_6+6P_3+2P_2-P_1}{4(1+4X_{18})^{3/2}}}$
       \\
      \hline
      $19$ &
      $\displaystyle{Z_{19} 
      =\frac{\eta_1^2\eta_{19}^2}{X_{19}^{5/3}}
      = \frac{19P_{19}-P_1}{6(3-5X_{19})}}$  \\
      \hline
      $20$ &
      $\displaystyle{Z_{20} 
      = \frac{\eta_2^2\eta_{10}^2}{X_{20}}
      = \frac{20P_{20}-5P_5+4P_4-P_1}{6(3-8X_{20})}}$
       \\
      \hline
      $21$ &
      $\displaystyle{Z_{21} 
      = \frac{\eta_1\eta_3\eta_7\eta_{21}}{X_{21}^{4/3}}
      = \frac{21P_{21}-7P_7+3P_3-P_1}{8(2+5X_{21})}}$ \\
      \hline
      $22$ &
      $\displaystyle{Z_{22} 
      = \frac{\eta_1\eta_2\eta_{11}\eta_{22}}{X_{22}^{3/2}}
      = \frac{22P_{22}+11P_{11}-2P_2-P_1}{6(5-6X_{22})}}$ \\
      \hline
      $23$ &
      $\displaystyle{Z_{23} 
      = \frac{\eta_1^2\eta_{23}^2}{X_{23}^2}
      = \frac{23P_{23}-P_1}{2(11-10X_{23}+3X_{23}^2)}}$ \\
      \hline
      $24$ &
      $\displaystyle{Z_{24} 
      = \frac{\eta_2\eta_4\eta_6\eta_{12}}{X_{24}}
      = \frac{24P_{24}-6P_{12}+8P_8+3P_6-2P_4-3P_3+P_2-P_1}{24(1-X_{24})}}$  \\
      \hline
      $25$ &
       $\displaystyle{Z_{25} 
      = \left(\frac{\eta_1^{12}\eta_{25}^{12}}{X_{25}^{13}(1+X_{25}-X_{25}^2)}\right)^{1/6}
      = \frac{25P_{25}-P_1}{24(1+X_{25}-X_{25}^2)}
      }$  \\
      \hline
      $26$ &
       $\displaystyle{Z_{26} 
      = \frac{\eta_1\eta_2\eta_{13}\eta_{26}}{X_{26}^{7/4}}
      = \frac{26P_{26}+13P_{13}-2P_{2}-P_1}{12(3+5X_{26})}
      }$  \\
      \hline
      $27$ &
      $\displaystyle{Z_{27} 
      = \left(\frac{\eta_1^6\eta_{27}^6}{X_{27}^{7}(1-3X_{27}+3X_{27}^2)}\right)^{1/3}
      = \frac{27P_{27}-9P_9+3P_3-P_1}{4(5-24X_{27}+45X_{27}^2-36X_{27}^3)}}$
      \\
      \hline
      $28$ &
       $\displaystyle{Z_{28} 
      = \frac{\eta_1\eta_{4}\eta_{7}\eta_{28}}{X_{28}^{5/3}}
      = \frac{28P_{28}-14P_{14}+7P_7-4P_4+2P_{2}-P_1}{6(3-8X_{28})}
      }$  \\
      \hline
      $29$ &
      $\displaystyle{Z_{29} 
      = \frac{\eta_1^2\eta_{29}^2}{X_{29}^{5/2}}
      = \frac{29P_{29}-P_1}{4(7-15X_{29}+5X_{29}^2)}
      }$  \\
      \hline 
      $30$ &
      $\displaystyle{Z_{30} 
      = \left(\frac{\eta_1\eta_2\eta_3\eta_5\eta_6\eta_{10}\eta_{15}\eta_{30}}{X_{30}^3}\right)^{1/2}
      = \frac{30P_{30}+15P_{15}-10P_{10}+6P_6-5P_5+3P_3-2P_2-P_1}{12(3+8X_{30})}
      }$  \\
      \hline
      $31$ &
      $\displaystyle{Z_{31} 
      = \frac{\eta_1^2\eta_{31}^2}{X_{31}^{8/3}}
      = \frac{31P_{31}-P_1}{6(5+14X_{31}+5X_{31}^2)}
      }$  \\
      \hline
      $32$ &
      $\displaystyle{Z_{32} 
      = \frac{\eta_4^2\eta_8^2}{X_{32}(1-2X_{32}+2X_{32}^2)}}
      = \frac{32P_{32}-8P_{16}+P_2-P_1}{24(1-X_{32})(1-2X_{32}+2X_{32}^2)}
      $  \\
      \hline
      $33$ &
      $\displaystyle{Z_{33} 
      = \frac{\eta_1\eta_3\eta_{11}\eta_{33}}{X_{33}^2}
      = \frac{33P_{33}+11P_{11}-3P_{3}-P_1}{8(5+8X_{33}+2X_{33}^2)}
      }$  \\
      \hline
      $34$ &
       $\displaystyle{Z_{34} 
      = \frac{\eta_1\eta_2\eta_{17}\eta_{34}}{X_{34}^{9/4}}
      = \frac{34P_{34}+17P_{17}-2P_2-P_1}{24(2-5X_{34})(1-X_{34})^{1/2}}
      }$  \\
      \hline 
      $35$ &
      $\displaystyle{Z_{35} 
      = \frac{\eta_1\eta_5\eta_{7}\eta_{35}}{X_{35}^2}
      = \frac{35P_{35}-7P_{7}+5P_{5}-P_1}{8(4-9X_{35}+14X_{35}^2)}
      }$  \\
      \hline
      $36$ &
      $\displaystyle{Z_{36} 
      = \frac{\eta_6^4}{X_{36}(1+X_{36})^{1/2}(1-3X_{36})^{1/2}}
      = \frac{36P_{36}+9P_{9}-4P_{4}-P_1}{8(1-3X_{36})(5+3X_{36})}
      }$  \\
      \hline
      $38$ &
       $\displaystyle{Z_{38} 
      = \frac{\eta_1\eta_2\eta_{19}\eta_{38}}{X_{38}^{5/2}}
      = \frac{38P_{38}+19P_{19}-2P_2-P_1}{6(9+13X_{38}+6X_{38}^2)}
      }$  \\
      \hline 
      $39$ &
      $\displaystyle{Z_{39} 
      = \frac{\eta_1\eta_3\eta_{13}\eta_{39}}{X_{39}^{7/3}}
      = \frac{39P_{39}-13P_{13}+3P_{3}-P_1}{4(7-36X_{39}+35X_{39}^2)}
      }$  \\
      \hline
      $41$ &
       $\displaystyle{Z_{41} 
      = \frac{\eta_1^2\eta_{41}^2}{X_{41}^{7/2}}
      = \frac{41P_{41}-P_1}{8(5+13X_{41}+5X_{41}^2-5X_{41}^3)}
      }$  \\
      \hline 
      $42$ &
      $\displaystyle{Z_{42} 
      = \left(\frac{\eta_1\eta_2\eta_{3}\eta_{6}\eta_7\eta_{14}\eta_{21}\eta_{42}}{X_{42}^{4}}\right)^{1/2}
      = \frac{42P_{42}-21P_{21}-14P_{14}+7P_7-6P_6+3P_3+2P_2-P_1}
      {12(1+4X_{42})^{1/2}(1+5X_{42}+8X_{42}^2)^{1/2}}
      }$  \\
      \hline
      $44$ &
      $\displaystyle{Z_{44} 
      = \frac{\eta_1\eta_4\eta_{11}\eta_{44}}{X_{44}^{5/2}}
      = \frac{44P_{44}-22P_{22}+11P_{11}-4P_4+2P_2-P_1}{6(5-16X_{44}+16X_{44}^2)}
      }$  \\
      \hline
      $45$ &
      $\displaystyle{Z_{45} 
      = \frac{\eta_1\eta_5\eta_{9}\eta_{45}}{X_{45}^{5/2}}
      = \frac{45P_{45}+9P_{9}-5P_5-P_1}{24(2-3X_{45})(1+X_{45}-X_{45}^2)^{1/2}}
      }$  \\
      \hline
      $46$ &
      $\displaystyle{Z_{46} 
      = \frac{\eta_1\eta_2\eta_{23}\eta_{46}}{X_{46}^{3}}
      = \frac{46P_{46}+23P_{23}-2P_2-P_1}{6(11+15X_{46}-9X_{46}^2-X_{46}^3)}
      }$  \\
      \hline
      $47$ &
      $\displaystyle{Z_{47} 
      = \frac{\eta_1^2\eta_{47}^2}{X_{47}^{4}}
      = \frac{47P_{47}-P_1}{2(23+58X_{47}+83X_{47}^2+52X_{47}^3+12X_{47}^4)}
      }$  \\
      \hline
      $49$ &
       $\displaystyle{Z_{49} 
      = \left(\frac{\eta_1^{12}\eta_{49}^{12}}
      {X_{49}^{25}(1-4X_{49}+3X_{49}^2+X_{49}^3)}\right)^{1/6}
      = \frac{49P_{49}-P_1}{24(2-5X_{49})(1-4X_{49}+3X_{49}^2+X_{49}^3)}
      }$  \\
      \hline 
      $50$ &
       $\displaystyle{Z_{50} 
      = \left(\frac{\eta_1^4\eta_2^4\eta_{25}^4\eta_{50}^4}{X_{50}^{13}(1+3X_{50}+X_{50}^2)}\right)^{1/4}
      = \frac{50P_{50}+25P_{25}-2P_2-P_1}{24(3+7X_{50})(1+3X_{50}+X_{50}^2)}
      }$  \\
      \hline
      $51$ &
       $\displaystyle{Z_{51} 
      = \frac{\eta_1\eta_3\eta_{17}\eta_{51}}{X_{51}^{3}}
      = \frac{51P_{51}+17P_{17}-3P_3-P_1}{8(8-37X_{51}+56X_{51}^2-30X_{51}^3)}
      }$  \\
      \hline 
      $54$ &
       $\displaystyle{Z_{54} 
      = \left(\frac{\eta_1^2\eta_2^2\eta_{27}^2\eta_{54}^2}
      {X_{54}^{7}(1-X_{54}+X_{54}^2)}\right)^{1/2}
      = \frac{54P_{54}+27P_{27}-2P_2-P_1}{6(13-35X_{54}+51X_{54}^2-35X_{54}^3+10X_{54}^4)}
      }$  \\
      \hline 
      $55$ &
       $\displaystyle{Z_{55} 
      = \frac{\eta_1\eta_5\eta_{11}\eta_{55}}{X_{55}^{3}}
      = \frac{55P_{55}+11P_{11}-5P_5-P_1}{12(5-8X_{55}-11X_{55}^2+4X_{55}^3)}
      }$  \\
      \hline 
      $56$ &
      $\displaystyle{Z_{56} 
      = \frac{\eta_1\eta_7\eta_{8}\eta_{56}}{X_{56}^{3}}
      = \frac{56P_{56}-14P_{28}-7P_{14}-8P_{8}+7P_{7}+2P_{4}+P_2-P_1}
      {12(3-4X_{56}+2X_{56}^2-4X_{56}^3)}
      }$  \\
      \hline
      $59$ &
       $\displaystyle{Z_{59} 
      = \frac{\eta_1^2\eta_{59}^2}{X_{59}^{5}}
      = \frac{59P_{59}-P_1}{58+140X_{59}+120X_{59}^2+90X_{59}^3+40X_{59}^4+8X_{59}^5}
      }$  \\
      \hline
      $60$ &
      \shortstack[l]{$\displaystyle{Z_{60} 
      = \frac{\eta_2\eta_6\eta_{10}\eta_{30}}{X_{60}^{2}}}$\\
     $\displaystyle{ = \frac{60P_{60}-30P_{30}+20P_{20}+15P_{15}-12P_{12}-10P_{10}+6P_6+5P_5-4P_4-3P_3+2P_2-P_1}
      {24(2-3X_{60}+4X_{60}^2)}
      }$ }  \\
      \hline
      $62$ &
       $\displaystyle{Z_{62} 
      = \frac{\eta_1\eta_2\eta_{31}\eta_{62}}{X_{62}^{4}}
      = \frac{62P_{62}+31P_{31}-2P_2-P_1}{6(15+19X_{62}+9X_{62}^2-9X_{62}^3-6X_{62}^4)}
      }$  \\
      \hline
      $66$ &
      $\displaystyle{Z_{66} 
      = \left(\frac{\eta_1\eta_2\eta_3\eta_6\eta_{11}\eta_{22}\eta_{33}\eta_{66}}
      {X_{66}^{6}}\right)^{1/2}
      = \frac{66P_{66}+33P_{33}+22P_{22}+11P_{11}-6P_6-3P_3-2P_2-P_1}
      {24(1-X_{66})^{1/2}(5+6X_{66}-14X_{66}^2)}}$  \\
      \hline
      $69$ &
       $\displaystyle{Z_{69} 
      = \frac{\eta_1\eta_3\eta_{23}\eta_{69}}{X_{69}^{4}}
      = \frac{69P_{69}+23P_{23}-3P_3-P_1}{8(11+14X_{69}-10X_{69}^2-6X_{69}^3+6X_{69}^4)}
      }$  \\
          \hline
      $70$ &
      $\displaystyle{Z_{70} 
      = \left(\frac{\eta_1\eta_2\eta_{5}\eta_{7}\eta_{10}\eta_{14}\eta_{35}\eta_{70}}{X_{70}^{6}}\right)^{1/2}
      =\frac{70P_{70}-35P_{35}+14P_{14}-10P_{10}-7P_7+5P_5-2P_2+P_1}
      {12(3+X_{70})(1-4X_{70})^{1/2}(1-3X_{70}+4X_{70}^2)^{1/2}}}$  \\
      \hline
      $71$ &
      $\displaystyle{Z_{71} 
      = \frac{\eta_1^2\eta_{71}^2}{X_{71}^{6}}
      = \frac{71P_{71}-P_1}{(70+164X_{71}+50X_{71}^2-180X_{71}^3-130X_{71}^4+40X_{71}^5+40X_{71}^6)}
      }$  \\
    \hline
      $78$ &
      $\displaystyle{Z_{78} 
      = \left(\frac{\eta_1\eta_2\eta_{3}\eta_{6}\eta_{13}\eta_{26}\eta_{39}\eta_{78}}{X_{78}^{7}}\right)^{1/2}
      =\frac{78P_{78}-39P_{39}+26P_{26}-13P_{13}-6P_6+3P_3-2P_2+P_1}
      {24(2-X_{78})(1-4X_{78})^{1/2}(1-2X_{78}+X_{78}^2-4X_{78}^3)^{1/2}}}$  \\
      \hline
      $87$ &
       $\displaystyle{Z_{87} 
      = \frac{\eta_1\eta_3\eta_{29}\eta_{87}}{X_{87}^{5}}
      = \frac{87P_{87}+29P_{29}-3P_3-P_1}{8(14+17X_{87}+40X_{87}^2+22X_{87}^3+18X_{87}^4+3X_{87}^5)}
      }$  \\
      \hline 
      $92$ &
      $\displaystyle{Z_{92} 
      = \frac{\eta_1\eta_4\eta_{23}\eta_{92}}{X_{92}^{5}}
      = \frac{92P_{92}-46P_{46}+23P_{23}-4P_4+2P_2-P_1}
      {6(11-40X_{92}+80X_{92}^2-104X_{92}^3+80X_{92}^4-32X_{92}^5)}
      }$  \\
      \hline
      $94$ &
       $\displaystyle{Z_{94} 
      = \frac{\eta_1\eta_2\eta_{47}\eta_{94}}{X_{94}^{6}}
      = \frac{94P_{94}+47P_{47}-2P_2-P_1}
      {6(23+27X_{94}-45X_{94}^2+5X_{94}^3+48X_{94}^4-42X_{94}^5+12X_{94}^6)}
      }$  \\
      \hline
      $95$ &
       $\displaystyle{Z_{95} 
      = \frac{\eta_1\eta_5\eta_{19}\eta_{95}}{X_{95}^{5}}
      = \frac{95P_{95}+19P_{19}-5P_5-P_1}{12(9+56X_{95}+115X_{95}^2+100X_{95}^3+41X_{95}^4+4X_{95}^5)}
      }$  \\
      \hline
      $105$ &
      \shortstack[l]{$\displaystyle{Z_{105} 
      = \left(\frac{\eta_1\eta_3\eta_5\eta_7\eta_{15}\eta_{21}\eta_{35}\eta_{105}}{X_{105}^{8}}\right)^{1/2}}$\\
  $\displaystyle{= \frac{105P_{105}-35P_{35}+21P_{21}+15P_{15}-7P_7-5P_5+3P_3-P_1}
      {24(4+9X_{105})(1+X_{105}-X_{105}^2)^{1/2}(1+4X_{105}+4X_{105}^2-4X_{105}^3)^{1/2}}
      }$}  \\
      \hline
      $110$ &
       \shortstack[l]{$\displaystyle{Z_{110} 
      = \left(\frac{\eta_1\eta_2\eta_5\eta_{10}\eta_{11}\eta_{22}\eta_{55}\eta_{110}}{X_{110}^{9}}\right)^{1/2}}$\\
    $\displaystyle{ = \frac{110P_{110}+55P_{55}+22P_{22}+11P_{11}-10P_{10}-5P_5-2P_2-P_1}
      {12(15+77X_{110}+183X_{110}^2+187X_{110}^3+58X_{110}^4)}
      }$ } \\
      \hline
      $119$ &
       $\displaystyle{Z_{119} 
      = \frac{\eta_1\eta_7\eta_{17}\eta_{119}}{X_{119}^{6}}
      = \frac{119P_{119}+17P_{17}-7P_7-P_1}{8(16+19X_{119}+44X_{119}^2+75X_{119}^3+44X_{119}^4+26X_{119}^5+7X_{119}^6)}
      }$  \\
      \hline
\end{longtable}

  \setlength{\tabcolsep}{4pt} 
\renewcommand{\arraystretch}{1.68}
\begin{longtable}{ | c | l | l |} 
  \caption{Polynomials $G(x)$ and $H(x)$ for Theorem~\ref{T:1}\\}
    \label{T:GH}\\
    \hline
    $\ell$ & $G=G_\ell$ & $H=H_{\ell}$ \\
    \hline
    $1$ & $1-1728x$ & $240x$ \\
    \hline
    $2$ & $1-256x$ & $48x$ \\
    \hline
    $3$ & $1-108x$ & $24x$ \\
    \hline
    $4$ & $1-64x$ & $16x$ \\
    \hline
    $5$ & $1-44x-16x^{2}$ & $12x(1+x)$ \\
    \hline
    $6$ & $(1+4x)(1-32x)$ & $8x(1+12x)$ \\
    \hline
    $7$ & $(1+x)(1-27x)$ & $8x(1+3x)$ \\
    \hline
    $8$ & $1-24x+16x^{2}$ & $8x(1-2x)$ \\
    \hline
    $9$ & $1-18x-27x^{2}$ & $3x(2+9x)$ \\
    \hline
    $10$ & $(1+4x)(1-16x)$ & $4x(1+15x)$ \\
    \hline
    $11$ & $1-20x+56x^{2}-44x^{3}$ & $8x(1-8x+11x^{2})$ \\
    \hline
    $12$ & $(1-4x)(1-16x)$ & $8x(1-8x)$ \\
    \hline
    $13$ & $(1+x)(1-10x-27x^{2})$ & $\frac{x}{4}(12+175x+231x^{2})$ \\
    \hline
    $14$ & $(1-4x)(1-18x+49x^{2})$ & $x(10-141x+392x^{2})$ \\
    \hline
    $15$ & $(1-12x)(1-2x+5x^{2})$ & $3x(2-11x+40x^{2})$ \\
    \hline
    $16$ & $(1-2x)^{2}(1-12x+4x^{2})$ & $8x(1-2x)(1-8x+4x^{2})$ \\
    \hline
    $17$ & $1-6x-27x^{2}-28x^{3}-16x^{4}$ & $x(2+35x+68x^{2}+60x^{3})$ \\
    \hline
    $18$ & $(1+x)^{2}(1+4x)(1-8x)$ & $24x^{2}(1+x)(2+5x)$ \\
    \hline
    $19$ & $(1+x)(1-13x+35x^{2}-27x^{3})$ & $x(6-31x-24x^{2}+105x^{3})$ \\
    \hline
    $20$ & $(1-4x)(1-12x+16x^{2})$ & $8x(1-10x+18x^{2})$ \\
    \hline
    $21$ & $(1+4x)(1-2x-27x^{2})$ & $-x(2-47x-240x^{2})$ \\
    \hline
    $22$ & $(1-8x)(1-4x^{2}+4x^{3})$ & $4x(1-3x)(1+4x-10x^{2})$ \\
    \hline
    $23$ & $(1-x^{2}+x^{3})(1-8x+3x^{2}-7x^{3})$ & $4x\big(1-x-x^{2}+12x^{3} -15x^{4}+14x^{5}\big)$ \\
    \hline
    $24$ & $(1+4x^{2})(1-8x+4x^{2})$ & $4x(1-3x+20x^{2}-16x^{3})$ \\
    \hline
    $25$ & $(1+x-x^{2})^2(1-4x-16x^{2})$ & $20x^{2}(1+x-x^{2})(2+2x-7x^{2})$ \\
    \hline
    $26$ & $(1+4x)(1-2x-15x^{2}-16x^{3})$ & $-x(1+2x)(2-35x-126x^{2})$ \\
    \hline
    $27$ & $(1-3x+3x^{2})^2(1-12x+36x^{2}-36x^{3})$ & $12x(1-3x)(1-3x+3x^{2})$ \\
    & & $\times(1-12x+36x^{2}-36x^{3})$ \\
    \hline
    $28$ & $(1-x)(1-8x)(1-5x+8x^{2})$ & $8x(1-4x)(1-7x+8x^{2})$ \\
    \hline
    $29$ & $1-10x+23x^{2}-10x^{3}$ & $x(6-37x+26x^{2}+68x^{3}-132x^{4}+140x^{5})$ \\
    & \quad$-15x^{4}+20x^{5}-16x^{6}$ &  \\
    \hline
    $30$ & $(1+x)(1+5x)(1-4x)(1+4x)$ & $-4x(1-4x-58x^{2}-75x^{3})$ \\
    \hline
    $31$ & $(1+4x+3x^{2}+x^{3})(1-17x^{2}-27x^{3})$ & $-4x(1-5x-69x^{2}-180x^{3}-161x^{4}-60x^{5})$ \\
    \hline
    $32$ & $(1-2x+2x^{2})^2$ & $8x(1-2x+2x^{2})(1-10x+36x^{2}$ \\
    & $\times (1-8x+12x^{2}-16x^{3}+4x^{4})$ & \qquad$-60x^{3}+60x^{4}-16x^{5})$ \\
    \hline
    $33$ & $(1-2x-11x^{2})(1+4x+8x^{2}+4x^{3})$ & $-x(2-15x-152x^{2}-404x^{3}-264x^{4})$ \\
    \hline
    $34$ & $(1-x)(1-x-4x^{2})(1-9x+16x^{2})$ & $\frac{x}{4}(28-197x-15x^{2}+1484x^{3}-1584x^{4})$ \\
    \hline
    $35$ & $(1-2x+5x^{2})(1-8x+16x^{2}-28x^{3})$ & $x(6-61x+296x^{2}-580x^{3}+840x^{4})$ \\
    \hline
    $36$ & $(1+x)^{2}(1-3x)^{2}(1-6x-3x^{2})$ & $3x(1+x)(1-3x)(2-3x-48x^{2}-27x^{3})$ \\
    \hline
    $38$ & $(1+4x+4x^{2}+4x^{3})$ & $-x(2-15x-116x^{2}-316x^{3}$ \\
    & $\times (1-2x-7x^{2}-8x^{3})$ & \qquad$-392x^{4}-280x^{5})$ \\
    \hline
    $39$ & $(1-4x)(1-3x-x^{2})(1-11x+27x^{2})$ & $4x(3-46x+210x^{2}-253x^{3}-168x^{4})$ \\
    \hline
    $41$ & $1+4x-8x^{2}-66x^{3}-120x^{4}$ & $-x(1+2x)(4-20x-174x^{2}-240x^{3}$ \\
    & $-56x^{5}+53x^{6}+36x^{7}-16x^{8}$ & \qquad$+136x^{4}+285x^{5}-126x^{6})$ \\
    \hline
    $42$ & $(1+x)(1-3x)(1+4x)(1+5x+8x^{2})$ & $-\frac{x}{4}(20+49x-567x^{2}-2508x^{3}-2304x^{4})$ \\
    \hline
    $44$ & $(1-4x+8x^{2}-4x^{3})$ & $8x(1-12x+57x^{2}-132x^{3}+160x^{4}-72x^{5})$ \\
    & $\times(1-8x+16x^{2}-16x^{3})$ &  \\
    \hline
    $45$ & $(1+x-x^{2})(1-3x+3x^{2})(1-3x-9x^{2})$ & $\frac{3x}{4}(4+13x-123x^{2}+87x^{3}+396x^{4}-324x^{5})$ \\
    \hline
    $46$ & $(1-2x-7x^{2})(1+2x-3x^{2}+x^{3})$ & $-x(2-23x-104x^{2}-28x^{3}+66x^{4}$ \\
    & $\times(1+2x+x^{2}+x^{3})$ & \qquad$-13x^{5}+144x^{6}-105x^{7})$ \\
    \hline
    $47$ & $(1+4x+7x^{2}+8x^{3}+4x^{4}+x^{5})$ & $-4x\big(1+3x-23x^{2}-177x^{3}-560x^{4}$ \\
    & $\times(1-5x^{2}-20x^{3}-24x^{4}-19x^{5})$ & \qquad$-1087x^{5}-1347x^{6}-1098x^{7}$ \\
    & & \qquad$-500x^{8}-114x^{9}\big)$\\
    \hline
    $49$ & $(1-4x+3x^{2}+x^{3})^2$ & $7x(1-4x+3x^{2}+x^{3})\big(2-27x+108x^{2}$ \\
    & $\times (1-10x+27x^{2}-10x^{3}-27x^{4})$ & \qquad$-105x^{3}-171x^{4}+216x^{5}+96x^{6}\big)$ \\
    \hline
    $50$ & $(1+4x)(1+3x+x^{2})^2$ & $-5x(1+3x+x^{2})\big(2+13x-17x^{2}$ \\
    & $\times (1+2x-7x^{2}-16x^{3})$ & \qquad$-247x^{3}-472x^{4}-204x^{5}\big)$ \\
    \hline
    $51$ & $(1-8x+16x^{2}-12x^{3})$ & $x\big(10-141x+798x^{2}-2488x^{3}$ \\
    & $\times (1-6x+15x^{2}-22x^{3}+17x^{4})$& \qquad$+4656x^{4}-4996x^{5}+2448x^{6}\big)$ \\
    \hline
    $54$ & $(1-x+x^{2})^{2}(1+4x^{3})$ & $3x(1-x+x^{2})\big(2-15x+39x^{2}+x^{3}$ \\
    & $\times (1-6x+9x^{2}-8x^{3})$ & \qquad$-156x^{4}+408x^{5}-424x^{6}+264x^{7}\big)$ \\
    \hline
    $55$ & $(1+x-x^{2})(1-7x+11x^{2})$ & $4x\big(1+2x-34x^{2}-7x^{3}+148x^{4}$ \\
    & $\times (1-4x^{2}-4x^{3})$ & \qquad$+64x^{5}-132x^{6}\big)$ \\
    \hline
    $56$ & $(1-x)(1-2x)(1+x+2x^{2})$ & $4x\big(1-4x+16x^{2}-59x^{3}+64x^{4}$ \\
    & $\times (1-4x-8x^{3}+4x^{4})$ & \qquad$-88x^{5}+152x^{6}-64x^{7}\big)$ \\
    \hline
    $59$ & $(1+2x+x^{3})\big(1-2x-4x^{2}-21x^{3}-44x^{4}$ & $-x\big(4+4x-86x^{2}-448x^{3}-1216x^{4}$ \\
    & $\qquad-60x^{5}-61x^{6}-46x^{7}-24x^{8}-11x^{9}\big)$ & \qquad$-2217x^{5}-3024x^{6}-3128x^{7}-2600x^{8}$ \\
    & & \qquad$-1748x^{9}-720x^{10}-385x^{11}\big)$ \\
    \hline
    $60$ & $(1-x)(1-4x)(1+4x^{2})(1-x+4x^{2})$ & $4x\big(1-8x+40x^{2}-83x^{3}+160x^{4}-144x^{5}\big)$ \\
    \hline
    $62$ & $(1+x^{2}-x^{3})(1+4x+5x^{2}+3x^{3})$ & $-x\big(2-7x-76x^{2}-224x^{3}-214x^{4}+7x^{5}$ \\
    & $\times(1-2x-3x^{2}-4x^{3}+4x^{4})$ & \qquad$+460x^{6}+351x^{7}+80x^{8}-288x^{9}\big)$ \\
    \hline
    $66$ & $(1-x)(1+3x)(1-x-8x^{2})$ & $-\frac{x}{4}\big(4-147x-133x^{2}+1756x^{3}$ \\
    & $\times(1-4x^{2}+4x^{3})$ & \qquad$-264x^{4}-5468x^{5}+4608x^{6}\big)$ \\
    \hline
    $69$ & $(1-x^{2}+x^{3})(1+4x+7x^{2}+5x^{3})$ & $-x\big(2-15x-82x^{2}-68x^{3}+186x^{4}+77x^{5}$ \\
    & $\times(1-2x-5x^{2}+6x^{3}-3x^{4})$ & \qquad$-618x^{6}-51x^{7}+480x^{8}-360x^{9}\big)$ \\
    \hline
    $70$ & $(1-4x)(1+x)(1-3x+4x^{2})$ & $\frac{x}{4}\big(20-97x-71x^{2}+1225x^{3}$ \\
    & $\times (1-x-x^{2}+5x^{3})$ & \qquad$-1729x^{4}-640x^{5}+3840x^{6}\big)$ \\
    \hline
    $71$ & $(1+4x+5x^{2}+x^{3}-3x^{4}-2x^{5}+x^{7})$ & $-4x\big(1-x-38x^{2}-112x^{3}-29x^{4}+362x^{5}$ \\
    & $\times(1-7x^{2}-11x^{3}+5x^{4}$ & \qquad$+563x^{6}-25x^{7}-717x^{8}-475x^{9}+248x^{10}$ \\
    & \qquad$+18x^{5}+4x^{6}-11x^{7})$ & \qquad$+372x^{11}+42x^{12}-132x^{13}\big)$ \\
    \hline
    $78$ & $(1-4x)(1-x+x^{2})(1-x-3x^{2})$ & $x\big(6-35x+56x^{2}-51x^{3}-98x^{4}$ \\
    & $\times(1-2x+x^{2}-4x^{3})$& \qquad$+583x^{5}-708x^{6}+756x^{7}\big)$ \\
    \hline
    $87$ & $(1-2x-x^{2}-x^{3})(1+2x+3x^{2}+3x^{3})$ & $-x\big(2+17x-6x^{2}-124x^{3}-646x^{4}$ \\
    & $\times(1+2x+7x^{2}+6x^{3}$ & \qquad$-1635x^{5}-2862x^{6}-4287x^{7}-4048x^{8}$ \\
    & \qquad $+13x^{4}+4x^{5}+8x^{6})$ & \qquad$-3876x^{9}-1800x^{10}-840x^{11}\big)$ \\
    \hline
    $92$ & $(1-x+2x^{2}-x^{3})(1-4x+4x^{2}-8x^{3})$ & $8x\big(1-13x+81x^{2}-316x^{3}+880x^{4}$ \\
    & $\times(1-5x+14x^{2}-25x^{3}$ & \qquad$-1851x^{5}+2996x^{6}-3772x^{7}+3636x^{8}$ \\
    & $\qquad+28x^{4}-20x^{5}+8x^{6})$ & \qquad$-2560x^{9}+1232x^{10}-288x^{11}\big)$ \\
    \hline
    $94$ & $(1-2x-3x^{2}+4x^{3}-4x^{4})$ & $-x\big(2-23x-52x^{2}+154x^{3}-66x^{4}$ \\
    & $\times(1+4x+3x^{2}-2x^{3}+2x^{4}+5x^{5})$ & \qquad$-745x^{5}+732x^{6}+266x^{7}-2108x^{8}$ \\
    & $\times(1-x^{2}+2x^{3}-2x^{4}+x^{5})$ & \qquad$+1904x^{9}+68x^{10}-2281x^{11}$ \\
    & & \qquad $+2184x^{12}-960x^{13}\big)$ \\
    \hline
    $95$ & $(1+4x+4x^{2}+4x^{3})$ & $-4x\big(3+40x+199x^{2}+403x^{3}-127x^{4}$ \\
    & $\times (1+5x+7x^{2}+5x^{3}+x^{4})$  & \qquad$-2516x^{5}-6309x^{6}-8606x^{7}-7036x^{8}$ \\
    & $\times(1+5x+3x^{2}-15x^{3}-19x^{4})$ & \qquad$-3264x^{9}-570x^{10}\big)$ \\
    \hline
    $105$ & $(1+x-x^{2})(1+x-5x^{2})(1+5x+7x^{2})$ & $-\frac{x}{4}\big(36+329x+593x^{2}-2781x^{3}$ \\
    & $\times(1+4x+4x^{2}-4x^{3})$ & \qquad$-11400x^{4}-7212x^{5}+17268x^{6}$ \\
    & & \qquad$+13384x^{7}-11200x^{8}\big)$ \\
    \hline
    $110$ & $(1+3x+x^{2})(1+3x+5x^{2})$ & $-x\big(10+143x+832x^{2}+2475x^{3}+3146x^{4}$ \\
    & $\times (1+4x+8x^{2}+4x^{3})$ & \qquad$-3007x^{5}-17636x^{6}-27496x^{7}$ \\
    & $\times(1+2x+x^{2}-8x^{3})$ & \qquad$-18000x^{8}-3960x^{9}\big)$ \\
    \hline
    $119$ & $(1+2x+3x^{2}+6x^{3}+5x^{4})$ & $-x\big(2+17x+66x^{2}+26x^{3}-190x^{4}$ \\
    & $\times(1+2x+3x^{2}+6x^{3}+4x^{4}+x^{5})$ & \qquad$-1077x^{5}-3578x^{6}-7492x^{7}-12836x^{8}$ \\
    & $\times(1-2x+3x^{2}-6x^{3}-7x^{5})$ & \qquad$-17746x^{9}-18692x^{10}-15617x^{11}$ \\
    & & \qquad$-7644x^{12}-1680x^{13}\big)$ \\
    \hline
\end{longtable}






\vfill
\pagebreak[4]

\begin{table}[H]
\caption{Positive rational and quadratic irrational series for  level~1.}
\label{T:1.1}
{\renewcommand{\arraystretch}{1.4}
\resizebox{\columnwidth}{!}{%
\begin{tabular}{|cc|c|c|l|c|l|}
\hline
$N$ & $\overline{N}$ & $X_1(e^{-2\pi\sqrt{N}})$ & $\lambda_1$ & Reference \\
\hline
1 d && $\frac{1}{2^6\cdot3^3}$ & undefined & \\
2 && $\frac{1}{2^6\cdot5^3}$ & $\frac{3}{2^2\cdot7}$& \cite[(4.4)]{BB2010}, \cite[p. 187]{borwein-book}\\
3 && $\frac{1}{2^4\cdot3^3\cdot5^3}$ & $\frac{1}{11}$ & \cite[(8.5)]{BB2010}, \cite[(33)]{ramanujan_pi} \\
4 && $\frac{1}{2^3\cdot 3^3\cdot11^3}$ & $\frac{5}{3^2\cdot7}$ & 
\cite[(5.5)]{BB2010}, \cite[p. 187]{borwein-book}\\
7 && $\frac{1}{3^3\cdot 5^3\cdot17^3}$ & $\frac{8}{7\cdot19}$ & \cite[(9.6)]{BB2010}, 
\cite[p. 187]{borwein-book}, \cite[(34)]{ramanujan_pi} \\
\hline
5 & Tbl.~\ref{T:1.2} & $\frac{1}{2^3\cdot(25 + 13\sqrt5)^3}$ & $\frac{93-15\sqrt{5}}{2^2\cdot11\cdot19}$ &
\cite[(43)]{bagis2} \\
6 & $\frac32$ & $\frac{1}{2^6\cdot3^3\cdot(1399 + 988\sqrt2)}$ & $\frac{25 - 5\sqrt2}{2^2\cdot3\cdot23}$ & \cite{aldawoud}, \cite[p. 372]{borweinrr} \\
\hline
\end{tabular}}}
\end{table}

\begin{table}[H]
\caption{Negative rational and quadratic irrational series for  level~1.}
\label{T:1.2}
{\renewcommand{\arraystretch}{1.8}
\resizebox{\columnwidth}{!}{%
\begin{tabular}{|cc|c|c|l|c|l|}
\hline
$N$ &$\overline{N}$ & $X_1(-e^{-\pi\sqrt{N}})$ & $\lambda_1$ & Reference \\
\hline
7 && $\frac{-1}{3^3\cdot5^3}$ & $\frac{8}{3^2\cdot7}$& \cite[(9.7)]{BB2010},
\cite[(1.4)]{chud}\\
11 && $\frac{-1}{2^{15}}$ & $\frac{15}{2\cdot7\cdot11}$ & \cite[(1.4)]{chud}\\
19 && $\frac{-1}{2^{15}\cdot 3^3}$ & $\frac{25}{2\cdot3^2\cdot19}$ & \cite[(1.4)]{chud}\\
27 && $\frac{-1}{2^{15}\cdot3\cdot 5^3}$ & $\frac{31}{2\cdot11\cdot23}$ & \cite[p. 371]{borweinrr}\\
43 && $\frac{-1}{2^{18}\cdot3^3\cdot 5^3}$ & $\frac{263}{2\cdot3^2\cdot7\cdot43}$ &\cite[(1.4)]{chud}\\
67 && $\frac{-1}{2^{15}\cdot3^3\cdot 5^3\cdot11^3}$ & $\frac{10177}{2\cdot3^2\cdot7\cdot31\cdot67}$ & \cite[(1.4)]{chud}\\
163 && $\frac{-1}{2^{18}\cdot3^3\cdot 5^3\cdot23^3\cdot29^3}$ & $\frac{13591409}{2\cdot3^2\cdot7\cdot11\cdot19\cdot127\cdot163}$ & \cite[(1.4)]{chud}\\
\hline\end{tabular}}}
\end{table}

\pagebreak[4]

\begin{table}[H]
\caption{Positive rational and quadratic irrational series for level~2.}
\label{T:2.1}
{\renewcommand{\arraystretch}{1.4}
\resizebox{\columnwidth}{!}{%
{\begin{tabular}{|cc|c|c|l|}
\hline
$N$ & $\overline{N}$ & $X_2(e^{-2\pi\sqrt{N/2}})$ & $\lambda_2$ & Reference \\
\hline
$1$ && $\frac{1}{2^8}$ & undefined &  \\ 
2 && $\frac{1}{2^3 \cdot 3^4}$ & $\frac{1}{7}$ & 
\cite[(5.4)]{BB2010}, \cite[(4.6)]{BerndtChanLiaw2001}\\ 
3 && $\frac1{2^8 \cdot 3^2}$ & $\frac18$ & \cite[(6.4)]{BB2010}, \cite[(40)]{ramanujan_pi} \\ 
5 && $\frac{1}{2^8 \cdot 3^4}$ & $\frac{1}{10}$ & 
\cite[(11.3)]{BB2010}, \cite[(41)]{ramanujan_pi} \\ 
9 && $\frac1{2^8 \cdot 7^4}$ & $\frac3{40}$ & 
\cite[(16.3)]{BB2010}, \cite[(42)]{ramanujan_pi} \\ 
11& & $\frac{1}{2^8 \cdot 3^4\cdot 11^2}$ & $\frac{19}{280}$ & 
\cite[(17.2)]{BB2010}, \cite[(43)]{ramanujan_pi} \\ 
29 && $\frac{1}{2^8 \cdot 3^8\cdot 11^4}$ & $\frac{1103}{26390}$ & \cite[(44)]{ramanujan_pi} \\
\hline
\end{tabular}}}}
\end{table}

\begin{table}[H]
\caption{Negative rational and quadratic irrational series for level~2.}
\label{T:2.2}
{\renewcommand{\arraystretch}{1.8}
\resizebox{\columnwidth}{!}{%
{\begin{tabular}{|cc|c|c|l|}
\hline
$N$ & $\overline{N}$ & $X_2(-e^{-\pi\sqrt{N}})$ & $\lambda_2$ & Reference \\
\hline
3 d && $\frac{-1}{2^4\cdot3^2}$ & $\frac15$ & \\
5 && $\frac{-1}{2^{10}}$ & $\frac{3}{20}$ & \cite[(7.3)]{BB2010}, \cite[(35)]{ramanujan_pi}  \\ 
7 && $\frac{-1}{3^{4}\cdot7^2}$ & $\frac{8}{65}$ & 
\cite[(9.5)]{BB2010}, \cite[(4.8)]{BerndtChanLiaw2001} \\ 
9 && $\frac{-1}{2^{12} \cdot 3}$ & $\frac3{28}$ & 
\cite[(10.3)]{BB2010}, \cite[(36)]{ramanujan_pi} \\ 
13 && $\frac{-1}{2^{10}\cdot 3^{4}}$ & $\frac{23}{260}$ & 
\cite[(12.3)]{BB2010}, \cite[(37)]{ramanujan_pi} \\ 
25 && $\frac{-1}{2^{14}\cdot 3^{4}\cdot5}$ & $\frac{41}{644}$ & 
\cite[(18.1)]{BB2010}, \cite[(38)]{ramanujan_pi} \\ 
37 && $\frac{-1}{2^{10}\cdot 3^{4}\cdot7^4}$ & $\frac{1123}{21460}$ &\cite[(39)]{ramanujan_pi} \\
\hline
\end{tabular}}}
}\end{table}

\begin{table}[H]
\caption{Positive rational and quadratic irrational series for level~3.}
\label{T:3.1}
{\renewcommand{\arraystretch}{1.4}
{\begin{tabular}{|cc|c|c|l|}
\hline
$N$ & $\overline{N}$ & $X_3(e^{-2\pi\sqrt{N/3}})$ & $\lambda_3$ & Reference \\
\hline
$1$ && $\frac{1}{2^2\cdot3^3}$ &  undefined & \\
2 && $\frac{1}{2^3\cdot3^3}$ & $\frac16$& \cite[p. 190]{borwein-book},
\cite{CL2000}\\
4 && $\frac{1}{2\cdot3^6}$ & $\frac{2}{15}$ & \cite[p. 190]{borwein-book},
\cite[(31)]{ramanujan_pi} \\
5 && $\frac{1}{3^3\cdot 5^3}$ & $\frac{4}{33}$ &  \cite[p. 190]{borwein-book},
\cite{CL2000}, \cite[(32)]{ramanujan_pi} \\
\hline
3 & Tbl.~\ref{T:3.2} & $\frac{1}{2^3\cdot3\cdot(7\sqrt{3}+12)}=\frac{1}{2\cdot3^{3/2}\cdot(1+\sqrt{3})^4}$ & $\frac{5-\sqrt{3}}{2\cdot11}$ & \cite{aldawoud},
\cite[(5.28)]{BB2008}, \cite[p. 190]{borwein-book}\\
6 & $\frac32$ & $\frac{1}{2^3\cdot(13 + 7\sqrt6)^2}$ & $\frac{33-3\sqrt{6}}{2\cdot5\cdot23}$ & \cite{aldawoud}, \cite[p. 190]{borwein-book}, \cite{Sun} \\
7 & Tbl.~\ref{T:3.2} & $\frac{1}{2^3\cdot3^3\cdot(34 + 13\sqrt7)}$ & $\frac{7-\sqrt7}{2\cdot3\cdot7}$ & \cite{aldawoud},
\cite{CL2000} \\
\hline
\end{tabular}}}
\end{table}

\begin{table}[H]
\caption{Negative rational and quadratic irrational series for level~3.}
\label{T:3.2}
{\renewcommand{\arraystretch}{1.4}
{\begin{tabular}{|cc|c|c|l|}
\hline
$N$ & $\overline{N}$ & $X_3(-e^{-\pi\sqrt{N/3}})$ & $\lambda_3$ & Reference \\
\hline
5 d && $\frac{-1}{3^3}$ & $\frac{4}{15}$ & \\
9 && $\frac{-1}{2^6\cdot3}$ & $\frac15$& \cite[(1.14)]{clt}\\
17 && $\frac{-1}{2^6\cdot3^3}$ & $\frac{7}{3\cdot17}$ & \cite[(1.15)]{clt}, \cite{bcs} \\
25 && $\frac{-1}{2^6\cdot3^3\cdot 5}$ & $\frac{1}{3^2}$ & \cite[(1.16)]{clt} \\
41 && $\frac{-1}{2^{12}\cdot3^3}$ & $\frac{53}{3\cdot5\cdot41}$& \cite[(1.17)]{clt}\\
49 && $\frac{-1}{2^6\cdot3^6\cdot7}$ & $\frac{13}{3\cdot5\cdot11}$ & \cite[(1.18)]{clt} \\
89 && $\frac{-1}{2^6\cdot3^3\cdot 5^6}$ & $\frac{827}{3\cdot53\cdot89}$ & \cite[(1.19)]{clt} \\
\hline
\end{tabular}}}
\end{table}

\begin{table}[H]
\caption{Rational and quadratic irrational series for level~4.
In Ramanujan's notation~\cite[p. 172]{borwein-book},~\cite[(1)]{ramanujan_pi},
$$
X_4(e^{-\pi\sqrt{N}}) = \frac{1}{64\,G_N^{24}} \quad\text{where}\quad
G_N = \frac{1}{(2q)^{1/4}}\left.\prod_{j=1}^\infty (1+q^{2j-1})\right|_{q=e^{-\pi\sqrt{N}}}
$$
and
$$
X_4(-e^{-\pi\sqrt{N}}) = \frac{-1}{64\,g_N^{24}} \quad\text{where}\quad
g_N = \frac{1}{(2q)^{1/4}}\left.\prod_{j=1}^\infty (1-q^{2j-1})\right|_{q=e^{-\pi\sqrt{N}}.}
$$
}
\label{T:4.1}
{\renewcommand{\arraystretch}{1.4}
{
\begin{tabular}{|cc|c|c|l|}
\hline
$N$ & $\overline{N}$ & $X_4(e^{-\pi\sqrt{N}})$ & $\lambda_4$ & Reference \\
\hline
$1$ && $\frac{1}{2^6}$ & undefined & \\
3 && $\frac{1}{2^8}$ & $\frac16$& \cite[(8.1)]{BB2010}, \cite[(28)]{ramanujan_pi}\\
7 && $\frac{1}{2^{12}}$ & $\frac{5}{42}$ & \cite[(9.1)]{BB2010}, \cite[(29)]{ramanujan_pi} \\
\hline
2 & below & $\frac{1}{2^{3}\cdot(1+\sqrt{2})^3}$ & $\frac27-\frac{\sqrt{2}}{14}$ & \cite{aldawoud}, \cite[(4.1)]{BB2010} \\
4 & below & $\frac{1}{2^{3/2}\cdot(1+\sqrt{2})^6}$ & $\frac27-\frac{2\sqrt{2}}{21}$ & \cite{aldawoud},
\cite[(5.1)]{BB2010}  \\
5 & d & $\frac{1}{2^6\cdot(2+\sqrt{5})^2}$ & $\frac14-\frac{\sqrt{5}}{20}$ & \cite{aldawoud}, \cite[(7.1)]{BB2010}, \cite[p. 172]{borwein-book} \\
9 & d & $\frac{1}{2^2\cdot(1 + \sqrt3)^8}$ & $\frac14-\frac{\sqrt{3}}{12}$ & \cite{aldawoud}, \cite[(10.1)]{BB2010}, \cite[pp. 172, 187]{borwein-book} \\
\hline

\end{tabular}}}
\end{table}

\begin{table}[H]
\caption{Positive rational and quadratic irrational series for level~5.}
\label{T:5.1}
{\renewcommand{\arraystretch}{1.4}
{\begin{tabular}{|cc|c|c|c|c|l|}
\hline
$N$ & $\overline{N}$ & $X_5(e^{-2\pi\sqrt{N/5}})$ & $\lambda_5$ & Reference \\
\hline
2 && $\frac{1}{2^3\cdot3^2}$ & $\frac15$& \cite{chancooper}\\
3 && $\frac{1}{3\cdot7^2}$ & $\frac{2}{11}$ & \cite{chancooper} \\
\hline
1 d & Tbl.~\ref{T:5.2} & $\frac{2^3}{(1+\sqrt{5})^5}$ & undefined & \\
4 & d & $\frac{2}{(13+5\sqrt{5})^2}$ & $\frac{3}{10} - \frac{3\sqrt{5}}{50}$ & \\
5 & Tbl.~\ref{T:5.2} & $\frac{4}{9\sqrt{5}(1+\sqrt{5})^4}$ & 
$\frac{52}{209} - \frac{9\sqrt5}{209}$ & \\
6 & $\frac32$ & $\frac{1}{3(7+5\sqrt{5})^2}$ & $\frac{7}{40} - \frac{3\sqrt5}{200}$ & \\
8 & Tbl.~\ref{T:5.2} & $\frac{1}{2\cdot3^2\cdot(79+25\sqrt{10})}$ & 
$\frac{43}{155} - \frac{3\sqrt{10}}{62}$ & \\
9 & Tbl.~\ref{T:5.2} & $\frac{1}{2^4\cdot(142+65\sqrt{5})}$ & $\frac{21}{110} - \frac{9\sqrt5}{275}$ & \\
11 & d & $\frac{2}{3^2\cdot(1241 + 555\sqrt5)}$ & $\frac{317}{1045} - \frac{459\sqrt5}{5225}$ & \\
12 & $\frac43$ & $\frac{2}{(87+25\sqrt{15})^2}$ & $\frac{112}{715} - \frac{2\sqrt{15}}{143}$ & \\
14 & $\frac72$ & $\frac{1}{7(39+15\sqrt5)^2}$ & $\frac{1831}{13640} - \frac{1197\sqrt5}{68200}$ & \\
17 & Tbl.~\ref{T:5.2} & $\frac{1}{2\cdot3^2\cdot(2979 + 325\sqrt{85})}$ & $\frac{752}{5605} - \frac{99\sqrt{85}}{19057}$ & \\
26 & $\frac{13}{2}$ & $\frac{1}{2^3\cdot3^2\cdot(74+35\sqrt5)^2}$ & $\frac{110771}{1073215} - \frac{80208\sqrt5}{5366075}$ & \\
38 & $\frac{19}{2}$ & $\frac{1}{2^3\cdot3^2\cdot19\cdot(77 + 25\sqrt{10})^2}$ & $\frac{2436169}{28249835} - \frac{101859\sqrt{10}}{11299934}$ & \\
\hline
\end{tabular}}}
\end{table}

\begin{table}[H]
\caption{Negative rational and quadratic irrational series for level~5.}
\label{T:5.2}
{\renewcommand{\arraystretch}{1.4}
{\begin{tabular}{|cc|c|c|c|c|l|}
\hline
$N$ & $\overline{N}$ & $X_5(-e^{-\pi\sqrt{N/5}})$ & $\lambda_5$ & Reference \\
\hline
3 d &  & $\frac{-1}3$ & $\frac25$ & \\
7 d &  & $\frac{-1}{2^2\cdot7}$ & $\frac{3}{10}$ & \\
23 && $\frac{-1}{2^2\cdot3^2\cdot23}$ & $\frac{29}{190}$& \cite{chancooper}\\
47 && $\frac{-1}{2^2\cdot3^4\cdot47}$ & $\frac{71}{682}$ & \cite{chancooper} \\
\hline
1 d & Tbl.~\ref{T:5.1} & $\frac{-2^3}{(\sqrt{5}-1)^5}$ & undefined & \\
2 d & Tbl.~\ref{T:5.1} & $\frac{-1}{2\cdot3^2\cdot(25\sqrt{10}-79)}$ & 
$\frac{43}{155} + \frac{3\sqrt{10}}{62}$ & \\
5 d & Tbl.~\ref{T:5.1} & $\frac{-4}{9\sqrt{5}(\sqrt{5}-1)^4}$ & 
$\frac{52}{209} + \frac{9\sqrt5}{209}$ & \\
\hline
\end{tabular}}}
\end{table}


\begin{table}[H]
\caption{Positive rational and quadratic irrational series for level~6.}
\label{T:6.1}
{\renewcommand{\arraystretch}{1.4}
{\begin{tabular}{|cc|c|c|c|c|l|}
\hline
$N$ & $\overline{N}$ & $X_6(e^{-2\pi\sqrt{N/6}})$ & $\lambda_6$ & Reference \\
\hline
1 d && $\frac{1}{2^5}$ & undefined & \\
2 && $\frac{1}{2\cdot5^2}$ & $\frac29$& \cite{ctyz}\\
3 && $\frac{1}{2^5\cdot3}$ & $\frac{1}{5}$ & \cite{ctyz} \\
5 && $\frac{1}{2^6\cdot5}$ & $\frac{1}{6}$ & \cite{ctyz} \\
7 && $\frac{1}{2^7\cdot7}$ & $\frac{13}{90}$ & \cite{ctyz} \\
13 && $\frac{1}{2^5\cdot5^2\cdot13}$ & $\frac{11}{102}$ & \cite{ctyz} \\
17 && $\frac{1}{2^5\cdot5^2\cdot7^2}$ & $\frac{53}{561}$ & \cite{ctyz} \\
\hline
\end{tabular}}}
\end{table}

\begin{table}[H]
\caption{Negative rational and quadratic irrational series for level~6.}
\label{T:6.2}
{\renewcommand{\arraystretch}{1.4}
{\begin{tabular}{|cc|c|c|c|c|l|}
\hline
$N$ & $\overline{N}$ & $X_6(-e^{-\pi\sqrt{N/3}})$ & $\lambda_6$ & Reference \\
\hline
1 d && $\frac{-1}{2^2}$ & undefined & \\
3 d && $\frac{-1}{2^4}$ & $\frac13$ & \\
5 && $\frac{-1}{7^2}$ & $\frac{4}{15}$ & \cite{chancooper} \\
7 && $\frac{-1}{2^4\cdot7}$ & $\frac{2}{9}$ & \cite{chancooper} \\
11 && $\frac{-1}{2^4\cdot5^2}$ & $\frac{17}{99}$ & \cite{chancooper} \\
19 && $\frac{-1}{2^4\cdot13^2}$ & $\frac{109}{855}$ & \cite{chancooper} \\
31 && $\frac{-1}{2^4\cdot7^2\cdot31}$ & $\frac{58}{585}$ & \cite{chancooper} \\
\hline
\end{tabular}}}
\end{table}


\begin{table}[H]
\caption{Rational and quadratic irrational series for level~7.}
\label{T:7.1}
{\renewcommand{\arraystretch}{1.4}
{\begin{tabular}{|cc|c|c|c|c|l|}
\hline
$N$ & $\overline{N}$ & $X_7(e^{-2\pi\sqrt{N/7}})$ & $\lambda_7$ & Reference \\
\hline
1 d && $\frac{1}{3^3}$ & undefined & \\
4 && $\frac{1}{5^3}$ & $\frac4{21}$& \cite{cooperlevel7}\\
\hline
2 & d & $\frac{1}{2\cdot(10+7\sqrt2)}$ & $\frac37 - \frac{\sqrt2}{7}$ & \\
3 & below & $\frac{1}{2\cdot(17 + 7\sqrt7)}$ & 
$\frac13 - \frac{\sqrt7}{21}$ & \\
6 & $\frac32$ & $\frac{1}{2\cdot(94 + 21\sqrt{14})}$ & $\frac15 - \frac{\sqrt{14}}{105}$ & \\
9 & d & $\frac{2}{3\cdot(417 + 91\sqrt{21})}$ & 
$\frac25 - \frac{2\sqrt{21}}{35}$ & \\
10 & $\frac52$ & $\frac{1}{2\cdot(472 + 315\sqrt2)}$ & 
$\frac{27}{155} - \frac{14\sqrt2}{465}$ & \\
\hline

\end{tabular}}}
\end{table}


\begin{table}[H]
\caption{Rational and quadratic irrational series for level~8.}
\label{T:8.1}
{\renewcommand{\arraystretch}{1.4}
{\begin{tabular}{|cc|c|c|c|c|l|}
\hline
$N$ & $\overline{N}$ & $X_8(e^{-2\pi\sqrt{N/8}})$ & $\lambda_8$ & Reference \\
\hline
1 d & d & $\frac{1}{2^2\cdot(1+\sqrt{2})^2}$ & undefined &  \\
2 & below & $\frac{1}{2^{5/2}\cdot(1 + \sqrt2)^2}$ & $\frac37 - \frac{\sqrt2}{7}$ &  \\
3 & d & $\frac{1}{(1+\sqrt3)^4}$ & $\frac12 - \frac{\sqrt3}6$ & \\
5 & d & $\frac{1}{2^2\cdot(3+\sqrt{10})^2}$ & $\frac12 - \frac{\sqrt{10}}{10}$ & \cite{chancooper} \\
9 & d & $\frac{1}{2^2\cdot(7+5\sqrt{2})^2}$ & $\frac12 - \frac{\sqrt2}4$ & \cite{chancooper} \\
11 & d & $\frac{1}{(3+\sqrt{11})^4}$ & $\frac12 - \frac{17\sqrt{11}}{154}$ & \cite{chancooper} \\
29 & d & $\frac{1}{2^2\cdot(99+13\sqrt{58})^2}$ & $\frac12 - \frac{111\sqrt{58}}{2030}$ & \cite{chancooper} \\
\hline
\hline
$N$ & $\overline{N}$ & $X_8(-e^{-2\pi\sqrt{N/8}})$ & $\lambda_8$ & Reference \\
\hline
1 d &  & $\frac{-1}{2^2}$ & $\frac12$ &  \\
\hline
$\frac12$ d & above & $\frac{-1}{2^{5/2}\cdot(\sqrt2-1)^2}$ & 
$\frac37 + \frac{\sqrt2}{7}$ &  \\
2 d & d & $\frac{-1}{2^{3/2}\cdot(1 + \sqrt2)^2}$ & $\frac47 - \frac{\sqrt2}{7}$ &  \\
3  & $\frac13$ d & $\frac{-1}{2^2\cdot(\sqrt2+\sqrt3)^2}$ & $\frac12 - \frac{\sqrt6}{12}$ & \\
5 & $\frac15$ d & $\frac{-1}{2^2\cdot(1+\sqrt{2})^4}$ & $\frac12 - \frac{\sqrt{2}}{5}$ & \cite{chancooper} \\
9 & $\frac19$ d & $\frac{1}{2^2\cdot(2+\sqrt{3})^4}$ & $\frac12 - \frac{\sqrt3}5$ & \cite{chancooper} \\
11 & $\frac{1}{11}$ d & $\frac{-1}{2^2\cdot(7\sqrt{2}+3\sqrt{11})^2}$ & $\frac12 - \frac{17\sqrt{22}}{220}$ & \cite{chancooper} \\
29 & $\frac{1}{29}$ d & $\frac{-1}{2^2\cdot(7+5\sqrt{2})^4}$ & $\frac12 - \frac{111\sqrt{2}}{377}$ & \cite{chancooper} \\
\hline
\end{tabular}}}
\end{table}


\begin{table}[H]
\caption{Rational and quadratic irrational series for level~9.}
\label{T:9.1}
{\renewcommand{\arraystretch}{1.4}
{\begin{tabular}{|cc|c|c|c|c|l|}
\hline
$N$ & $\overline{N}$ & $X_9(e^{-2\pi\sqrt{N/9}})$ & $\lambda_9$ & Reference \\
\hline
2 && $\frac{1}{3^3}$ & $\frac14$ & \\
\hline
1 d & below & $\frac{1}{3(3+2\sqrt{3})}$ & undefined & \\
4 & d & $\frac{1}{3(12+7\sqrt{3})}$ & $\frac12 - \frac{\sqrt3}{6}$ & \\
5 & below & $\frac{1}{3^3\cdot(2+\sqrt{5})}$ & $\frac{7}{22} - \frac{3\sqrt5}{55}$ & \\
7 & d & $\frac{2}{3\cdot(87 + 19\sqrt{21})}$ & $\frac12 - \frac{\sqrt{21}}{14}$ & \\
8 & below & $\frac{1}{3^3\cdot(7 + 5\sqrt2)}$ & $\frac8{23} - \frac{3\sqrt2}{23}$ & \\
\hline\hline
$N$ & $\overline{N}$ & $X_9(-e^{-\pi\sqrt{N/9}})$ & $\lambda_9$ & Reference \\
\hline
3 d && $\frac{-1}{3}$ & $\frac12$ & \cite{chancooper}\\
11 && $\frac{-1}{3^3}$ & $\frac{7}{22}$ & \cite{chancooper} \\
\hline
1 d & above & $\frac{-1}{3(2\sqrt{3}-3)}$ & undefined & \\
2 d& above & $\frac{-1}{3^3\cdot(5\sqrt2-7)}$ & $\frac8{23} + \frac{3\sqrt2}{23}$ & \\
5 d& above & $\frac{-1}{3^3\cdot(\sqrt{5}-2)}$ & $\frac{7}{22} + \frac{3\sqrt5}{55}$ & \\
7 d& d & $\frac{-2}{3\cdot(3+\sqrt{21})}$ & $\frac12 - \frac{\sqrt{21}}{42}$ & \\
19 & d & $\frac{-1}{3\cdot(15 + 2\sqrt{57})}$ & $\frac12 - \frac{2\sqrt{57}}{57}$ & \\
35 & $\frac75$ d & $\frac{-1}{3^3\cdot(9 + 4\sqrt5)}$ & $\frac{5}{14} - \frac{3\sqrt5}{35}$ & \\
43 & d & $\frac{-1}{3\cdot(159 + 14\sqrt{129})}$ & $\frac12 - \frac{4\sqrt{129}}{129}$ & \\
67 & d & $\frac{-1}{3\cdot(879 + 62\sqrt{201})}$ & 
$\frac12 - \frac{38\sqrt{201}}{1407}$ & \\
163 & d & $\frac{-1}{3\cdot(106719 + 4826\sqrt{489})}$ & $\frac12 - \frac{724\sqrt{489}}{37653}$ & \\
\hline

\end{tabular}}}
\end{table}


\begin{table}[H]
\caption{Rational and quadratic irrational series for level~10.}
\label{T:10.1}
{\renewcommand{\arraystretch}{1.4}
{\begin{tabular}{|cc|c|c|c|c|l|}
\hline
$N$ & $\overline{N}$ & $X_{10}(e^{-2\pi\sqrt{N/10}})$ & $\lambda_{10}$ & Reference \\
\hline
1 d && $\frac{1}{2^4}$ & undefined & \\
3 && $\frac{1}{2^2\cdot3^2}$ & $\frac14$& \cite[(2.16)]{ctyz}\\
7 && $\frac{1}{2^2\cdot7^2}$ & $\frac{11}{60}$& \cite{cooper10a}\\
13 && $\frac{1}{2^4\cdot3^4}$ & $\frac9{65}$& \cite{cooper10a}\\
19 && $\frac{1}{2^4\cdot19^2}$ & $\frac{47}{408}$& \cite{cooper10a}\\
\hline
2 & d & $\frac{1}{11+5\sqrt{5}}$ & $\frac12-\frac{\sqrt{5}}{10}$ & \\
4 & below & $\frac{1}{2\cdot(13+5\sqrt{10})}$ & $\frac13 - \frac{\sqrt{10}}{30}$ & \\
\hline
\end{tabular}}}
\end{table}


\begin{table}[H]
\caption{Rational and quadratic irrational series for level~11.}
\label{T:11.1}
{\renewcommand{\arraystretch}{1.4}
{\begin{tabular}{|cc|c|c|c|c|l|}
\hline
$N$ & $\overline{N}$ & $X_{11}(e^{-2\pi\sqrt{N/11}})$ & $\lambda_{11}$ & Reference \\
\hline
2 && $\frac{1}{22}$ & $\frac{3}{14}$ & \cite[(6.3)]{coopergeye}\\
\hline
3 & below & $\frac{1}{2\cdot(8 + 5\sqrt3)}$ & 
$\frac12 - \frac{\sqrt3}{6}$ & \cite[Tbl. 7]{huberActa} \\
5 & d & $\frac{2^4}{11\cdot(1 + \sqrt5)^4}$ & 
$\frac{10}{19} - \frac{14\sqrt5}{95}$ & \cite[Tbl. 7]{huberActa}\\
8 & below & $\frac{1}{109 + 33\sqrt{11}}$ & 
$\frac{83}{154} - \frac{17\sqrt{11}}{154}$ & \cite[Tbl. 7]{huberActa}\\
23 & below & $\frac{1}{2\cdot(2186 + 465\sqrt{23})}$ & 
$\frac{1866443}{10442894} - \frac{3457747\sqrt{23}}{240186562}$ & \cite[Tbl. 7]{huberActa}\\
\hline
\hline
$N$ & $\overline{N}$ & $X_{11}(-e^{-\pi\sqrt{N/11}})$ & $\lambda_{11}$ & Reference \\
\hline
9 d && $\frac{-1}{12}$ & $\frac5{11}$ & \cite{coopergeye}\\
\hline
\end{tabular}}}
\end{table}


\begin{table}[H]
\caption{Rational and quadratic irrational series for level~12.}
\label{T:12.1}
{\renewcommand{\arraystretch}{1.4}
{\begin{tabular}{|cc|c|c|c|l|}
\hline
$N$ & $\overline{N}$ & $X_{12}(e^{-2\pi\sqrt{N/12}})$ & $\lambda_{12}$ & Reference \\
\hline
1 d && $\frac{1}{2^4}$ & undefined & \\
5 && $\frac{1}{2^6}$ & $\frac15$& \cite[(1.3)]{domb}\\
\hline
2 & below & $\frac{1}{2\cdot(5+3\sqrt{3})}$ & 
$\frac12-\frac{\sqrt{3}}{6}$ & \cite[(3.11)]{rogers} \\
3 & d & $\frac{1}{2^2\cdot(1+\sqrt{3})^2}$ & $\frac12 - \frac{\sqrt3}{6}$ & \\
4 & below & $\frac{1}{22+9\sqrt{6}}$ & $\frac8{15} - \frac{2\sqrt6}{15}$ & \\
7 & d & $\frac{1}{2^2\cdot(3+\sqrt{7})^2}$ & $\frac12 - \frac{5\sqrt7}{42}$ & \\
11 & d & $\frac{1}{2^2\cdot(5+3\sqrt{3})^2}$ & $\frac12 - \frac{13\sqrt3}{66}$ & \\
\hline
\end{tabular}}}
\end{table}


\begin{table}[H]
\caption{Rational and quadratic irrational series for level~13.}
\label{T:13.1}
{\renewcommand{\arraystretch}{1.4}
{\begin{tabular}{|cc|c|c|c|c|l|}
\hline
$N$ & $\overline{N}$ & $X_{13}(e^{-2\pi\sqrt{N/13}})$ & $\lambda_{13}$ & Reference \\
\hline
1 d & below & $\frac{1}{5+2\sqrt{13}}$ & undefined &  \\
3 & d & $\frac{2}{(23 + 7\sqrt{13})}$ & 
$\frac{7}{12} - \frac{\sqrt{13}}{12}$ & \cite{cooperye13} \\
4 & d & $\frac{1}{18 + 5\sqrt{13}}$ & 
$\frac{17}{24} - \frac{\sqrt{13}}{8}$ & \cite{cooperye13} \\
6 & $\frac32$ & $\frac{1}{2\cdot(22 + 3\sqrt{26})}$ &
$\frac{139}{506} - \frac{8\sqrt{26}}{759}$ & \cite{cooperye13} \\
10 & $\frac52$ & $\frac{1}{135 + 32\sqrt{13}}$ & 
$\frac{19}{80} - \frac{\sqrt{13}}{60}$ & \cite{cooperye13} \\
\hline
\hline
$N$ & $\overline{N}$ & $X_{13}(-e^{-\pi\sqrt{N/13}})$ & $\lambda_{13}$ & Reference \\
\hline
7 d &  & $\frac{-1}{2^3}$ & $\frac{17}{42}$ &  \\
31 & & $\frac{-1}{5^3}$ & $\frac{79}{372}$ & \cite{cooperye13} \\
\hline
1 d & above & $\frac{-1}{2\sqrt{13}-5}$ & undefined &  \\
3 d & d & $\frac{-2}{3+\sqrt{13}}$ & $\frac34 - \frac{\sqrt{13}}{12}$ &  \\
15 & $\frac53$ d & $\frac{-2^2}{(\sqrt{13} + 3\sqrt5)^2}$ & $\frac38 - \frac{\sqrt{65}}{120}$ & \cite{cooperye13} \\
55 & $\frac{11}{5}$ d & $\frac{-1}{8\cdot(40 + 11\sqrt{13})}$ &
$\frac{299}{990} - \frac{4\sqrt{13}}{99}$ & \cite{cooperye13} \\
79 & d & $\frac{-1}{2^6\cdot(18 + 5\sqrt{13})}$ & 
$\frac{3265}{6794} - \frac{992\sqrt{13}}{10191}$ & \cite{cooperye13} \\
\hline
\end{tabular}}}
\end{table}


\begin{table}[H]
\caption{Rational and quadratic irrational series for level~14.}
\label{T:14.1}
{\renewcommand{\arraystretch}{1.4}
{\begin{tabular}{|cc|c|c|c|c|l|}
\hline
$N$ & $\overline{N}$ & $X_{14}(e^{-2\pi\sqrt{N/14}})$ & $\lambda_{14}$ & Reference \\
\hline
2 && $\frac{1}{2\cdot3^2}$ & $\frac17$& \cite{cooperye14}\\
3 && $\frac{1}{5^2}$ & $\frac{1}{6}$& \cite{cooperye14}\\
5 && $\frac{1}{7^2}$ & $\frac{11}{60}$& \cite{cooperye14}\\
\hline
1 d && $\frac{1}{(1+2\sqrt{2})^2}$ & undefined & \\
8 & below & $\frac{1}{58 + 45\sqrt2}$ & $\frac{89}{217} - \frac{107\sqrt2}{651}$ & \\
9 & d & $\frac{1}{(7 + 4\sqrt2)^2}$ & $\frac{19}{28} - \frac{5\sqrt2}{14}$ & \\
11 & d & $\frac1{(7 + 2\sqrt{22})^2}$ & 
$\frac{5}{21} - \frac{29\sqrt{22}}{1848}$ & \\
\hline
\end{tabular}}}
\end{table}


\begin{table}[H]
\caption{Rational and quadratic irrational series for level~15.}
\label{T:15.1}
{\renewcommand{\arraystretch}{1.4}
{\begin{tabular}{|cc|c|c|c|c|l|}
\hline
$N$ & $\overline{N}$ & $X_{15}(e^{-2\pi\sqrt{N/15}})$ & $\lambda_{15}$ & Reference \\
\hline
1 d && $\frac{1}{2^2\cdot3}$ & undefined & \\
2 && $\frac{1}{3\cdot5}$ & $\frac14$& \cite{cooperye14}\\
4 && $\frac{1}{2\cdot3\cdot5}$ & $\frac{3}{13}$& \cite{cooperye14}\\
\hline
3 & below & $\frac{1}{5\cdot(2 + \sqrt5)}$ & $\frac6{11} - \frac{3\sqrt5}{22}$ & \\
7 & below & $\frac{1}{3^2\cdot(4 + \sqrt{21})}$ & $\frac25 - \frac{3\sqrt{21}}{70}$ & \\
8 & below & $\frac{1}{3\cdot(17 + 7\sqrt6)}$ & $\frac{71}{125} - \frac{19\sqrt6}{125}$ & \\
11 & below & $\frac{1}{15\cdot(7 + 2\sqrt{15})}$ & $\frac{1681}{5236} - \frac{197\sqrt{15}}{5236}$ & \\
\hline
\end{tabular}}}
\end{table}


\begin{table}[H]
\caption{Rational and quadratic irrational series for level~16.}
\label{T:16.1}
{\renewcommand{\arraystretch}{1.4}
{\begin{tabular}{|cc|c|c|c|c|l|}
\hline
$N$ & $\overline{N}$ & $X_{16}(e^{-2\pi\sqrt{N/16}})$ & $\lambda_{16}$ & Reference \\
\hline
1 d & d & $\frac{1}{2(1+\sqrt{2})^2}$ & undefined & \\
3 & d & $\frac{1}{(2+\sqrt{6})^2}$ & $1 -\frac{\sqrt6}{3}$ & \cite{ye16} \\
7 & d & $\frac{1}{(4+3\sqrt{2})^2}$ & $1-\frac{4\sqrt{2}}7$ & \cite{ye16}\\
\hline
\hline
$N$ & $\overline{N}$ & $X_{16}(-e^{-2\pi\sqrt{N/16}})$ & $\lambda_{16}$ & Reference \\
\hline
1 d &  & $-\frac12$ & $1$ & \\
3 & $\frac13$ d & $\frac{-1}{2(1+\sqrt{2})^2}$ & $1 -\frac{\sqrt2}{3}$ & \cite{ye16}\\
7 & $\frac17$ d & $\frac{-1}{(4+\sqrt{14})^2}$ & $1-\frac{4\sqrt{14}}{21}$ & \cite{ye16}\\
\hline
\end{tabular}}}
\end{table}


\begin{table}[H]
\caption{Rational and quadratic irrational series for level~17.}
\label{T:17.1}
{\renewcommand{\arraystretch}{1.4}
{\begin{tabular}{|cc|c|c|c|c|l|}
\hline
$N$ & $\overline{N}$ & $X_{17}(e^{-2\pi\sqrt{N/17}})$ & $\lambda_{17}$ & Reference \\
\hline
2 & d & $\frac{2}{11 + 3\sqrt{17}}$ & $\frac78 - \frac{\sqrt{17}}8$ & \cite{huber17}, \cite[Tbl. 6]{huberActa} \\
5 & below & $\frac{1}{2\cdot(7 + \sqrt{85})}$ & $\frac7{19} - \frac{\sqrt{85}}{95}$ & \cite{huber17}, \cite[Tbl. 6]{huberActa} \\
6 & $\frac32$ & $\frac{1}{3\cdot(1 + 2\sqrt2)^2}$ & $\frac{107}{340} - \frac{3\sqrt2}{68}$ & \cite{huber17}, \cite[Tbl. 6]{huberActa} \\
\hline
\hline
$N$ & $\overline{N}$ & $(a,b,c):\; a + bX_{17}+cX_{17}^2=0$ & $\lambda_{17}$ & Reference \\
\hline
3 d &  & $\frac{-1}{3}$ & $\frac12$ &\cite{huber17}, \cite[Tbl. 6]{huberActa} \\
11 &  & $\frac{-1}{11}$ & $\frac{13}{34}$ &\cite{huber17}, \cite[Tbl. 6]{huberActa} \\
\hline
5 d & above & $\frac{-1}{2\cdot(\sqrt{85}-7)}$ & $\frac7{19} + \frac{\sqrt{85}}{95}$ & \cite{huber17}, \cite[Tbl. 6]{huberActa} \\
19 & d & $\frac{-2}{23 + 7\sqrt{17}}$ & $\frac{56}{95} - \frac{13\sqrt{17}}{190}$ & \cite{huber17}, \cite[Tbl. 6]{huberActa}\\
35 & $\frac75$ d & $\frac{-4}{7\cdot(3 + \sqrt{17})^2}$ & 
$\frac{36}{95} - \frac{7\sqrt{17}}{190}$ & \cite{huber17}, \cite[Tbl. 6]{huberActa}\\
59 & d & $\frac{-1}{173 + 42\sqrt{17}}$ & 
$\frac{20123}{25370} - \frac{1909\sqrt{17}}{12685}$ &\cite{huber17}, \cite[Tbl. 6]{huberActa} \\
83 & d & $\frac{-1}{3^2\cdot(57 + 14\sqrt{17})}$ & $\frac{4221657}{9931946} - \frac{337626\sqrt{17}}{4965973}$ & \cite{huber17}, \cite[Tbl. 6]{huberActa}\\
\hline
\end{tabular}}}
\end{table}


\begin{table}[H]
\caption{Rational and quadratic irrational series for level~18.}
\label{T:18.1}
{\renewcommand{\arraystretch}{1.4}
{\begin{tabular}{|cc|c|c|l|}
\hline
$N$ & $\overline{N}$ & $X_{18}(e^{-2\pi\sqrt{N/18}})$  & $\lambda_{18}$ &  Reference \\
\hline
1 d && $\frac{1}{8}$ & undefined &  \\ 
2 && $\frac{1}{5+\sqrt3}$ & $1-\frac{\sqrt{3}}{3}$ & \cite{cooperlevel7} \\ 
4 & below & $\frac{1}{8+9\sqrt2}$ & $\frac13$ &  \\ 
5 && $\frac{1}{14+6\sqrt6}$ & $\frac{6}{5}-\frac{11\sqrt{6}}{30}$& \cite{cooperlevel7} \\ 
7 && $\frac{1}{26+18\sqrt2}$ & $\frac{16}{21}-\frac{5\sqrt2}{14}$&  \\ 
11 && $\frac{1}{68+12\sqrt{33}}$ & $\frac94-\frac{47\sqrt{33}}{132}$& \cite{cooperlevel7} \\ 
29 && $\frac{1}{1454 + 594\sqrt6}$ & $\frac{3038}{145} -\frac{7397\sqrt6}{870}$& \cite{cooperlevel7} \\
\hline
\hline
$N$ & $\overline{N}$ & $X_{18}(-e^{-\pi\sqrt{N/9}})$  & $\lambda_{18}$ &  Reference \\
\hline
1 d && $\frac{-1}{4}$ & undefined &  \\ 
5 && $\frac{-1}{10}$ & $\frac13$ &  \\ 
\hline
2 d & above & $\frac{-1}{9\sqrt2-8}$ & $\frac13$ &  \\ 
7 && $\frac{-4}{(3\sqrt3 + \sqrt7)^2}$ & $\frac25 - \frac{2\sqrt{21}}{105}$ &  \\ 
13 && $\frac{-1}{22 + 12\sqrt3}$ & $\frac1{13} + \frac{4\sqrt3}{39}$ & \cite{cooperlevel7} \\ 
17 && $\frac{-1}{37 + 9\sqrt{17}}$ & $\frac{5}{6} - \frac{5\sqrt{17}}{34}$ &  \\ 
25 && $\frac{-1}{2\cdot(2\sqrt5 + 3\sqrt3)^2}$ & $-\frac97 + \frac{8\sqrt{15}}{21}$ & \cite{cooperlevel7} \\ 
37 && $\frac{-1}{292 + 168\sqrt3}$ & $-\frac{325}{74} + \frac{583\sqrt3}{222}$ & \cite{cooperlevel7} \\ 
\hline
\end{tabular}}}
\end{table}


\begin{table}[H]
\caption{Rational and quadratic irrational series for level~19.}
\label{T:19.1}
{\renewcommand{\arraystretch}{1.4}
{\begin{tabular}{|cc|c|c|c|c|l|}
\hline
$N$ & $\overline{N}$ & $X_{19}(e^{-2\pi\sqrt{N/19}})$ & $\lambda_{19}$ & Reference \\
\hline
3 & below & $\frac{1}{7 + 2\sqrt{19}}$ & $\frac7{12} - \frac{\sqrt{19}}{12}$ &  \\
7 & below & $\frac{1}{22 + 10\sqrt7}$ & $\frac{26}{57} - \frac{37\sqrt7}{399}$ & \\
10 & $\frac52$ & $\frac{1}{56 + 19\sqrt5}$ & $\frac{51}{248} - \frac{19\sqrt5}{3720}$ & \\
\hline
\hline
$N$ & $\overline{N}$ & $X_{18}(-e^{-\pi\sqrt{N/19}})$ & $\lambda_{19}$ & Reference \\
\hline
1 d &  & $-1$ & undefined &  \\
3 d & above & $\frac{-1}{2\sqrt{19}-7}$ & $\frac7{12} + \frac{\sqrt{19}}{12}$ &  \\
7 d & above & $\frac{-1}{10\sqrt7-22}$ & $\frac{26}{57} + \frac{37\sqrt7}{399}$ & \\
9 d & d & $\frac{-2}{5 + \sqrt{57}}$ & $\frac54 - \frac{\sqrt{57}}{12}$ &  \\
17 & d & $\frac{-2}{13+5\sqrt{17}}$ & $\frac{25}{52} - \frac{11\sqrt{17}}{884}$ &  \\
25 & d & $\frac{-1}{16+8\sqrt5}$ & $\frac13$ &  \\
33 & $\frac{11}{3}$ d & $\frac{-1}{31 + 2\sqrt{209}}$ & $\frac{17}{29} - \frac{41\sqrt{209}}{1914}$ &  \\
73 & div. & $\frac{-2}{469 + 55\sqrt{73}}$ & 
$\frac{1201}{1332} - \frac{8353\sqrt{73}}{97236}$ &  \\
\hline
\end{tabular}}}
\end{table}


\begin{table}[H]
\caption{Rational and quadratic irrational series for level~20.}
\label{T:20.1}
{\renewcommand{\arraystretch}{1.4}
{\begin{tabular}{|cc|c|c|c|c|l|}
\hline
$N$ & $\overline{N}$ & $X_{20}(e^{-2\pi\sqrt{N/20}})$ & $\lambda_{20}$ & Reference \\
\hline
3 && $\frac{1}{2^4}$ & $\frac16$ & \cite{huber20} \\
\hline
1 d & d & $\frac{1}{(1+\sqrt{5})^2}$ & undefined & \cite{huber20} \\
2 & d & $\frac{1}{6 + 2\sqrt{10}}$ & $\frac23 - \frac{\sqrt{10}}6$ & \cite{huber20} \\
5 & d & $\frac{4}{(1 + \sqrt5)^4}$ & $\frac34 - \frac{\sqrt5}4$ & \cite{huber20} \\
9 & d & $\frac{1}{2^2\cdot(2 + \sqrt5)^2}$ & $\frac34 - \frac{\sqrt5}4$ & \\
11 & d & $\frac{1}{(1 + \sqrt5)^4}$ &
$\frac{13}{22} - \frac{2\sqrt5}{11}$ & \cite{huber20} \\
17 & d & $\frac{1}{(9 + \sqrt{85})^2}$ & $\frac34 - \frac{13\sqrt{85}}{204}$ & \cite{huber20} \\
\hline
\hline
$N$ & $\overline{N}$ & $X_{20}(-e^{-2\pi\sqrt{N/20}})$ & $\lambda_{20}$ & Reference \\
\hline
2 d && $\frac{-1}{2^2}$ & $\frac34$ & \cite{huber20} \\
\hline
1 d & d & $\frac{-1}{\sqrt{5}-1}$ & $\frac12$ & \cite{huber20} \\
4 & d & $\frac{-1}{2^2\cdot(1 + \sqrt5)}$ & $\frac12$ & \cite{huber20} \\
6 & $\frac23$ d & $\frac{-1}{(3 + \sqrt5)^2}$ & $\frac34 - \frac{\sqrt5}6$ & \\
8 & d & $\frac{-1}{2^3\cdot(3 + \sqrt{10})}$ & $\frac56 - \frac{\sqrt{10}}{6}$ & \cite{huber20} \\
12 & $\frac43$ d & $\frac{-1}{2^3\cdot(\sqrt3 + \sqrt5)^2}$ 
& $\frac{87}{154} - \frac{20\sqrt{15}}{231}$ & \cite{huber20} \\
14 & $\frac27$ d & $\frac{-4}{(3 + \sqrt5)^4}$ & 
$\frac34 - \frac{17\sqrt5}{70}$ & \cite{huber20} \\
26 & $\frac{2}{13}$ d & $\frac{-1}{2^2\cdot(2 + \sqrt5)^4}$ & 
$\frac34 - \frac{53\sqrt5}{195}$ & \cite{huber20} \\
38 & $\frac{2}{19}$ d & $\frac{-1}{(2\sqrt5 + 3\sqrt2)^4}$ & 
$\frac34 - \frac{259\sqrt{10}}{1292}$ & \cite{huber20} \\
\hline
\end{tabular}}}
\end{table}


\begin{table}[H]
\caption{Rational and quadratic irrational series for level~21.}
\label{T:21.1}
{\renewcommand{\arraystretch}{1.4}
{\begin{tabular}{|cc|c|c|c|c|l|}
\hline
$N$ & $\overline{N}$ & $X_{21}(e^{-2\pi\sqrt{N/21}})$ & $\lambda_{21}$ & Reference \\
\hline
2 && $\frac{1}{2^3}$ & $\frac12$& \cite{anusha}\\
\hline
1 d & below & $\frac{1}{1 + 2\sqrt7}$ & undefined & \\
3 & d & $\frac{2}{3\cdot(3 + \sqrt{21})}$ & 
$\frac9{10} - \frac{\sqrt{21}}{10}$ & \\
4 & d & $\frac{2}{(3 + \sqrt7)^2}$ & 
$\frac56 - \frac{\sqrt7}{6}$ & \\
5 & below & $\frac{1}{8 + 3\sqrt{21}}$ & $\frac{1}{5} + \frac{\sqrt{21}}{30}$ & \\
8 & below & $\frac{1}{22 + 10\sqrt7}$ & $\frac{\sqrt7}{6} - \frac16$ & \\
10 & $\frac{5}{2}$ & $\frac{1}{(2\sqrt2 + \sqrt{35})^2}$ & 
$\frac{47}{132} - \frac{3\sqrt{70}}{220}$ & \\
\hline
\end{tabular}}}
\end{table}


\begin{table}[H]
\caption{Rational and quadratic irrational series for level~22.}
\label{T:22.1}
{\renewcommand{\arraystretch}{1.4}
{\begin{tabular}{|cc|c|c|c|c|l|}
\hline
$N$ & $\overline{N}$ & $X_{22}(e^{-2\pi\sqrt{N/22}})$ & $\lambda_{22}$ & Reference \\
\hline
1 d && $\frac18$ & undefined &\\
\hline
3 & d & $\frac{2}{(\sqrt3 + \sqrt{11})^2}$& $\frac34 - \frac{\sqrt{33}}{12}$&\\
4 & below & $\frac1{7 + 3\sqrt{11}}$& $\frac12 - \frac{\sqrt{11}}{14}$&\\
7 & d & $\frac{1}{18 + 4\sqrt{22}}$ & $\frac{415}{546} - \frac{61\sqrt{22}}{546}$ & \\
9 & d & $\frac{2}{(5 + \sqrt{33})^2}$ & $\frac{59}{68} - \frac{23\sqrt{33}}{204}$ & \\
15 & $\frac53$ & $\frac{2}{5(3\sqrt3 + \sqrt{11})^2}$ & 
$\frac{281}{1218} - \frac{5\sqrt{33}}{609}$ & \\
19 & d & $\frac{1}{2^2\cdot(43 + 3\sqrt{209})}$ & $\frac{5011}{7540} - \frac{4937\sqrt{209}}{143260}$ & \\
21 & $\frac73$ & $\frac{1}{7(2\sqrt{3} + \sqrt{22})^2}$ & $\frac{303}{1394} - \frac{329\sqrt{66}}{46002}$ & \\
\hline
\hline
$N$ & $\overline{N}$ & $X_{22}(-e^{-\pi\sqrt{N/11}})$ & $\lambda_{22}$ & Reference \\
\hline
3 d && $\frac{-1}{2^2}$ & $\frac23$& \\
23 && $\frac{-1}{2^2\cdot23}$ & $\frac{157}{665}$ & \cite{anusha}\\
\hline
2 d & above & $\frac{-1}{3\sqrt{11}-7}$& $\frac12 + \frac{\sqrt{11}}{14}$&\\
5 & d & $\frac{-2^4}{(1 + \sqrt5)^4}$& $1-\frac{\sqrt{5}}{5}$& \\
7 & d & $\frac{-1}{4+2\sqrt{11}}$ & $\frac{43}{70} - \frac{3\sqrt{11}}{70}$ & \\
15 & $\frac53$ d & $\frac{-1}{5(1 + \sqrt3)^2}$ & $\frac{129}{242} - \frac{95\sqrt3}{726}$ & \\
35 & $\frac75$ d & $\frac{-1}{(9 + \sqrt{55})^2}$ & 
$\frac{383}{798} - \frac{53\sqrt{55}}{1330}$ & \\
\hline
\end{tabular}}}
\end{table}


\begin{table}[H]
\caption{Rational and quadratic irrational series for level~23.}
\label{T:23.1}
{\renewcommand{\arraystretch}{1.4}
{\begin{tabular}{|cc|c|c|c|c|l|}
\hline
$N$ & $\overline{N}$ & $X_{23}(e^{-2\pi\sqrt{N/23}})$ & $\lambda_{23}$ & Reference \\
\hline
2 & d & $\frac1{5 + 3\sqrt2}$ & $1 - \frac{\sqrt2}{2}$ & \cite{coopergeye}\\
11 & below & $\frac{1}{3^2\cdot(4 + \sqrt{23})}$ & 
$\frac{121211}{314545} - \frac{11538\sqrt{23}}{314545}$ & \cite{coopergeye} \\
\hline
\hline
$N$ & $\overline{N}$ & $X_{23}(-e^{-\pi\sqrt{N/23}})$ & $\lambda_{23}$ & Reference \\
\hline
5 d &  & $\frac{-1}{3}$ & $\frac{41}{87}$ & \\
\hline
11 d & above & $\frac{-1}{3^2\cdot(\sqrt{23}-4)}$ & 
$\frac{121211}{314545} + \frac{11538\sqrt{23}}{314545}$ & \\
21 & $\frac73$ & $\frac{-1}{10 + \sqrt{69}}$ & 
$\frac{419}{595} - \frac{67\sqrt{69}}{1785}$ & \cite{coopergeye}\\
29 & d & $\frac{-1}{16 + 3\sqrt{29}}$ & 
$\frac{947}{805} - \frac{3701\sqrt{29}}{23345}$ & \cite{coopergeye} \\
\hline
\end{tabular}}}
\end{table}


\begin{table}[H]
\caption{Rational and quadratic irrational series for level~24.}
\label{T:24.1}
{\renewcommand{\arraystretch}{1.4}
{\begin{tabular}{|cc|c|c|c|c|l|}
\hline
$N$ & $\overline{N}$ & $X_{24}(e^{-2\pi\sqrt{N/24}})$ & $\lambda_{24}$ & Reference \\
\hline
1 & d & $\frac{1}{(1+\sqrt{3})^2}$ & undefined & \cite{alchemy}\\
2 & below & $\frac{1}{2(2 + \sqrt6)}$ & $\frac45 - \frac{\sqrt6}{5}$& \cite{alchemy}\\
3 & d & $\frac{1}{2(1 + \sqrt2)^2}$ & $1-\frac{\sqrt{2}}{2}$& \cite{alchemy}\\
5 & d & $\frac{1}{(2+\sqrt6)^2}$ & $1 - \frac{3\sqrt6}{10}$& \cite{alchemy}\\
7 & d & $\frac{1}{(3+\sqrt7)^2}$ & $1 - \frac{2\sqrt7}{7}$& \cite{alchemy}\\
13 & d & $\frac{1}{(5+3\sqrt3)^2}$ & $1 - \frac{6\sqrt3}{13}$& \cite{alchemy}\\
17 & d & $\frac{1}{(7+\sqrt{51})^2}$ & $1 - \frac{39\sqrt{51}}{340}$& \cite{alchemy}\\
\hline
\hline
$N$ & $\overline{N}$ & $X_{24}(-e^{-2\pi\sqrt{N/24}})$ & $\lambda_{24}$ & Reference \\
\hline
1 && $\frac{-1}{2}$ & $1$& \cite{alchemy}\\
\hline
$\frac12$ d & above & $\frac{-1}{2(\sqrt6-2)}$ & $\frac45 + \frac{\sqrt6}{5}$& \cite{alchemy}\\
2 & $\frac12$ d & $\frac{-1}{2+\sqrt6}$ & $\frac65-\frac{\sqrt6}{5}$& \cite{alchemy}\\
3 & $\frac13$ d & $\frac{-1}{(1+\sqrt{3})^2}$ & $1-\frac{\sqrt3}4$& \cite{alchemy}\\
5 & $\frac15$ d & $\frac{-1}{(\sqrt3+\sqrt5)^2}$ & $1-\frac{3\sqrt{15}}{20}$& \cite{alchemy}\\
7 & $\frac17$ d & $\frac{-2}{(1+\sqrt{3})^4}$ & $1-\frac{8\sqrt3}{21}$& \cite{alchemy}\\
13 & $\frac1{13}$ d& $\frac{-1}{2(2\sqrt{3}+\sqrt{13})^2}$ & $1-\frac{8\sqrt{39}}{65}$& \cite{alchemy}\\
17 &$\frac1{17}$ d& $\frac{-1}{(5\sqrt{2}+4\sqrt{3})^2}$ & $1-\frac{39\sqrt6}{119}$& \cite{alchemy}\\
\hline
\end{tabular}}}
\end{table}


\begin{table}[H]
\caption{Rational and quadratic irrational series for level~$25$.}
\label{T:25.1}
{\renewcommand{\arraystretch}{1.4}
{\begin{tabular}{|cc|c|c|}
\hline
$N$ & $\overline{N}$ & $X_{25}(e^{-2\pi\sqrt{N/25}})$ & $\lambda_{25}$  \\
\hline
$1$ d & below & $\frac1{2(1+\sqrt{5})}$ & undefined \\
4 & d & $\frac{2^3}{(1+\sqrt{5})^4}$ & $\frac{3}{2} - \frac{\sqrt5}{2}$  \\
\hline
\hline
$N$ & $\overline{N}$ & $X_{25}(-e^{-\pi\sqrt{N/25}})$ & $\lambda_{25}$  \\
\hline
$1$ d & above & $\frac{-1}{2(\sqrt{5}-1)}$ & undefined \\
$3$ d &  & $\frac{-1}{3}$ & 0 \\
11 & d & $\frac{-1}{3 + 2\sqrt5}$ & $\frac{8}{11} - \frac{2\sqrt5}{11}$ \\ 
19 & d & $\frac{-1}{8 + 3\sqrt5}$ & $\frac{35}{19} - \frac{13\sqrt5}{19}$ \\ 
\hline
\end{tabular}}}
\end{table}


\begin{table}[H]
\caption{Rational and quadratic irrational series for level~$26$.}
\label{T:26.1}
{\renewcommand{\arraystretch}{1.4}
{\begin{tabular}{|cc|c|c|}
\hline
$N$ & $\overline{N}$ & $X_{26}(e^{-2\pi\sqrt{N/26}})$ & $\lambda_{26}$  \\
\hline
 3 && $\frac19$ & $\frac12$ \\ 
 5 && $\frac1{16}$ & $\frac25$ \\ 
 \hline 
 2 & d & $\frac1{3 + \sqrt{13}}$ & $\frac76 - \frac{\sqrt{13}}6$ \\ 
 17 & d & $\frac{2}{161+39\sqrt{17}}$ & $\frac{97}{104} - \frac{313\sqrt{17}}{1768}$ \\
 23 & d & $\frac{1}{185+36\sqrt{26}}$ & $\frac{906}{1265} - \frac{3241\sqrt{26}}{30360}$ \\
\hline
\hline
$N$ & $\overline{N}$ & $X_{26}(-e^{-\pi\sqrt{N/13}})$ & $\lambda_{26}$  \\
\hline
1 d && $-\frac{1}{4}$ & undefined \\ 
\hline
3 d & d & $\frac{-2^2}{(1+\sqrt{13})^2}$ & $\frac{5}{6} - \frac{\sqrt{13}}{6}$ \\ 
5 & d & $\frac{-2}{7+\sqrt{65}}$ & $\frac{7}{8} - \frac{3\sqrt{65}}{40}$ \\ 
9 & d & $\frac{-2^2}{(1+\sqrt{3})^4}$ & $1 -  \frac{5\sqrt{3}}{12}$ \\ 
21 & $\frac73$ d & $\frac{-2^2}{(3\sqrt{3}+\sqrt{91})^2}$ & 
$\frac{25}{114} + \frac{\sqrt{273}}{1596}$ \\ 
61 & d & $\frac{-1}{452 + 16\sqrt{793}}$ & 
$\frac{7481}{7788} - \frac{13745\sqrt{793}}{475068}$ \\ 
\hline
\end{tabular}}}
\end{table}


\begin{table}[H]
\caption{Simplest examples of cubic irrational series for level~$27$; not a complete list.}
\label{T:27.1}
{\renewcommand{\arraystretch}{1.4}
{\begin{tabular}{|c|c|c|}
\hline
$N$ & $X_{27}(e^{-2\pi\sqrt{N/27}})$ & $\lambda_{27}$ \\
\hline
4 & $\frac{2-2^{1/3}}6$ & 0 \\
\hline
\hline
$N$ & $X_{27}(-e^{-\pi\sqrt{N/27}})$ & $\lambda_{27}$ \\
\hline
9 d & $-\frac{(3^{1/3} + 3^{-1/3} - 1)}{4}$ & $\frac{3^{4/3}(3^{2/3} + 3^{-2/3} - 1)}{5}$\\
25 & $-\frac{1}{3}\left(\frac{2}{5^{1/3}}-1\right)$ & $2(2-5^{1/3})$\\
49 & $-\frac{1}{3}\left(\frac{2}{7^{1/3}}-1\right)$ & $\frac{18}{5}(2-7^{1/3})$\\
\hline
\end{tabular}}}
\end{table}


\begin{table}[H]
\caption{Rational and quadratic irrational series for level~$28$.}
\label{T:28.1}
{\renewcommand{\arraystretch}{1.4}
{\begin{tabular}{|cc|c|c|}
\hline
$N$ & $\overline{N}$ & $X_{28}(e^{-2\pi\sqrt{N/28}})$ & $\lambda_{28}$  \\
\hline
1 d && $\frac18$ & undefined \\
\hline
3 & d & $\frac1{2(3 + \sqrt{7})}$ & $1 - \frac{\sqrt7}{3}$\\
4 & below & $\frac1{7+ 5\sqrt{2}}$ & $\frac{20}{21} - \frac{4\sqrt2}{7}$ \\
9 & d & $\frac1{2(\sqrt3 + \sqrt{7})^2}$ & $\frac45 - \frac{2\sqrt{21}}{15}$ \\
19 & d & $\frac1{2(45 + 17\sqrt7)}$ & $1 - \frac{89\sqrt7}{285}$ \\
\hline
\hline
$N$ & $\overline{N}$ & $X_{28}(-e^{-2\pi\sqrt{N/28}})$ & $\lambda_{28}$  \\
\hline
1 d && $-1$ & $\frac43$ \\
4 && $-\frac{1}{8}$ & $\frac23$ \\ 
\hline
$\frac14$ d & above & $\frac{-1}{5\sqrt{2}-7}$ & $\frac{20}{21} + \frac{4\sqrt2}{7}$ \\
2 d & d & $\frac{-1}{2\sqrt{2}}$ & 1 \\
6 & $\frac23$ d &$\frac{-1}{2(4 + \sqrt{14})}$& $1 - \frac{2\sqrt{14}}{15}$\\
10 & $\frac25$ d & $\frac{-1}{2(10 + 7\sqrt2)}$ & $1 - \frac{7\sqrt2}{15}$ \\
16 & d & $\frac{-1}{8(7 + 5\sqrt2)}$ & $\frac{22}{21} - \frac{4\sqrt2}7$ \\ 
\hline
\end{tabular}}}
\end{table}


\begin{table}[H]
\caption{Rational and quadratic irrational series for level~$29$.}
\label{T:29.1}
{\renewcommand{\arraystretch}{1.4}
{\begin{tabular}{|cc|c|c|}
\hline
$N$ & $\overline{N}$ & $X_{29}(e^{-2\pi\sqrt{N/29}})$ & $\lambda_{29}$  \\
\hline
2 && $\frac18$ & $\frac17$ \\
\hline
8 & below & $\frac{1}{2(7+\sqrt{58})}$ & $\frac{129}{161} - \frac{25\sqrt{58}}{322}$ \\ 
\hline
\hline
$N$ & $\overline{N}$ & $X_{29}(-e^{-2\pi\sqrt{N/29}})$ & $\lambda_{29}$  \\
\hline
2 d & above & $\frac{-1}{2(\sqrt{58}-7)}$ & $\frac{129}{161} + \frac{25\sqrt{58}}{322}$ \\ 
7 d & d & $\frac{-2}{1+\sqrt{29}}$ & $\frac47 + \frac{\sqrt{29}}{14}$ \\
15 & $\frac53$ d& $\frac{-4}{3(1+\sqrt{5})^2}$ & $\frac{30}{29} - \frac{51\sqrt5}{290}$ \\
23 & d & $\frac{-2}{13 + 3\sqrt{29}}$ & 
$\frac{275}{161} - \frac{73\sqrt{29}}{322}$ \\
\hline
\end{tabular}}}
\end{table}


\begin{table}[H]
\caption{Rational and quadratic irrational series for level~$30$.}
\label{T:30.1}
{\renewcommand{\arraystretch}{1.4}
{\begin{tabular}{|cc|c|c|}
\hline
$N$ & $\overline{N}$ & $X_{30}(e^{-2\pi\sqrt{N/30}})$ & $\lambda_{30}$  \\
\hline
1 d && $\frac14$ & undefined \\
2 && $\frac15$ & $\frac23$ \\
7 && $\frac1{20}$ & $\frac{11}{28}$ \\
11 && $\frac1{44}$ & $\frac{25}{84}$ \\
\hline
3 & d & $\frac1{2(1 + \sqrt6)}$ & $1 - \frac{\sqrt6}6$ \\
4 & below & $\frac1{2 + 3\sqrt6}$ & $\frac15 + \frac{2\sqrt6}{15}$ \\
5 & d & $\frac1{2(3 + \sqrt{10})}$ & $1 - \frac{\sqrt{10}}6$ \\
\hline
\end{tabular}}}
\end{table}


\begin{table}[H]
\caption{Rational and quadratic irrational series for level~$31$.}
\label{T:31.1}
{\renewcommand{\arraystretch}{1.4}
{\begin{tabular}{|cc|c|c|}
\hline
$N$ & $\overline{N}$ & $X_{31}(e^{-2\pi\sqrt{N/31}})$ & $\lambda_{31}$  \\
\hline
3 & below & $\frac{1}{2 + \sqrt{31}}$ & $\frac{-1}{33} + \frac{4\sqrt{31}}{33}$ \\
\hline
\hline
$N$ & $\overline{N}$ & $X_{31}(-e^{-\pi\sqrt{N/31}})$ & $\lambda_{31}$  \\
\hline
13 && $-\frac{1}{8}$ & $\frac{4}{39}$ \\
\hline
3 d & above &  $\frac{-1}{\sqrt{31}-2}$ & $\frac{-1}{33} - \frac{4\sqrt{31}}{33}$ \\ 
5 d & d & $\frac{-1}{2+\sqrt{5}}$ & $1 - \frac{3\sqrt5}{5}$ \\
\hline
\end{tabular}}}
\end{table}


\begin{table}[H]
\caption{Series involving algebraic numbers of degree $\leq 4$ for level~$32$.}
\label{T:32.1}
{\renewcommand{\arraystretch}{1.4}
{\begin{tabular}{|c|c|}
\hline
$N$ & Equation satisfied by $X_{32}(e^{-2\pi\sqrt{N/32}})$   \\
\hline
1 d &  $4x^4 - 16x^3 + 12x^2 - 8x + 1=0$, $\lambda_{32}$ is undefined  \\
2 &  $4x^4 - 4x^2 + 8x - 1=0$  \\
3 &  $4x^4 - 24x^3 + 28x^2 - 12x + 1=0$  \\
5 &  $4x^4 - 32x^3 + 28x^2 - 16x + 1=0$  \\
9 &  $4x^4 - 64x^3 + 60x^2 - 32x + 1=0$  \\
11 &  $4x^4 - 88x^3 + 92x^2 - 44x + 1=0$  \\
29 &  $4x^4 - 800x^3 + 796x^2 - 400x + 1=0$  \\
\hline
\hline
$N$ & Equation satisfied by $X_{32}(-e^{-2\pi\sqrt{N/32}})$  \\
\hline
1 d &  $x^2 - 2=0$, $\lambda_{32}=2$  \\
2 d &  $4x^4 - 16x^3 + 4x^2 - 1=0$  \\
3 d &  $4x^4 + 8x^3 - 4x^2 + 4x + 1=0$  \\
5 &  $4x^4 + 16x^3 - 20x^2 + 8x + 1=0$  \\
9 &  $4x^4 + 48x^3 - 52x^2 + 24x + 1=0$  \\
11 &  $4x^4 + 72x^3 - 68x^2 + 36x + 1=0$  \\
29 &  $4x^4 + 784x^3 - 788x^2 + 392x + 1=0$  \\
\hline
\end{tabular}}}
\end{table}


\begin{table}[H]
\caption{Rational and quadratic irrational series for level~33.}
\label{T:33.1}
{\renewcommand{\arraystretch}{1.4}
{\begin{tabular}{|cc|c|c|c|c|l|}
\hline
$N$ & $\overline{N}$ & $X_{33}(e^{-2\pi\sqrt{N/33}})$ & $\lambda_{33}$ & Reference \\
\hline
1 d & below & $\frac{1}{1 +2\sqrt{3}}$ & undefined & \\
2 & d & $\frac{2}{5+\sqrt{33}}$ & $\frac{11}8 - \frac{\sqrt{33}}{8}$ &\\ 
4 & d & $\frac{1}{4 + 3\sqrt{3}}$ & $\frac{29}{26} - \frac{9\sqrt3}{26}$ &\\ 
5 & below & $\frac{1}{2(2 + \sqrt{15})}$ & $\frac{15}{34} + \frac{\sqrt{15}}{170}$ &\\ 
10 & $\frac52$ & $\frac{2}{(\sqrt{5} + \sqrt{33})^2}$ & $\frac{85}{182} - \frac{16\sqrt{165}}{1365}$ &\\ 
14 & $\frac72$ & $\frac{1}{(\sqrt6+2\sqrt{7})^2}$ & $\frac{31121}{77198} - \frac{11945\sqrt{42}}{540386}$ &\\ 
\hline
\hline
$N$ & $\overline{N}$ & $X_{33}(-e^{-\pi\sqrt{N/33}})$ & $\lambda_{33}$ & Reference \\
\hline
3 d && $-\frac{1}{3}$ & $\frac12$ &\\
19 && $-\frac{1}{11}$ & $\frac{81}{266}$ & \cite{anusha} \\
\hline
1 d & above & $\frac{-1}{2\sqrt{3}-1}$ & undefined & \\
5 d & above & $\frac{-1}{2(\sqrt{15}-2)}$ & $\frac{15}{34} - \frac{\sqrt{15}}{170}$ &\\ 
11 & d & $\frac{-2}{7 + \sqrt{33}}$ & $\frac54 - \frac{7\sqrt{33}}{44}$ &\\
35 & $\frac75$ d & $\frac{-1}{2(7+\sqrt{33})}$ & $\frac37 - \frac{\sqrt{33}}{35}$ &\\
59 & d & $\frac{-2}{67 + 5\sqrt{177}}$ & 
$\frac{6071}{5324} - \frac{21677\sqrt{177}}{314116}$ &\\
91 &$\frac{13}{7}$ d & $\frac{-2}{187 + 11\sqrt{273}}$ & 
$\frac{5121}{13588} - \frac{42823\sqrt{273}}{3709524}$ &\\
131 & d & $\frac{-2}{523 + 91\sqrt{33}}$ & 
$\frac{237789}{217460} - \frac{35419\sqrt{33}}{217460}$ &\\
\hline
\end{tabular}}}
\end{table}


\begin{table}[H]
\caption{Rational and quadratic irrational series for level~$34$.}
\label{T:34.1}
{\renewcommand{\arraystretch}{1.4}
{\begin{tabular}{|cc|c|c|}
\hline
$N$ & $\overline{N}$ & $X_{34}(e^{-2\pi\sqrt{N/34}})$ & $\lambda_{34}$  \\
\hline
3 && $\frac{1}{9}$ & $\frac{1}{16}$ \\ 
\hline
1 d & d & $\frac{4}{(1 + \sqrt{17})^2}$ & undefined \\
7 & d & $\frac{1}{(3+\sqrt{2})^2}$ & $\frac{19}{28} - \frac{5\sqrt2}{14}$ \\ 
13 & d & $\frac{1}{(3 + \sqrt{17})^2}$ & $\frac{115}{104} - \frac{23\sqrt{17}}{104}$ \\ 
\hline \hline
$N$ & $\overline{N}$ & $X_{34}(-e^{-\pi\sqrt{N/17}})$ & $\lambda_{34}$  \\
\hline
5 d && $-\frac{1}{4}$ & $\frac{9}{10}$ \\
\hline
1 d & d & $\frac{-2}{\sqrt{17}-1}$ & undefined \\ 
9 & d & $\frac{-1}{4+\sqrt{17}}$ & $\frac{23}{16} - \frac{3\sqrt{17}}{16}$ \\
21 & $\frac73$ d & $\frac{-4}{(7 + \sqrt{17})^2}$ & 
$\frac{415}{546} - \frac{103\sqrt{17}}{1092}$ \\
41 & d & $\frac{-2}{127 + 5\sqrt{697}}$ & $\frac{79}{112} - \frac{83\sqrt{697}}{4592}$ \\
\hline
\end{tabular}}}
\end{table}


\begin{table}[H]
\caption{Rational and quadratic irrational series for level~35.}
\label{T:35.1}
{\renewcommand{\arraystretch}{1.4}
{\begin{tabular}{|cc|c|c|c|c|l|}
\hline
$N$ & $\overline{N}$ & $X_{35}(e^{-2\pi\sqrt{N/35}})$ & $\lambda_{35}$ & Reference \\
\hline
2 && $\frac{1}{7}$ & $\frac14$ & \cite{anusha} \\
\hline
3 & below & $\frac{1}{4 + \sqrt{21}}$ & $1 - \frac{\sqrt{21}}{6}$ &\\ 
6 & $\frac32$ & $\frac{2}{(1 + \sqrt{21})^2}$ & $\frac{41}{170} - \frac{\sqrt{21}}{255}$ &\\ 
8 & below & $\frac{1}{11 + 3\sqrt{14}}$ & $\frac{33}{31} - \frac{7\sqrt{14}}{31}$ &\\ 
11 & below & $\frac{1}{2(8 + 3\sqrt{11})}$ & $\frac{833}{1406} - \frac{1769\sqrt{11}}{15466}$ &\\ 
\hline
\hline
$N$ & $\overline{N}$ & $X_{35}(-e^{-\pi\sqrt{N/35}})$ & $\lambda_{35}$ & Reference \\
\hline
9 d && $-\frac{1}{3}$ & $\frac{11}{10}$ &\\
17 && $-\frac{1}{7}$ & $\frac{123}{170}$ & \cite{anusha} \\
41 && $-\frac{1}{28}$ & $\frac{1004}{2665}$ & \cite{anusha} \\
\hline
2 d & above & $\frac{-1}{3\sqrt{14}-11}$ & $\frac{33}{31} + \frac{7\sqrt{14}}{31}$ & \\ 
3 d & above & $\frac{-1}{\sqrt{21}-4}$ & $1 +\frac{\sqrt{21}}{6}$ &\\ 
11 d & above & $\frac{-1}{2(3\sqrt{11}-8)}$ & 
$\frac{833}{1406} + \frac{1769\sqrt{11}}{15466}$ &\\
33 & $\frac{11}{3}$ d & $\frac{-1}{(\sqrt{3} + \sqrt{7})^2}$ & 
$\frac{58}{55} - \frac{2\sqrt{21}}{15}$ &\\
57 & $\frac{19}{3}$ d& $\frac{-4}{(5\sqrt{3} + \sqrt{35})^2)}$ & 
$\frac{2537}{3116} - \frac{3313\sqrt{105}}{65436}$ &
\\
\hline
\end{tabular}}}
\end{table}


\begin{table}[H]
\caption{Rational and quadratic irrational series for level~$36$.}
\label{T:36.1}
{\renewcommand{\arraystretch}{1.4}
{\begin{tabular}{|cc|c|c|}
\hline
$N$ & $\overline{N}$ & $X_{36}(e^{-2\pi\sqrt{N/36}})$ & $\lambda_{36}$  \\
\hline
1 d & d & $\frac{1}{3 + 2\sqrt3}$ & undefined \\
2 & below & $\frac{1}{3(1 + \sqrt2)}$ & $0$ \\ 
5 & d & $\frac{1}{3(2 + \sqrt5)}$ & $\frac32 - \frac{3\sqrt5}{5}$ \\ 
7 & d & $\frac{1}{9 + 2\sqrt{21}}$ & $\frac32 - \frac{2\sqrt{21}}{7}$ \\ 
\hline \hline
$N$ & $\overline{N}$ & $X_{36}(-e^{-2\pi\sqrt{N/36}})$ & $\lambda_{36}$  \\
\hline
2 d && $-\frac{1}{3}$ & $\frac34$ \\
\hline
$\frac12$ d & above &  $\frac{-1}{3(\sqrt2-1)}$ & $0$ \\ 
1 d & d & $\frac{-1}{\sqrt{3}}$ & $\frac32-\frac{\sqrt{3}}{2}$ \\ 
4 & d & $\frac{-1}{3 + 2\sqrt3}$ & $\frac32 - \frac{\sqrt3}{2}$ \\
8 & d & $\frac{-1}{3(1 + \sqrt2)^2}$ & $\frac32 - \frac{3\sqrt2}{4}$ \\
\hline
\end{tabular}}}
\end{table}

\begin{table}[H]
\caption{Rational and quadratic irrational series for levels~$38$, 39 and~41.}
\label{T:38.1}
{\renewcommand{\arraystretch}{1.4}
{\begin{tabular}{|cc|c|c|}
\hline
$N$ & $\overline{N}$ & $X_{38}(e^{-2\pi\sqrt{N/38}})$ & $\lambda_{38}$  \\
\hline
3 & d & $\frac{2}{5 + \sqrt{57}}$ & $\frac54 - \frac{\sqrt{57}}{12}$ \\
5 && $\frac{1}{10}$ & $\frac12$ \\ 
11 & d & $\frac{2}{(\sqrt{11} + \sqrt{19})^2}$ & $\frac{941}{780} - \frac{523\sqrt{209}}{8580}$ \\ 
\hline
$N$ & $\overline{N}$ & $X_{38}(-e^{-\pi\sqrt{N/19}})$ & $\lambda_{38}$  \\
\hline
3 d && $-\frac{1}4$ & $\frac13$ \\ 
7 && $-\frac{1}{7}$ & $\frac{19}{60}$ \\
\hline
%
%
\hline
$N$ & $\overline{N}$ & $X_{39}(e^{-2\pi\sqrt{N/39}})$ & $\lambda_{39}$  \\
\hline
1 d & d & $\frac{2}{11+\sqrt{13}}$ & undefined \\
2 && $\frac{1}{8}$ & $-\frac{1}{2}$ \\ 
4 & d & $\frac{4}{(3 + \sqrt{13})^2}$ & $1 - \frac{\sqrt{13}}3$ \\ 
7 & below & $\frac{1}{8 + \sqrt{91}}$ & $\frac{11}{15} - \frac{8\sqrt{91}}{105}$ \\ 
8 & below & $\frac{1}{2(5 + 3\sqrt3)}$ & $\frac{43}{46} - \frac{71\sqrt3}{138}$ \\ 
\hline 
$N$ & $\overline{N}$ & $X_{39}(-e^{-\pi\sqrt{N/39}})$ & $\lambda_{39}$  \\
\hline
1 d & d & $\frac{-2}{\sqrt{13}-3}$ & undefined \\ 
2 d & above & $\frac{-1}{2(3\sqrt3-5)}$ & $\frac{43}{46} + \frac{71\sqrt3}{138}$ \\ 
5 d && $-1$ & $\frac85$ \\
7 d & above & $\frac{-1}{\sqrt{91}-8}$ & $\frac{11}{15} + \frac{8\sqrt{91}}{105}$ \\ 
13 d & d & $\frac{-1}{\sqrt{13}}$ & $\frac73-\frac{\sqrt{13}}{3}$ \\
29 & d & $\frac{-1}{5 + 2\sqrt{13}}$ & $\frac{152}{87} - \frac{26\sqrt{13}}{87}$ \\
37 & d & $\frac{-2}{15 + \sqrt{481}}$ & $\frac{11}6 - \frac{13\sqrt{481}}{222}$ \\
53 & d & $\frac{-1}{18 + 5\sqrt{13}}$ & $\frac{1195}{901} - \frac{231\sqrt{13}}{901}$ \\
77 & $\frac{11}{7}$ d & $\frac{-1}{40 + 11\sqrt{13}}$ & $\frac{601}{561} - \frac{851\sqrt{13}}{3927}$ \\
85 & $\frac{17}{5}$ d & $\frac{-4}{(3\sqrt{13} + \sqrt{85})^2}$ & $\frac{503}{506} - \frac{941\sqrt{1105}}{43010}$ \\
\hline
%
\hline
$N$ & $\overline{N}$ & $X_{41}(e^{-2\pi\sqrt{N/41}})$ & $\lambda_{41}$  \\
\hline
2 & d & $\frac{2}{3+\sqrt{41}}$ & $\frac{13}8 - \frac{\sqrt{41}}8$ \\
\hline 
$N$ & $\overline{N}$ & $X_{41}(-e^{-\pi\sqrt{N/41}})$ & $\lambda_{41}$  \\
\hline
3 d && $-\frac{1}3$ & $-\frac12$ \\ 
35 & $\frac75$ d & $\frac{-2}{7(3+\sqrt{5})}$ &
$\frac{-605}{2542} + \frac{1288\sqrt5}{6355}$ \\
\hline
\end{tabular}}} 
\end{table}

\begin{table}[H]
\caption{Rational and quadratic irrational series for levels~$42$ and~44.}
\label{T:42.1}
{\renewcommand{\arraystretch}{1.4}
{\begin{tabular}{|cc|c|c|}
\hline
$N$ & $\overline{N}$ & $X_{42}(e^{-2\pi\sqrt{N/42}})$ & $\lambda_{42}$  \\
\hline
1 d & d & $\frac{1}{3}$ & undefined \\
5 & &$\frac{1}{8}$ & $\frac{19}{30}$ \\ 
11 && $\frac{1}{24}$ & $\frac{43}{110}$ \\ 
\hline
2 & d & $\frac{1}{1+\sqrt{7}}$ & $\frac43-\frac{\sqrt{7}}{6}$ \\ 
3 & d & $\frac{1}{2(1+\sqrt{2})}$ & $\frac32 - \frac{\sqrt2}2$ \\ 
4 & below & $\frac{1}{1+2\sqrt{7}}$ & $\frac38 + \frac{\sqrt7}8$ \\ 
19 & d & $\frac{2}{67 + 9\sqrt{57}}$ & $\frac{19}{16} - \frac{37\sqrt{57}}{304}$ \\ 
31 & d & $\frac{2}{3(73 + 5\sqrt{217})}$ & $\frac{255}{208} - \frac{451\sqrt{217}}{6448}$ \\
\hline
$N$ & $\overline{N}$ & $X_{42}(-e^{-\pi\sqrt{N/21}})$ & $\lambda_{42}$  \\
\hline
$1$ d && $-\frac{1}4$ & undefined \\ 
5 & &$-\frac{1}{6}$ & $\frac{1}{10}$ \\
13 && $-\frac{1}{13}$ & $\frac{19}{96}$ \\
17 && $-\frac{1}{18}$ & $\frac{79}{374}$ \\
\hline
2 & above & $\frac{-1}{2\sqrt7 - 1}$ & $\frac38 - \frac{\sqrt7}8$ \\
3 & d & $\frac{-4}{(\sqrt3 + \sqrt{7})^2}$ & $\frac{21}{20} - \frac{13\sqrt{21}}{60}$ \\
25 &d& $\frac{-4}{(3\sqrt{5} + \sqrt{21})^2}$ & $\frac{65}{64} - \frac{5\sqrt{105}}{64}$ \\
29 &d& $\frac{-1}{4(5 + 2\sqrt7)}$ & $\frac{26}{29} - \frac{15\sqrt7}{58}$ \\
37 &d& $\frac{-1}{2(17 + 6\sqrt7)}$ & $\frac{3833}{4218} - \frac{562\sqrt7}{2109}$ \\
53 &d& $\frac{-1}{2(37 + 14\sqrt7)}$ & $\frac{3293}{3074} - \frac{3606\sqrt7}{10759}$ \\
65 & $\frac{13}{5}$ & $\frac{-4}{(5\sqrt{13} + 3\sqrt{21})^2}$ &
$\frac{391}{3520} + \frac{503\sqrt{273}}{137280}$ \\
\hline
%
%
\hline
$N$ & $\overline{N}$ & $X_{44}(e^{-2\pi\sqrt{N/44}})$ & $\lambda_{44}$  \\
\hline
2 & below & $\frac{1}{3+\sqrt{11}}$ & $\frac{17}{14} - \frac{5\sqrt{11}}{14}$ \\
3 &d& $\frac{1}{(1+\sqrt{3})^2}$ & $\frac32 - \frac{5\sqrt3}{6}$ \\ 
5 &d& $\frac{1}{(1+\sqrt{5})^2}$ & $1 - \frac{2\sqrt5}5$ \\ 
23 &d& $\frac{1}{(5+\sqrt{23})^2}$ & $\frac32 - \frac{1681\sqrt{23}}{6118}$ \\ 
\hline
$N$ & $\overline{N}$ & $X_{44}(-e^{-2\pi\sqrt{N/44}})$ & $\lambda_{44}$  \\
\hline
2 d && $-\frac{1}2$ & $\frac32$ \\ 
\hline
$\frac12$ d & above & $\frac{-1}{\sqrt{11}-3}$ & $\frac{17}{14} + \frac{5\sqrt{11}}{14}$ \\
8 &d & $\frac{-1}{2(3+\sqrt{11})}$ & $\frac{25}{14} - \frac{5\sqrt{11}}{14}$ \\
\hline
\end{tabular}}}
\end{table}

\begin{table}[H]
\caption{Rational and quadratic irrational series for levels~$45$,
46, 47, 49 and 50.}
\label{T:45.1}
{\renewcommand{\arraystretch}{1.4}
{%
}} 
\end{table}

\begin{table}[H]
\caption{Rational and quadratic irrational series for levels~$51$, 55 and~56.
There are no rational or quadratic irrational series for level~54.}
\label{T:51.1}
{\renewcommand{\arraystretch}{1.4}
{\begin{tabular}{|cc|c|c|}
\hline
$N$ & $\overline{N}$ & $X_{51}(e^{-2\pi\sqrt{N/51}})$ & $\lambda_{51}$  \\
\hline
2 && $\frac{1}{6}$ & $\frac{-3}{10}$ \\
7 & below & $\frac{2}{11+3\sqrt{21}}$ & $1-\frac{3\sqrt{21}}{14}$ \\ 
8 & d & $\frac{1}{7+\sqrt{51}}$ & $\frac{421}{290} - \frac{57\sqrt{51}}{290}$ \\ 
\hline 
$N$ & $\overline{N}$ & $X_{51}(-e^{-\pi\sqrt{N/51}})$ & $\lambda_{51}$  \\
\hline
7 d & above & $\frac{-2}{3\sqrt{21}-11}$ & $1+\frac{3\sqrt{21}}{14}$ \\ 
25 & d & $\frac{-2}{5+\sqrt{85}}$ & $\frac{11}{7} - \frac{\sqrt{85}}{14}$ \\ 
65 & $\frac{13}{5}$ d & $\frac{-2}{33+9\sqrt{13}}$ & $\frac{38544}{24565} - \frac{205389\sqrt{13}}{638690}$ \\
\hline 
\end{tabular}}} 
\end{table}

\begin{table}[H]
\caption{Rational and quadratic irrational series for levels~$59$, 60 and~62.}
\label{T:59.1}
{\renewcommand{\arraystretch}{1.4}
{\begin{tabular}{|cc|c|c|}
\hline
$N$ & $\overline{N}$ & $X_{59}(e^{-2\pi\sqrt{N/59}})$ & $\lambda_{59}$  \\
\hline
3 & below & $\frac{1}{1+2\sqrt{3}}$ & $\frac14+\frac{5\sqrt{3}}{12}$ \\ 
\hline
$N$ & $\overline{N}$ & $X_{59}(-e^{-\pi\sqrt{N/59}})$ & $\lambda_{59}$  \\
\hline
3 d & above & $\frac{-1}{2\sqrt{3}-1}$ & $\frac14-\frac{5\sqrt{3}}{12}$ \\
17 & d & $\frac{-2}{7+\sqrt{17}}$ & $\frac{153}{52}-\frac{605\sqrt{17}}{884}$\\
\hline
\hline
$N$ & $\overline{N}$ & $X_{60}(e^{-2\pi\sqrt{N/60}})$ & $\lambda_{60}$  \\
\hline
$1$ d && $\frac14$ & undefined \\ 
2 & below & $\frac{1}{2+\sqrt{6}}$ & $\frac75 - \frac{2\sqrt{6}}{5}$ \\
\hline
\end{tabular}}} 
\end{table}

\begin{table}[H]
\caption{Rational and quadratic irrational series for levels~$66$ and~69.}
\label{T:66.1}
{\renewcommand{\arraystretch}{1.4}
{%
}} 
\end{table}

\begin{table}[H]
\caption{Rational and quadratic irrational series for levels~$70$, 71 and~78.}
\label{T:70.1}
{\renewcommand{\arraystretch}{1.4}
{\begin{tabular}{|cc|c|c|}
\hline
$N$ & $\overline{N}$ & $X_{70}(e^{-2\pi\sqrt{N/70}})$ & $\lambda_{70}$  \\
\hline
$1$ d &&  $\frac14$ & undefined \\ 
3 && $\frac15$ & $\frac16$ \\ 
4 & below & $\frac{1}{2+\sqrt{14}}$ & $\frac{15}{22} - \frac{3\sqrt{14}}{22}$ \\ 
9 && $\frac{2}{11+\sqrt{105}}$ & $\frac{51}{32} - \frac{13\sqrt{105}}{96}$ \\ 
13 && $\frac{1}{8+\sqrt{65}}$ & $\frac{51}{32} - \frac{71\sqrt{65}}{416}$ \\ 
31 && $\frac{1}{33+4\sqrt{70}}$ & $\frac{188275}{124806} - \frac{12973\sqrt{70}}{83204}$ \\ 
\hline
$N$ & $\overline{N}$ & $X_{70}(-e^{-\pi\sqrt{N/35}})$ & $\lambda_{70}$  \\
\hline
\end{tabular}}} 
\end{table}

\begin{table}[H]
\caption{Rational and quadratic irrational series for levels~87, 92, 94 and 95. }
\label{T:87.1}
{\renewcommand{\arraystretch}{1.4}
{%
}} 
\end{table}
\hfill

\begin{table}[H]
\caption{Rational and quadratic irrational series for levels~105, 110 and~119. }
\label{T:105.1}
{\renewcommand{\arraystretch}{1.4}
{\begin{tabular}{|cc|c|c|}
\hline
$N$ & $\overline{N}$ & $X_{105}(e^{-2\pi\sqrt{N/105}})$ & $\lambda_{105}$  \\
\hline
1 d & below & $\frac{2}{\sqrt{21}-1}$ & undefined \\ 
2 d&& $\frac12$ & $\frac32$ \\ 
4 &d& $\frac2{1+\sqrt{21}}$ & $\frac52 - \frac{\sqrt{21}}4$ \\ 
8 & below & $\frac1{1+\sqrt{15}}$ & $\frac4{11} + \frac{2\sqrt{15}}{11}$ \\ 
13 & below& $\frac1{2+\sqrt{39}}$ & $\frac{21}{310} + \frac{242\sqrt{39}}{2015}$ \\ 
\hline
$N$ & $\overline{N}$ & $X_{105}(-e^{-\pi\sqrt{N/105}})$ & $\lambda_{105}$  \\
\hline
1 d &above& $\frac{-2}{1+\sqrt{21}}$ & undefined \\
\hline
\end{tabular}}} 
\end{table}

\hfill

\setlength{\tabcolsep}{10pt} 
\renewcommand{\arraystretch}{2.8} 

\begin{table}[H]
   \caption{The coefficients $T(n)$ in terms of binomial coefficients (where known)
   and references to the OEIS~\cite{oeis}.}
   \label{T:00X}
    %
  \end{table}

 \vfill
\pagebreak[4]

\end{document}